\documentclass[10pt, reqno]{amsart}
\usepackage[utf8x]{inputenc}
\usepackage[margin=1.0in]{geometry}
\usepackage{amsfonts}
\usepackage{amsmath}
\usepackage{amssymb}
\usepackage{amsthm}
\usepackage{hyperref}
\usepackage{verbatim}
\usepackage[T1]{fontenc}
\usepackage{palatino}
\usepackage{tikz}
\usepackage{mathrsfs}
\usetikzlibrary{trees}
\usepackage{dsfont}
\usepackage{upgreek}
\usepackage[titletoc,title]{appendix}
\usepackage{pgfplots}
\usepackage{float}
\usepackage{enumerate}

\newcommand{\R}{\mathbb{R}}
\newcommand{\N}{\mathbb{N}}

\newcommand{\pr}{\mathbb{P}}
\newcommand{\ex}{\mathbb{E}}

\newcommand{\h}{\mathcal{H}}
\newcommand{\e}{\mathcal{E}}

\newcommand{\blangle}{\big\langle}
\newcommand{\brangle}{\big\rangle}

\newtheorem{lem}{Lemma}

\newtheorem{prop}{Proposition}
\newtheorem{thm}{Theorem}
\newtheorem{dfn}{Definition}

\newtheorem{innercustomthm}{Assumption}
\newenvironment{customthm}[1]
{\renewcommand\theinnercustomthm{#1}\innercustomthm}
{\endinnercustomthm}

\theoremstyle{definition}
\newtheorem{rem}{Remark}
\newtheorem{example}{Example}
\makeatletter
\newcommand{\leqnomode}{\tagsleft@true\let\veqno\@@leqno}
\newcommand{\reqnomode}{\tagsleft@false\let\veqno\@@eqno}
\makeatother

\DeclareFontEncoding{FMS}{}{}
\DeclareFontSubstitution{FMS}{futm}{m}{n}
\DeclareFontEncoding{FMX}{}{}
\DeclareFontSubstitution{FMX}{futm}{m}{n}
\DeclareSymbolFont{fouriersymbols}{FMS}{futm}{m}{n}
\DeclareSymbolFont{fourierlargesymbols}{FMX}{futm}{m}{n}
\DeclareMathDelimiter{\VERT}{\mathord}{fouriersymbols}{152}{fourierlargesymbols}{147}

\numberwithin{lem}{section}
\numberwithin{cor}{section}
\numberwithin{prop}{section}
\numberwithin{thm}{section}
\numberwithin{dfn}{section}

\title{Uniform attraction and exit problems for stochastic damped wave equations}
\author{Ioannis Gasteratos}
\address[Ioannis Gasteratos]{Technische Universit\"at Berlin, Institute of Mathematics}
\email{i.gasteratos@tu-berlin.de}
 \author{Michael Salins}
 \address[Michael Salins]{Boston University, Department of Mathematics and Statistics}
 \email{msalins@bu.edu}
 \author{Konstantinos Spiliopoulos}
 \address[Konstantinos Spiliopoulos]{Boston University, Department of Mathematics and Statistics}
 \email{kspiliop@bu.edu}
	\thanks{
	 KS was partially supported by the National Science Foundation (DMS 2107856, DMS 2311500) and Simons Foundation Award 672441. IG acknowledges financial support from the EPSRC grant EP/T032146/1.}      
\begin{document}

\begin{abstract} We consider a class of wave equations with constant damping and polynomial nonlinearities that are perturbed by small, multiplicative, space-time white noise. The equations are defined on a one-dimensional bounded interval with Dirichlet boundary conditions, continuous initial position and distributional initial velocity. In the first part of this work, we study the corresponding deterministic dynamics and prove that certain neighborhoods of asymptotically stable equilibria are uniformly attracting in the topology of uniform convergence. Then, we consider exit problems for local solutions of the stochastic damped wave equations from bounded domains $D$ of uniform attraction. Using tools from large deviations along with novel controllability results, we obtain logarithmic asymptotics for exit times and exit places, in the vanishing noise limit, that are expressed in terms of the corresponding quasipotential. In doing so, we develop arguments that take into account the lack of both smoothing and exact controllability that are inherent to the problem at hand. Moreover, our exit time results provide asymptotic lower bounds for the mean explosion time of local solutions.  We introduce a novel notion of "regular" boundary points allowing to avoid the question of boundary smoothness in infinite dimensions and leading to the proof of a large deviations lower bound for the exit place.  We illustrate this notion by providing explicit examples for different classes of domains $D$. Conditions under which lower and upper bounds for exit time and exit place logarithmic asymptotic hold, are also presented.  In addition, we obtain deterministic stability results for linear damped wave equations that are of independent interest.    
 \end{abstract}

	\maketitle


 \section{Introduction}
  We are concerned with the stochastic damped wave equations
	      \begin{equation}\label{eq:model}
	    \left\{\begin{aligned}
	    &\partial^2_t u^\epsilon+\alpha\partial_t u^\epsilon=\partial^2_x u^{\epsilon}  +b\big( u^\epsilon\big)+\epsilon\sigma\big(u^{\epsilon}\big)\dot{W}
	    \\&u^{\epsilon}(0,x)=u(x), \partial_t u^{\epsilon}(0,x)=v(x),\;  x\in(0,\ell),\\& u^{\epsilon}(t,x)=0,\; (t,x)\in[0,\infty)\times\{0,\ell\}\;,\epsilon>0
	    \end{aligned}\right.
	    \end{equation}
     defined on the one-dimensional spatial interval $(0,\ell)$ with Dirichlet boundary conditions. The solution $u^\epsilon$ can be thought as the displacement field of a one-dimensional  string of length $\ell>0$ with fixed ends and unit density.     The string is subject to small, state-dependent random forcing $\sigma\dot{W}$ of intensity $\epsilon\ll 1,$ a nonlinear polynomial force $b$ and a damping force of constant magnitude $\alpha>0$ that is proportional to the velocity field $\partial_t u^\epsilon.$ Throughout this work, $\dot{W}$ is space-time white noise and the initial position/velocity pair $(u,v)\in C_0(0,\ell)\times C^{-1}(0,\ell)$ is given respectively by a continuous function that vanishes on the boundary points and a distribution that can be thought as the generalized derivative of a continuous function (see Section \ref{Sec:Setup} for precise definitions). 

    Random perturbations of dynamical systems such as \eqref{eq:model} have been extensively studied in the literature since the classical work of Freidlin-Wentzell in the 70's \cite{freidlin1998random}. A significant part of their theory is devoted to studying the long-time effects of small noise on the behavior of a finite-dimensional dynamical system. As a prototypical example, they consider the exit problem of a finite-dimensional diffusion process $Z^\epsilon$ from a domain of attraction. Letting $D$ be a bounded, uniformly attracting domain that contains an asymptotically stable equilibrium $z^*$ of the noiseless dynamics, they prove logarithmic asymptotics for the exit times $\tau^\epsilon_D$ of $Z^\epsilon$ from $D$ in the zero-noise limit $\epsilon\to 0.$ Moreover, their analysis provides asymptotic results for the law of exit places $Z^\epsilon(\tau^\epsilon_D)$ that is supported on the boundary $\partial D.$

    Roughly speaking, as $\epsilon\to 0,$ exit times from $D$ are expected to diverge since the noiseless dynamics will stay trapped in $D$ without exiting. When $\epsilon$ is small but non-zero, random trajectories issued from points in $D$ tend to spend a large amount of time near small neighborhoods of $z^*$ with overwhelming probability. Nevertheless, fast "rare" exits to the boundary take place in time scales of order $\exp(1/\epsilon^2).$ Moreover, it can be shown that such exits are facilitated when the random trajectories remain close to certain deterministic paths of minimal "energy". Following this rationale and making use of large deviations estimates and the Markov property, Freidlin-Wentzell \cite{freidlin1998random} show that, when $\epsilon\approx 0,$
       $$ \tau_D^\epsilon\approx\exp\bigg(  \frac{V(z^*,\partial D)}{\epsilon^2}\bigg), $$
    where the asymptotic growth rate is given by the \textit{quasipotential} $V(z^*,\partial D).$ The latter quantifies the minimal energy required by a controlled trajectory issued  from $z^*$ to reach $\partial D$ in finite time. It is defined, see also Definition \ref{dfn:quasipotential}, by 
    \begin{equation}\label{eq:QuasipotentialIntro}
        V(z^*,\partial D):=\inf\bigg\{ I_{z,T}(\phi)\;\bigg|\; T>0, \phi\in C([0,T];\e): \phi(0)=z^*, \phi(T)\in\partial D    \bigg\}
    \end{equation}    
    where $\e$ is the corresponding state space and 
    $I_{z^*,T}:C([0,T];\e)\rightarrow[0,\infty]$ the large deviations \textit{rate function} of $\{Z^\epsilon; \epsilon>0\}:$
    $$  I_{z,T}(\phi):=\inf\bigg\{\frac{1}{2}|u|^2_{L^2([0,T];H)}\bigg| u\in L^2([0,T];H):  \phi=Z_z^u       \bigg\}.$$
    Here, $u, Z_z^u, H$ denote respectively a deterministic control, a controlled trajectory of the underlying deterministic equation (also known as the "skeleton" equation) issued from $z\in D$ and an appropriate Hilbert space depending on the covariance of the driving noise (for $d$-dimensional diffusions $H:=\R^d)$ . As for the distribution of exit places, the classical theory shows that $Z^\epsilon(\tau^\epsilon_D)$ concentrates on minimizers of $V(z^*,\partial D)$ with probability converging to $1$ as $\epsilon\to 0.$ Such results on exit problems for diffusions can be found e.g. in \cite[Chapter 4]{freidlin1998random} (see also \cite[Section 5.7]{dembo2009large} and references therein). Several extensions in finite dimensions have been explored in settings such as characteristic boundaries (i.e. when $\partial D$ contains saddle points) \cite{day1987recent, day1989boundary,day1990large}, delay equations \cite{lipshutz2018exit} and McKean-Vlasov equations \cite{tugaut2012exit}.

   Turning to infinite dimensions, a natural question is whether similar asymptotics continue to hold when $Z^\epsilon$ is the solution of a parabolic Stochastic Partial Differential Equation (SPDE). Essential difficulties arise when it comes to both controllability properties and regularity of the quasipotential. In particular, lower bounds for the exit times $\tau^{\epsilon}_D$ rely on certain continuity properties of $V.$ In order to prove the latter, one needs to show that two states in $\e$ that are arbitrarily close can be joined with a controlled path $Z_z^u$ with arbitrarily small energy. In contrast to finite dimensions, where linear paths are sufficient to obtain such statements, the presence of unbounded operators prevents from using classical arguments. Nevertheless, the authors of \cite{salins2021metastability} were able to overcome these obstacles and proved exit time and place asymptotics for a class of parabolic SPDEs with Allen-Cahn type nonlinearities, in spatial dimension $d=1$, and in the case where $\e=C_0(0,\ell)$ and $D$ is a domain with characteristic boundary (thus extending results from \cite{day1990large} to an infinite-dimensional setting). In \cite{salins2021metastability}, the smoothing properties of the heat kernel along with the superlinear dissipativity (or "coming-down-from-infinity")  of the dynamics are key ingredients that allow for such extensions.  Related results for SPDEs in any spatial dimension $d>1$ with "smooth" additive noise and Lipschitz nonlinearities have been obtained in \cite{cerrai2011approximation}; see also \cite{chenal1997uniform, da1991minimum, debussche2013dynamics}.
   
   In the present work, we are interested in exit problems for hyperbolic SPDEs such as \eqref{eq:model}. In the literature, similar problems have been considered  exclusively in the context of small-mass asymptotics (or Smoluchowski-Kramers approximations) which  amount to studying \eqref{eq:model} with a small constant $\mu$ multiplying the acceleration term $\partial^2_t u^\epsilon$. In particular, \cite{chen2005smoluchowski, freidlin2004some} consider finite-dimensional Lengevin equations and compare exit problem asymptotics for the position component as $\mu,\epsilon\to 0.$ As for the infinite-dimensional case, \cite{cerrai2016smoluchowski} focuses on small-mass, small-noise damped wave equations in $d\geq 1$ spatial dimensions, with additive colored noise and Lipschitz nonlinearities. Letting $\e=L^2\times H^{-1}, G\subset L^2$ open and bounded, they  generalize exit problem results from \cite{chen2005smoluchowski} to infinite dimensions
   by studying quasipotentials and exits from cylindrical domains $D=G\times H^{-1}$ as $\epsilon,\mu\to0.$ Next, we mention a number of works that develop general approaches to Freidlin-Wentzell exit problems for small-noise stochastic evolution equations that take values in Banach spaces: \cite{salins2014general} obtains exit time and place asymptotics for equations with multiplicative noise under strong dissipativity assumptions on the nonlinearity $b$ and in the case of a single equilibrium point. Uniform large deviation principles and their applications to exit problems are investigated in \cite{salins2019uniform} under the assumptions that the associated semigroups are smoothing and that the quasipotential is regular. Generic estimates for exit time and place asymptotics for additive noise SPDEs are included in \cite[Chapter 12.5]{da2014stochastic}. We conclude the literature review by mentioning a few other works concerned with small-noise, long time behaviour for wave equations: \cite{martirosyan2017large} proves large deviations for invariant measures of a $3d$-version of \eqref{eq:model} driven by spatially smooth noise. Finally, mean transition frequencies between metastable states for $1$d, additive-noise, undamped wave equations with double well-potentials are computed in \cite{newhall2017metastability}, under the assumption that initial data is sampled from a stationary distribution. Recently, using potential-theoretic methods, the authors of \cite{barashkov2024eyring} generalized \cite{newhall2017metastability} to the singular regime of spatial dimensions $d=2,3,$ with $b(u)=u-u^3,$ periodic boundary conditions and working under the assumption that $\ell<2\pi$ which guarantees the existence of a single saddle point between the wells. 

   In contrast to the aforementioned works, we aim to study the long-time, small-noise behaviour of the Markov process $Z^\epsilon=(u^\epsilon, \partial_tu^\epsilon)$ consisting of the displacement/velocity components of $\eqref{eq:model}$ and viewed as a random evolution on the phase space $\e=C_0(0,\ell)\times C^{-1}(0,\ell).$ To be more precise, we obtain exit time and exit place (or exit shape) asymptotics for general bounded domains $D\subset\e$ that are uniformly attracted to asymptotically stable equilibria $z^*=(x^*, 0)\in\e.$ 
   The nonlinearity $b$ is allowed to be any degree polynomial such that a certain stability condition for the elliptic problem $\partial^{2}_{x}u+b(u)=0$ with Dirichlet boundary conditions holds, Assumption \ref{Assumption:xstar}, and the noise coefficient $\sigma$ is locally Lipschitz.
   Consequently, the noiseless dynamics $Z^0$ will typically feature multiple stable equilibria and the corresponding quasipotential is not available in closed form. 
   As we shall discuss below, our choice for the phase space $\e$ and the boundedness of $D$ enable us to prove our results for local solutions. In turn, we bypass several restrictive assumptions on $b$ that imply global well-posedness. Nevertheless, for the sake of completeness, we shall also discuss exit problems for global solutions in Section \ref{sec:GlobalSolutions}.

   Our setting presents a number of significant difficulties due to the following reasons:
   
   a) \textit{Lack of superlinear dissipativity:} this poses challenges at the level of the deterministic dynamics. Since $b$ is non-Lipschitz, non-dissipative and we work with supremum-norm topologies, $L^2-$ type energy estimates are not sufficient to prove uniform attraction to stable equilibria. Consequently, we need to investigate existence of uniformly attracting and invariant domains $D$ in the topology of $\e=C_0(0,\ell)\times C^{-1}(0,\ell)$ before studying the exit problem for the stochastic dynamics. Such investigations take place in Section \ref{sec:NonlinearStability} and our main result on uniformly attracting domains for the noiseless dynamics is Theorem \ref{thm:domainofattraction}. Moreover, stability results for the (one-dimensional) linear damped wave equation in the topology of $\e$ are not available in the literature (see Section \ref{Sec:Linfty-decay} below). For this reason, we provide here such a result (Theorem \ref{thm:Linftydecay}) that is necessary for our subsequent analysis but is also of independent interest.
   
    b)\textit{ Lack of smoothing properties:} Since wave equations do not regularize "rough" initial data and the noise only appears in the velocity component of  $Z^\epsilon,$
    the controllability arguments from the parabolic case in \cite{salins2021metastability} are no longer applicable here. Moreover, exact controllability in $H^1\times L^2$ is known to fail in the setting of wave equations with superlinear nonlinearities (see e.g. \cite[Theorem 2]{Zuazua1993}; other exact controllability results for wave equations have appeared e.g. in \cite{FuYongZhang_ExactControlability2007,lasiecka1991exact} and are also not applicable to the setting we are working on, see Section \ref{sec:Controllability}). Therefore, proving regularity properties for the quasipotential and exit time lower bounds becomes more involved. To be precise, we are only able to show that any point in $z\in H^1(0,\ell)\times L^2(0,\ell)$ can be connected to $z^*$ by paths of small energy, provided that $z^*, z$ are close. Nevertheless, this controllability property is sufficient to prove that an associated quasipotential is inner regular, a property that is stated as an assumption in \cite{salins2019uniform}. These results are contained in Section \ref{sec:Controllability}.
    
    In this work, we suggest that the natural topology to investigate metastability for wave equations with superlinear forcing is $\e=C_0(0,\ell)\times C^{-1}(0,\ell).$ The Hilbert space topology $L^2 \times H^{-1},$ in which the wave semigroup is a contraction, is used throughout the previously mentioned literature on exit problems for wave equations. This topology is particularly well-suited for the important but limited setting of nonlinearities $b$ that are globally Lipschitz continuous. In contrast, and in alignment with applications in the physical literature (see e.g. \cite{campbell1983resonance} and references therein), we allow $b$ to be a polynomial of arbitrary degree and hence the Hilbert space $L^2 \times H^{-1}$ is unsuitable for our purposes. Briefly, if $u$ is a square-integrable function defined on $[0,\ell]$ and $b(u) = -u^r, r\geq 3$ is a nonlinearity of Klein-Gordon type, then the functional composition $b(u(\cdot))$ is no longer guaranteed to belong to $L^2$.
    
    Apart from resolving such issues with function composition, the supremum norm topology allows for a definition of \textit{local mild solutions} that exist until the first time that the position reaches an infinite value (see Definition \ref{dfn:local mild solutions}, below). Such localization methods enable us to study Freidlin-Wentzell exit problems from bounded domains $D \subset \mathcal{E}$. We believe that similar localization methods are not possible in $L^p$ topologies. Moreover, while it remains an open problem whether global solutions to the one-dimensional stochastic wave equation exist when $r>3$ \cite{millet2001nonlinear}, our exit time results can be used to prove quantitative lower bounds on the mean explosion time of local solutions (Remark \ref{rem:explosiontimes} below). Technical aspects aside, the topology of $\e$ allows for a more natural interpretation of the exit problem: exiting neighborhoods in the uniform topology amounts to exiting from "tubes" around the stable equilibria $z^*.$ Finally, by working in $\e,$ we obtain intuitive  descriptions of geometric notions such as "inward" and "outward" pointing vectors (see Section \ref{Sec:BoundaryPoints} and in particular Examples \ref{ex:regularculinders}- \ref{ex:regularspheres}). For these reasons we argue that $C_0(0,\ell)\times C^{-1}(0,\ell)$ is the most useful topology for investigating Freidlin-Wentzell exit problems, and we work in this topology for the rest of the paper.

    To the best of our knowledge, the Freidlin-Wentzell exit problem for (local) solutions of damped wave equations with multiplicative noise and multiple stable equilibria is considered here for the first time. In summary, our contribution is fourfold:
    \begin{itemize}
\item[a)] We prove new linear and non-linear stability results for the \textit{deterministic damped wave equation} in topologies of uniform convergence. Recent developments in the linear, deterministic damped wave equation literature \cite{chitour2024exponential,chitour2024p,kafnemer2022lp} have focused on the long time behavior of the damped wave semigroup in $L^p$ topologies, $p\in (1,\infty)$ and general stability estimates in the case $p=\infty$ have been conjectured to hold in \cite[pp.4]{chitour2024exponential}. We believe that our Theorem \ref{thm:Linftydecay} is the first to address the topology of $\e$ and, thus, is of independent interest.

\item[b)] Our mean exit time asymptotic upper and lower bounds from Theorem \ref{thm:ExitTimeUpperBnd}(1), \ref{thm:ExitTimeLowerBnd}(2)
    read: for all $z\in D:$
    \begin{equation}\label{eq:meanexittimeIntro}
         V(z^*, \partial D)\leq \liminf_{\epsilon\to 0}\epsilon^2\log\ex\big[ \tau_D^{\epsilon,z}\big]\leq \limsup_{\epsilon\to 0}\epsilon^2\log\ex\big[ \tau_D^{\epsilon,z}\big]\leq V(z^*, \bar{D}^c),    
    \end{equation}
    where the quasipotential on the right-hand side is defined analogously to \eqref{eq:QuasipotentialIntro} and
    $\tau_D^{\epsilon,z}$ is the first exit time from $D$ of the trajectories $Z^{\epsilon}_z=(u_z^\epsilon, \partial_t u_z^\epsilon)$ starting from $Z^{\epsilon}_z(0)=z\in\e.$ We investigate the equality of the upper and lower bounds by  introducing a notion of "regular" boundary points (Definition \ref{def:regularpoints}) and identifying conditions on $D$ that guarantee   $V(z^*, \bar{D}^c)=V(z^*, \partial D)$ (Section \ref{Sec:BoundaryPoints}). This equality is automatic in finite dimensions as a consequence of smoothness of the boundary and is often stated as an assumption in relevant literature on infinite dimensions. 
\item[c)] Apart from the typical behavior of asymptotic exit shapes (Theorem \ref{thm:exitshapeldp}(1)), we also obtain large deviations bounds for exit shapes (albeit with non-matching upper and lower bounds); see Theorem \ref{thm:exitshapeldp}(2),(3). The exit shape large deviations bounds are in terms of two other related quasipotentials $V_D$ and $V_{\bar{D}}$  that play a significant role and they are defined similarly to $V$ but with the additional requirement that the trajectories stay in $D$ and $\bar{D}$ respectively for  all times (see Section \ref{sec:Controllability}). In the context of wave equations, such estimates cannot be found elsewhere in the literature.
\item[d)] Our exit time and exit shape asymptotic results apply to local solutions of \eqref{eq:model} and do not rely on stationarity of solutions. 
Moreover, they provide asymptotic lower bounds for the mean explosion time of local solutions per Remark \ref{rem:explosiontimes}.
    \end{itemize}
    
    The rest of this article is organized as follows: In Section \ref{Sec:Setup} we introduce all the necessary notation, function spaces and assumptions on \eqref{eq:model}. In Section \ref{Sec:Linfty-decay} we present our first main result, Theorem \ref{thm:Linftydecay}, on linear stability of damped wave equations. Section \ref{sec:NonlinearStability} is devoted to the study of the nonlinear, noiseless dynamics. The main result here is Theorem \ref{thm:domainofattraction} in which we prove uniform attraction for the nonlinear deterministic equation in $\e.$ In Section \ref{sec:wellposedness} we first show existence and uniqueness of local, $\e-$valued mild solutions
 to a controlled version of \eqref{eq:model} under the assumption that $b$ is a polynomial of arbitrary degree (Assumption \ref{Assumption:b}). Then we prove estimates and continuity properties for the "skeleton" equation as well as a local Large Deviation Principle for $\{Z_z^\epsilon;\epsilon>0\}$ that is uniform over bounded sets of initial data (ULDP). In Section \ref{sec:Controllability} we introduce  quasipotentials $V_D$  and $V_{\bar{D}}$, \eqref{eq:VDquasipotential} and \eqref{eq:VDbarquasipotential} respectively, and prove controllability and inner regularity results for $V_D$ Lemma \ref{lem:innerregularity}. Section \ref{Sec:Metastability} contains all our analysis and results for exit problems. The main results here, given by Theorems \ref{thm:ExitTimeUpperBnd}, \ref{thm:ExitTimeLowerBnd}, \ref{thm:exitshapeldp}, treat exit time upper bounds, exit time lower bounds and exit place/shape asymptotics respectively. As mentioned above, this section also contains a discussion on the matching of the upper and lower bounds in \eqref{eq:meanexittimeIntro}. Moreover, we provide here several examples of regular and irregular boundary points of cylindrical and spherical domains. In Section \ref{sec:GlobalSolutions} we consider the case of global solutions which requires the additional assumption that $b$ has a contracting nonlinearity of at-most-cubic growth (e.g. $b(u)=u-u^3$). Global well-posedness results exist in the literature but for spaces  $\h=H\times H^{-1}$, see \cite{cerrai2003stochastic,cerrai2006smoluchowski}. 
 In Section \ref{sec:GlobalSolutions}, we prove global well-posedness on the smaller space $\e$, which we believe is of independent interest.
     Proofs of several auxiliary results are collected in Appendix \ref{section:Appendix}.

 \section{Setup and preliminaries}\label{Sec:Setup}
 \subsection{Notation, function spaces and assumptions}\label{Sec:Notation} The lattice notation $\wedge, \vee$ is used to denote minimum and maximum of real numbers respectively. We use $\mathds{1}_E$ to denote the indicator of a set $E.$ We write $\textnormal{sgn}$ to denote the sign function $\textnormal{sgn}(x)=1$ if $x>0,$ $\textnormal{sgn}(x)=-1$ if $x<0$ and $\textnormal{sgn}(0)=0.$ The identity operator on a vector space will be denoted by $I,$ provided that there is no need for further indication of its domain and co-domain.
 
    We denote by $E^\star$ the continuous dual of a Banach space $E.$  Given another Banach space $E',$ $\mathscr{L}(E;E')$ is the space of bounded, linear operators from $E$ to $E'.$ The open ball of radius $R>0,$ centered at $x\in E,$ the closure, interior, boundary and complement of a set $D\subset E$ are denoted respectively by $B_E(x,R), \bar{D}, \textnormal{int}(D),$ $\partial D,$ $ D^c.$ We write $\textnormal{dist}_E(x, D):=\inf_{y\in D}|  x-y |_E$  for the distance between $x$ and $D.$ For another set $D'\subset E,$  $\textnormal{dist}_E(D, D'):=\inf_{x\in D', y\in D}|x-y|_E. $ Moreover, we use the notation $\bar{D}^c$ to denote the set $E\setminus\bar{D}.$

     Throughout this work and unless otherwise stated  the brackets $\langle\cdot,\cdot\rangle$ with no subscript will denote the duality pairing between dual topological vector spaces.
      The subdifferential of a function $f:E\rightarrow\R$ at the point $x_0\in E$ is defined to be the set
      \begin{equation*}\label{eq:subdifferentialdfn}
          \begin{aligned}
              \partial f(x_0):=\big\{ x^\star\in E^\star: f(x)\geq f(x_0)+\langle x^\star,x-x_0\rangle \big\}.
          \end{aligned}
      \end{equation*}
      If $f$ is Lipschitz continuous, we denote its Lipschitz constant by $|f|_{Lip}$ .

For $\ell>0$ we set $H:=L^2(0,\ell)$ and denote respectively by $\{e_k\}_{k\in\N}\subset H,$ $ \{a_{k}\}_{k\in\N}\subset[0,\infty)$ the orthonormal basis consisting of the Dirichlet eigenfunctions of the second derivative operator $-\partial_x^2$ in $H$ and the corresponding eigenvalues. In particular, we have for each $x\in(0,\ell), k\in\N$ 
\begin{equation*}\label{eq:Laplacianeigenpairs}
    \begin{aligned}
        e_k(x)=\sqrt{\frac{2}{\ell}}\sin\bigg(\frac{k\pi x}{\ell}   \bigg)\;,\;\; a_k=\frac{k^2\pi^2}{\ell^2}.
    \end{aligned}
\end{equation*}

For $\delta\in\R,$ $H^\delta$ is the completion of $C^\infty(0,\ell)$ with respect to the norm
$$|f|^2_{H^\delta}:=\sum_{n=1}^{\infty}a^\delta_k\langle f, e_k\rangle^2_H=\sum_{n=1}^{\infty}a^\delta_kf_k.$$ With this notation we have $H^0=H=L^2((0,\ell)).$ Moreover, we shall refer to $f_k$ as the $k-$th Fourier coefficient of $f.$
We define
	      $$\h_\delta:=H^{\delta}(0,\ell)\oplus H^{\delta-1}(0,\ell),$$
       where $\oplus$ denotes the Hilbert space direct sum and we shall frequently use the notation $$\h:=\mathcal{H}_0=L^2(0,\ell)\oplus H^{-1}(0,\ell).$$ The  space of real-valued continuous functions $f,$ defined on $[0,\ell]$ and such that $f(0)=f(\ell)=0$ is denoted by    $C_0(0,\ell).$ The latter is  a Banach space with the topology of uniform convergence induced by the supremum norm $|\cdot|_\infty$.  $\mathcal{M}(0,\ell):=C(0,\ell)^\star$ denotes the continuous dual of $C(0,\ell)$ i.e.  the space of finite, Borel signed measures on $(0,\ell).$

       The space of distributions on $(0,\ell)$ that can be expressed as weak derivatives of continuous functions is denoted by $C^{-1}(0,\ell).$ In particular,
        \begin{equation}\label{eq:C-1definition}
            C^{-1}(0,\ell):=\bigg\{ f\in (C^{\infty}_c(0,\ell))^\star\bigg|\; \exists \hat{f}\in C(0,\ell),\; \forall \phi\in\; C^{\infty}_c(0,\ell)\;:  \langle f,\phi\rangle=-\int_{0}^{\ell}\hat{f}(y)\phi'(y)dy
    	       \bigg\},
        \end{equation}
     
       where $C_c^\infty(0,\ell)$ is the set of smooth, compactly supported test functions.
            $C^{-1}(0,\ell)$ is a Banach space endowed with the norm 
\begin{equation*}\label{eq:Cminusnorm}
            \begin{aligned}  |f|_{C^{-1}}:=|\hat{f}|_{\infty}\equiv\sup_{x\in(0,\ell)}\bigg|\int_0^x f(y)dy\bigg| 
            \end{aligned}
        \end{equation*}
        where we remark that the anti-derivative notation is used formally for both functions and distributions. Furthermore, we emphasize that we shall frequently use the notation $\hat{f}$ for the antiderivative of $f$ vanishing at $0.$

        In the sequel we fix
    	   \begin{equation}\label{eq:Espace}
     \e:=C_0(0,\ell)\oplus C^{-1}(0,\ell)
    	   \end{equation}
      and note that $\e\subset\h$ with continuous inclusion. The Hilbert spaces $\h_\delta$ come with natural projection operators $\Pi_1:\h_\delta \rightarrow H^{\delta}(0,\ell), \Pi_2: \h\rightarrow H^{\delta-1}(0,\ell)$ defined, for $i=1,2,$ by $ \Pi_i(u_1,u_2)=u_i.$ Unless otherwise stated, the Banach space direct sums  $\e,\h$ are respectively endowed with the norms $|\cdot|_{\e}:=|\cdot|_\infty+|\cdot|_{C^{-1}} ,$ $|\cdot|_{\h}:=|\cdot|_H+|\cdot|_{H^{-1}}.$

      \begin{rem} The space $C^{-1}(0,\ell)$ is a closed subspace of the Sobolev space $W_0^{-1,\infty}(0,\ell).$ Indeed the latter can be represented as the space of distributions $f$ for which there exists $\hat{f}\in L^{\infty}(0,\ell)$ such that $f=-\partial \hat{f}$ (see e.g. \cite[Theorem 3.12]{adams2003sobolev}). Note that, in contrast to $W_0^{-1,\infty},$  $C^{-1}(0,\ell)$ is separable.
      \end{rem}

      We turn to our assumptions for the coefficients of the wave equation \eqref{eq:model}.
      \begin{customthm}{1}\label{Assumption:b}   
    The function $b:\R\rightarrow \R$ is a polynomial of degree $\gamma\in\N.$
       \end{customthm}

       Assumption \ref{Assumption:b} guarantees local well-posedness for the stochastic wave equation \eqref{eq:model} on $\e$ (see Proposition \ref{prop:WellPosedness} below) and we impose it throughout Sections \ref{sec:NonlinearStability}-\ref{Sec:Metastability}.

       Regarding the noise coefficient $\sigma$ we have: 
          \begin{customthm}{2}\label{Assumption:sigma} The function $\sigma:\R\rightarrow\R$ is locally Lipschitz continuous.
   \end{customthm}

   In Section \ref{sec:GlobalSolutions}, we show that all the results of this work are valid for global solutions under a more restrictive set of assumptions on $b$ and $\sigma$.

   The random forcing $\dot{W}$ is a space-time white noise on a complete filtered probability space $(\Omega, \mathcal{F}, \{  \mathcal{F}_t  \}_{t\geq 0}, \pr)$, i.e. the distributional time-derivative of a cylindrical Wiener process $W$ i.e.
   a Gaussian process with covariance given for $t,s\geq 0, f,g\in H$ by
   \begin{equation*}
       \begin{aligned}
           \ex[\langle W(t), g\rangle \langle W(s), g\rangle       ]=t\wedge s\langle f,g\rangle_H.
       \end{aligned}
   \end{equation*}

Next, we introduce an assumption that concerns the stability properties of the deterministic problem. 
  \begin{customthm}{3}\label{Assumption:xstar} There exists at least one solution $x^*\in C_0(0,\ell)$ of the elliptic problem $\partial^2_x u  +b\big( u\big)=0$ with Dirichlet boundary conditions. Moreover, the linear, self-adjoint operator $-\partial^2_x-b'(x^*),$ defined on $H,$ has a countable, non-decreasing sequence of eigenvalues $\{a_k^b\}_{k\in\N}$ corresponding to a complete orthonormal set of eigenvectors $\{e_n^b\}_{n\in\N}\subset C_0(0,\ell).$ Finally, we assume that  $\lambda_b:=a_2^b-a_1^b>0.$
   \end{customthm}

   Assumption \ref{Assumption:xstar} is always assumed to hold from Section \ref{sec:NonlinearStability} onwards. For a discussion of examples that are covered by this assumption we refer to Example \ref{rem:Xstarexamples} below. Finally, we shall impose a non-degeneracy condition on $\sigma$ from Section \ref{sec:Controllability} onwards.

   \begin{customthm}{4}\label{Assumption:sigmaNondeg} $\sigma$ is non-degenerate, i.e. there exists $C_\sigma>0$ such that for all $x\in\R:$ $\sigma(x)>C_\sigma.$
   \end{customthm}

     Two additional assumptions, Assumption \ref{Assumption:Dinterior}, \ref{Assumption:MinimizingRegularPoints}, on domains of attraction related to the exit problem will be introduced in Section \ref{Sec:Metastability}.

 \subsection{Semigroups and linear stability} Letting $v^\epsilon:=\partial_tu^\epsilon$ denote the velocity of the solution to  \eqref{eq:model}, we re-write the latter as a first-order system 
 \begin{equation}\label{eq:modelsystem}
     \begin{aligned}
        \partial_t\begin{pmatrix}
		u^\epsilon \\v^\epsilon
	\end{pmatrix}=A_\alpha
		\begin{pmatrix}u^\epsilon \\v^\epsilon
	\end{pmatrix}+\begin{pmatrix}0\\b(u^\epsilon) \end{pmatrix}+\begin{pmatrix}0\\ \epsilon\sigma\big(u^{\epsilon}\big)\end{pmatrix} \dot{W}\;,\;\;\begin{pmatrix} u^\epsilon(0) \\v^\epsilon(0)\end{pmatrix}=\begin{pmatrix} u \\v
	\end{pmatrix},
     \end{aligned}
 \end{equation}
 where \begin{equation}\label{eq:dampedwavegenerator}
 	A_\alpha=\begin{pmatrix}
 		0&I\\
 		\partial_x^2& -\alpha I
 	\end{pmatrix},\;\;\alpha\geq 0
 \end{equation}
is the damped wave operator. For any $\delta\in\R,\alpha>0,$ $A_\alpha$ generates a strongly continuous group $\{S_\alpha(t)\}_{t\in\R}\subset\mathscr{L}(\h_\delta)$ of bounded linear operators; a proof of this can be found e.g. in \cite[Section 7.4]{pazy1983semigroups}, \cite[Proposition 2.4]{cerrai2006PTRFsmoluchowski1}. For $u,v\in\h_\delta$ and
\begin{equation*}
    \begin{aligned}
        \Pi_1S_\alpha(t)(u,v)=\sum_{k=1}^{\infty}f_ke_k\;,\;\;\Pi_2S_\alpha(t)(u,v)=\sum_{k=1}^{\infty}g_ke_k,
    \end{aligned}
\end{equation*}
it is possible to express the Fourier coefficients $f_k, g_k$ explicitly in terms of those of $u,v$ via:
 \begin{equation*}
    f_k=\frac{1}{2}\exp\bigg(-\frac{\alpha t}{2}\bigg)\bigg\{\bigg[\bigg(1+\frac{\alpha}{2\gamma^\alpha_k}\bigg)\exp(\gamma_k^\alpha t)+ \bigg(1-\frac{\alpha}{2\gamma^\alpha_k}\bigg)\exp(-\gamma_k^\alpha t)    \bigg]u_k+\frac{1}{\gamma_k^\alpha}\bigg[ \exp(\gamma_k^\alpha t)-\exp(-\gamma_k^\alpha t)      \bigg]v_k\bigg\},
 \end{equation*}
 and 
 \begin{equation*}
    g_k=\frac{1}{2}\exp\bigg(-\frac{\alpha t}{2}\bigg)\bigg\{-\frac{\alpha a_k}{\gamma_k}\bigg[ \exp(\gamma_k^\alpha t)-\exp(-\gamma_k^\alpha t)      \bigg]u_k+ \bigg[\bigg(1-\frac{\alpha}{2\gamma^\alpha_k}\bigg)\exp(\gamma_k^\alpha t)+ \bigg(1+\frac{\alpha}{2\gamma^\alpha_k}\bigg)\exp(-\gamma_k^\alpha t)    \bigg]    v_k           \bigg\},
 \end{equation*}
 where $\gamma_k^\alpha=\sqrt{\alpha^2-4a_k}/2,$  $a_k$ are given in Section \ref{eq:Laplacianeigenpairs} and, if $\gamma^\alpha_k=0,$ we use the convention $$\frac{1}{\gamma^\alpha_k}(e^{\gamma_k t}+e^{-\gamma_k t})=2t.  $$  A similar formula for the Fourier coefficients in the case of wave equations with small mass and unit friction can be found e.g. in \cite[Proposition 2.2]{cerrai2006PTRFsmoluchowski1}.

 Turning to the action of the semigroup on $\e,$ we shall start with the case $\alpha=0.$ To this end, we let $u^o\in C(\R)$ denote the odd, $2\ell-$periodic extension of a function $u\in C(0,\ell)$ by \begin{equation*}
	u^o(x) := \begin{cases}
		u(x) & \text{ if } x \in [0,\ell],\\
		-u(-x) & \text{ if } x \in [-\ell,0],\\
		u^o(x + 2\ell k ) := u^o(x) & \text { for all  }  x \in [-\ell,\ell], k \in \mathbb{Z}.
	\end{cases}
\end{equation*}
Similarly one can define an even $2\ell-$periodic extension $u^e.$

 For each $(u,v)\in\e, t\geq 0, x\in(0,\ell)$ we have, by virtue of d'Alembert's formula,
 \begin{equation}\label{equation:D'Alembert1}
\begin{aligned}
    \Pi_1S_0(t)(u,v)(x)=\frac{1}{2}\bigg(u^o(x-t)+u^o(x+t)\bigg)+\frac{1}{2}\int^{x+t}_{x-t} v^o(y)dy.
\end{aligned}
 \end{equation}
 When the initial velocity is a distribution, the second term on the right-hand side should be interpreted as the antiderivative of the distribution $v^o\in C^{-1}(\R).$ In turn, $v^o$ is the odd, $2\ell-$periodic extension of $v\in C^{-1}_0(0,\ell)$ i.e. the unique distribution $v$ such that
 \begin{equation*}\label{eq:oddextension}
     \begin{aligned}
	v^o := \begin{cases}
		v & \text{ on } [0,\ell],\\
		-\check{v} & \text{ on } [-\ell,0],\\
		T_{2\ell k} v^o := v^o & \text { on} [-\ell, \ell] , k \in \mathbb{Z},
	\end{cases}
     \end{aligned}
 \end{equation*}
 where, for any test function $\phi$ and $r\in\R,$ $\check{\phi}(x):=\phi(-x), T_r\phi(x)=\phi(x+r)$ and analogously $\langle \check{v},\phi\rangle:=\langle v,\check{\phi}\rangle, \langle T_rv, \phi\rangle=\langle v, T_{-r}\phi\rangle.$ Existence and uniqueness of such a distribution can be found e.g. in \cite[Theorem 24.1]{treves2016topological}.
For $t>0, \phi\in C_c^\infty(0,\ell)$ the component of $S_0(t)$ in $C^{-1}(0,\ell)$ is given by \begin{equation}\label{equation:D'Alembert2}
\begin{aligned}
    \blangle \Pi_2S_0(t)(u,v), \phi\brangle=\frac{1}{2}\blangle u^o(\cdot-t)-u^o(\cdot+t), \phi'\brangle
-\frac{1}{2}\blangle  V^o(\cdot+t)+V^o(\cdot-t), \phi'\brangle, 
\end{aligned}
 \end{equation}
 where $V^o\in C(\R)$ is the anti-derivative of $v^0$ vanishing at $0$ as in \eqref{eq:C-1definition}.
 From the last two displays it is straightforward to deduce that $\{S_0(t)\}_{t\geq 0}\subset\mathscr{L}(\e)$ is a $C_0-$semigroup.
 
 \begin{rem} \begin{enumerate}
     \item The function $z(t)=\Pi_1S_0(t)(u,v)$ is the unique solution of the linear wave equation $$(\partial^2_t- \partial^2_x)z=0, (z(0),\partial_tz(0))=(u,v)\in\e.$$  

 \item The formula for $\Pi_2S_0(t)(u,v):=\partial_t \Pi_1S_0(t)(u,v)$ follows by integration-by-parts since, for any test function $\phi\in C^\infty_{c}(\R),$ it holds that
\begin{equation*}\label{testv}
\begin{aligned}
\langle   \Pi_1S_0(t)(u,v)(t,\cdot), \phi\rangle&=\frac{1}{2}\langle u^o, \phi(\cdot-t)+\phi(\cdot+t)\rangle
+\frac{1}{2}\langle V^o, \phi(\cdot-t)-\phi(\cdot+t)\rangle.
\end{aligned}
\end{equation*}
\noindent Differentiation under the sign of the integral and another integration by parts then yields \eqref{equation:D'Alembert2}.
\end{enumerate}   
 \end{rem}

Due to the presence of friction, the semigroup  $\{S_\alpha(t)\}_{t\geq 0}\subset\mathscr{L}(\h)$ is of negative type. The following stability estimate provides an explicit upper bound for the exponential decay rate in the topology of the Hilbert space $\h$:
\begin{prop}{\cite[Ch.IV, Proposition 1.2]{temam2012infinite}}\label{prop:L2-decay} Let $\alpha>0,$ and $ a_1$ as in \eqref{eq:Laplacianeigenpairs}.  For any $\theta<\tfrac{\alpha}{8}\wedge\tfrac{a_1}{\alpha},$ there exists $M>0$ such that for all $t>0$
$$ |S_\alpha(t)(u,v)|_{\h}\leq Me^{-\theta t}|(u,v)|_{\h}.      $$
\end{prop}

Section \ref{Sec:Linfty-decay}  and in particular Theorem \ref{thm:Linftydecay} shows that an interpolation of this estimate yields stronger stability estimates for the solution of the linear damped wave equation in the topology of $\e$ defined in \eqref{eq:Espace}.

\section{Exponential decay of the damped wave semigroup in the supremum norm} \label{Sec:Linfty-decay}

In this section we prove Theorem \ref{thm:Linftydecay}, which shows that the damped wave semigroup $S_\alpha(t)$ decays exponentially in the $\e$ norm for any size of friction $\alpha>0$. Beyond the Hilbert space setting of \cite{temam2012infinite}, stability estimates for the one-dimensional damped wave semigroup with initial data in the Banach spaces $W_0^{1,p}(0,\ell)\times L^p(0,\ell), p\in(1,\infty)$ have been recently proved in \cite[Theorem 4.3]{kafnemer2022lp} with Fourier multiplier techniques. Regarding the case $p=\infty,$ the authors obtain a partial exponential stability result (Proposition 5.1 in the same reference) in the case where the damping constant $\alpha$ is small enough. Moreover, they conjecture that such stability estimates are true without the imposed restrictions on $\alpha$ (see also \cite[pp.4]{chitour2024exponential}, \cite[Section 3.3]{chitour2024p} for more recent references to the open case of $p=\infty$). To the best of our knowledge, Theorem \ref{thm:Linftydecay} is the first result to provide such estimates with initial data in $C(0,\ell)\times C^{-1}(0,\ell)$ (i.e. in Banach spaces that impose no weak differentiability). In fact, from a closer look at the proof of Theorem \ref{thm:Linftydecay}, it is straightforward to obtain the same result with initial data in $W_0^{k,\infty}(0,\ell)\times W^{k-1,\infty}(0,\ell)$ for any $k\in\N.$ In particular, the case $k=1$ generalizes \cite[Proposition 5.1]{kafnemer2022lp} to arbitrary values of $\alpha,$ thus confirming the aforementioned conjecture in the case of constant damping.

\begin{thm}\label{thm:Linftydecay} 
	Let $(u,v)\in\e$. Then for any $ \theta< \tfrac{\alpha}{8}\wedge \tfrac{a_1}{4\alpha}$, there exists $M=M(\theta)>0$ such that for any $t>0$,
	\begin{equation*}
		\left|S_{\alpha}(t)(u,v) \right|_{\e} \leq M e^{-\theta t} |(u,v)|_{\e}.
	\end{equation*}
\end{thm}

In order to prove Theorem \ref{thm:Linftydecay} we take advantage of the following transformation: if $u$ solves the damped linear wave equation $(\partial_t^2+\alpha\partial_t-\partial_x^2)u=0$ with initial data in $(u,v)\in\e$ then $q := e^{\frac{\alpha t}{2}}u$  solves
\begin{equation*}
	\partial_t^2q = \partial_x^2q + \frac{\alpha^2}{4}q,
\end{equation*}
or in mild form
\begin{equation} \label{eq:q-mild}
	q(t,x)=  \Pi_1 S_0(t) \begin{pmatrix} q(0) \\ \partial_tq(0) \end{pmatrix} + \frac{\alpha^2}{8} \int_0^t \int_{x-(t-s)}^{x+(t-s)} q^o(s,y)dyds,
\end{equation}
with $(q(0), \partial_t q(0))=(u, \tfrac{\alpha}{2}u+v)$ and $q^o$ is the odd-periodic extension of $q$.

Moreover, for any odd $2\ell$ periodic $v: \mathbb{R} \to \mathbb{R}$, define \begin{equation}\label{eq:Iv_definition}
    I_v(t,x): = \frac{1}{2}\int_{x-t}^{x+t} v(y)dy.
\end{equation}
Before we proceed to the proof of Theorem \ref{thm:Linftydecay}, we state two auxiliary lemmas that establish time-periodicity and uniform spatio-temporal bounds for $I_v.$ 

\begin{lem} \label{lem:I-periodic}
	For any odd $2\ell$ periodic $v: \mathbb{R} \to \mathbb{R}$ and any $x \in [0,\ell]$, $t \mapsto I_v(t,x)$ is $2\ell$ periodic.
\end{lem}
\begin{proof}
	Because $v$ is odd and periodic, for any $k \in \mathbb{Z}$, $x \in [0,\ell]$
	\begin{equation*}
		\int_{x-\ell k}^{x + \ell k} v(y)dy = 0.
	\end{equation*}
	Furthermore, by odd periodicity we have, for any $t \in [0,\ell],$
	\begin{equation*}
		\int_{x-2\ell k -t}^{x + 2\ell k + t} v(y)dy = \int_{x - t}^{x + t }v(y)dy.
	\end{equation*}
	and
	\begin{equation*}
		\int_{x-(2k + 1)\ell  -t}^{x + (2k+1)\ell  + t} v(y)dy = -\int_{x - t}^{x + t }v(y)dy.
	\end{equation*}
\end{proof}

\begin{lem} \label{lem:integral-bound}
	Assume that $v \in H$. Then there exists $C>0$ such that
	\begin{equation*}
		\sup_{t>0} \sup_{x \in [0,\ell]  } |I_v(t,x)| \leq C|v|_{H}.
	\end{equation*}
\end{lem}
\begin{proof}
	By odd periodicity in time (Lemma \ref{lem:I-periodic}),
	\begin{equation*}
		\sup_{t>0} \sup_{x \in [0,\ell]}|I_v(t,x)| = \sup_{t \in [0,\ell]} \sup_{x \in [0,\ell]} |I_v(t,x)|.
	\end{equation*}
	By H\"older's inequality, for any $t \in [0,\ell]$ and $x \in [0,\ell]$,
	\begin{align}
		|I_v(t,x)| &= \left|\frac{1}{2}\int_{x - t}^{x + t} v(y)dy \right|\nonumber \leq \frac{\sqrt{ 2t}}{2} |v|_{L^2[x-t,x+t]} \nonumber\leq \frac{\sqrt{2\ell}}{2} |v|_{L^2[-\ell,\ell]} \nonumber \leq \sqrt{\ell} |v|_{H} .\end{align}
\end{proof}

\noindent We are now ready to estimate the exponential decay rate of the damped wave semigroup on $\e.$
\begin{proof}[Proof of Theorem \ref{thm:Linftydecay}]
	Let $q = e^{\frac{\alpha t}{2}}u$ so that, in view of \eqref{eq:q-mild}, \eqref{eq:Iv_definition}, $q$ satisfies
	\begin{align}
		q(t,x) = &\frac{1}{2}\left(q^o(0,x - t) + q^o(0,x+t)\right) + \frac{1}{2}I_{(\partial_tq)^o(0)}(t,x)\nonumber+ \frac{\alpha^2}{4}\int_0^t I_{q^o(s)}(t-s,x)ds,\;\;(x,t)\in [0,\ell]\times(0,\infty).
	\end{align}
	By definition of $C$ and $C^{-1}$ norms the first two terms satisfy
	\begin{equation*}
	\left|	\frac{1}{2}\left(q^o(0,x - t) + q^o(0,x+t)\right) \right|\leq |q(0)|_\infty,
	\end{equation*}
	\begin{equation*}
	\left|	\frac{1}{2}I_{(\partial_tq)^o(0)}(t,x)\right| \leq |\partial_tq(0)|_{C^{-1}}.
	\end{equation*}
	By Lemma \ref{lem:integral-bound},
	\begin{equation*}
		\left| \int_0^t I_{q^o(s)}(t-s,x)ds\right| \leq \sqrt{2 \ell}\int_0^t |q(s)|_{H}ds.
	\end{equation*}
	Using the definition of $q$ and Proposition \ref{prop:L2-decay} it follows that 
	\begin{align}
	    \left| \int_0^t I_{q^o(s)}(t-s,x)ds\right| &\leq \sqrt{2 \ell}\int_0^t |u(s)|_{H}e^{\frac{\alpha s}{2}}ds\nonumber\\&
       \leq \sqrt{2 \ell}\int_0^t M|(u(0), \partial_tu(0))|_{\h}e^{\left(\frac{\alpha }{2} - \theta \right)s}ds.\nonumber
	\end{align}	
	These estimates prove that
	\begin{equation*}
	\sup_{x \in [0,\ell]} |q(t,x)| \leq |Q(0)|_\mathcal{E} + \frac{M \alpha^2 \sqrt{2 \ell}}{4 \left(\frac{\alpha}{2} - \theta \right)} |(u(0), \partial_tu(0))|_{\h} e^{\left(\frac{\alpha}{2} - \theta\right)t}
	\end{equation*}
	where $Q(0) = (q(0),\partial_tq(0)) = (u(0), \alpha u(0)/2  + \partial_tu(0)).$
	Finally we multiply by $e^{-\frac{\alpha t}{2}}$ to conclude that for any $t>0$,
	\begin{align*}\label{eq:Linftydecayconstant}
		 \sup_{x \in [0,\ell]} |u(t,x)| &\leq e^{-\frac{\alpha t}{2}} |Q(0)|_{\e} + \frac{M \alpha^2 \sqrt{2 \ell}}{4\left(\frac{\alpha}{2} - \theta \right)}e^{-\theta t}|(u(0), \partial_tu(0))|_\h \nonumber \\
		&\leq \tilde{M} e^{-\theta t} |(u(0), \partial_tu(0))|_{\e}.
	\end{align*}
	A similar estimate is true for the velocity component. Indeed, 
 \begin{equation}\label{eq:uvelocity}
     \partial_tu=\Pi_2u=\partial_t(e^{-\frac{\alpha t}{2}}q)=-\frac{\alpha }{2}e^{-\frac{\alpha t}{2}}q+e^{-\frac{\alpha t}{2}}\partial_tq.  
 \end{equation}   
\noindent From \eqref{eq:q-mild} we have
\begin{equation} \label{eq:qvelocity}
	\partial_tq(t,x)=  \Pi_2 S_0(t) \begin{pmatrix} q(0) \\ \partial_tq(0) \end{pmatrix} + \frac{\alpha^2}{8}\int_0^t\Pi_2S_0(t-s)\begin{pmatrix} 0 \\ q(s) \end{pmatrix}ds.
\end{equation}
From \eqref{equation:D'Alembert2} and for each $t>0$ 
\begin{equation*}
    \widehat{\Pi_2S_0}(t)\begin{pmatrix} q(0) \\ \partial_tq(0) \end{pmatrix}(x)=\frac{1}{2}\big(q^o(0, x-t)-q^o(0, x+t)\big)-\frac{1}{2}\big(\widehat{(\partial_t q)^o }(0, x+t)+\widehat{(\partial_t q)^o }(0, x-t)\big),
\end{equation*}
(recall here that $\hat{\cdot}, \cdot^o $ denote anti-differentiation and odd periodic extension respectively)  hence 
\begin{equation*}
\bigg|\Pi_2S_0(t)\begin{pmatrix} q(0) \\ \partial_tq(0) \end{pmatrix}\bigg|_{C^{-1}}=\bigg|\widehat{\Pi_2S_0}(t)\begin{pmatrix} q(0) \\ \partial_tq(0) \end{pmatrix}\bigg|_{\infty}\leq |q(0)|_{\infty}+|\partial_tq(0)|_{C^{-1}}.
\end{equation*}
Turning to the convolution term,
\begin{equation*}
-2\widehat{\Pi_2S_0}(t-s)\begin{pmatrix} 0 \\ q(s) \end{pmatrix}(x)=\int_0^{x+(t-s)}q^o(s,y)dy+\int_0^{x-(t-s)}q^o(s,y)dy
\end{equation*}
From the latter, along with  the odd periodicity of $q^0$ and Proposition \ref{prop:L2-decay} we obtain 

$$ \bigg|\Pi_2S_0(t-s)\begin{pmatrix} 0 \\ q(s) \end{pmatrix}\bigg|_{C^{-1}}\leq  x\wedge (t-s) |q^o(s)|_\infty\leq \ell e^{(\frac{\alpha}{2}-\theta) s}|Q(0)|_\e.$$
Combining these estimates with \eqref{eq:uvelocity}, \eqref{eq:qvelocity} we deduce
$$ |\partial_t u(t)|_{C^{-1}}\leq C e^{-\theta t}|(u(0),\partial_tu(0))|_\e.   $$
The proof is complete.
\end{proof}

 \section{Uniform attraction and invariant sets for the non-linear dynamics}\label{sec:NonlinearStability}

We remind the reader that, from this point on, Assumption \ref{Assumption:xstar} is assumed to hold. The latter guarantees that there exists an asymptotically stable equilibrium $x^*$ for the parabolic problem $\partial_tu=\partial_x^2u+b(u)$ with Dirichlet boundary conditions on $(0,\ell).$ The goal of this section is to show that solutions $Z^0:=(u^0, \partial_tu^0)$ of the nonlinear hyperbolic problem
\begin{equation}\label{eq:deterministicnonlinearwave}
\partial_t^2u^0+\alpha\partial_tu^0=\partial_x^2u^0+b(u^0) 
\end{equation}
on $(0,\ell),$ with initial data in $\e,$  converge to the stable equilibrium $z^*:=(x^*, 0)$ as $t\to\infty$ uniformly over initial data that are sufficiently close to $z^*.$

The main result of this section, Theorem \ref{thm:domainofattraction}, guarantees the existence of domains $D$ that are uniformly attractive to $z^*$ in the topology of $\e.$ Next, we prove in Proposition \ref{corr:DExistence} that there exist uniformly attractive domains that are also invariant for the solution flow on $\e.$ Existence of such domains is essential for the study of exit problems in Section \ref{Sec:Metastability}. 

We start with a few preliminary stability properties of the linearized problem at $z^*.$

\begin{lem}\label{lem:linearizedsemigroup} Let $x^*, \lambda_b$ as in Assumption \ref{Assumption:xstar} and $\{\bar{S}_\alpha(t)\}_{t\geq 0}\subset \mathscr{L}(\h)$ be the semigroup generated by the operator $\partial_t^2+\alpha\partial_tu-\partial_x^2-b'(x^*).$ Under Assumption \ref{Assumption:xstar} the following hold:
\begin{enumerate}
    \item For any $\theta<\tfrac{\alpha}{8}\wedge\tfrac{\lambda_b}{\alpha},$ there exists $M>0$ such that for all $t>0, z\in\h$
$$ |\bar{S}_\alpha(t)z|_{\h}\leq Me^{-\theta t}|z|_{\h}.      $$
\item There exists $\theta_0>0$ such that for all $\theta<\theta_0$ there exists $M>0$  such that for all $t>0, z\in\e$
$$ |\bar{S}_\alpha(t)z|_{\e}\leq Me^{-\theta t}|z|_{\e}.      $$
\end{enumerate}    
\end{lem}

\begin{proof}\begin{enumerate}
    \item This follows once again by \cite[Chapter IV, Proposition 1.2]{temam2012infinite} by setting $A=\partial_x+b'(x^*)$ and $\lambda_1=\lambda_b.$
    \item Apart from the presence of an additional linear term, the proof is nearly identical to that of Theorem \ref{thm:Linftydecay} above. In particular, we perform the transformation $q := e^{\frac{\alpha t}{2}}u,$ where $\partial_t^2u+\alpha\partial_tu=\partial_x^2u+b'(x^*)u.$ Then $q$ satisfies the linear wave equation
       \begin{equation*}
	\partial_t^2q = \partial_x^2 + \bigg(\frac{\alpha^2}{4}+b'(x^*)\bigg)q.
\end{equation*}
Writing the latter in mild form and using the $L^2-$bound from part (1), as we did in Theorem \ref{thm:Linftydecay}, yields the desired estimate for a constant $M$ that depends on $|b'(x^*)|_{\infty}.$\end{enumerate}   \end{proof}
We are now ready to proceed to the stability analysis of the nonlinear problem. We emphasize here that our main theorem below applies to polynomial nonlinearities of arbitrary degree, provided that Assumption \ref{Assumption:xstar} is satisfied.
\begin{thm}[Uniform attraction]\label{thm:domainofattraction} Let $M$ be as in Lemma \ref{lem:linearizedsemigroup}(2). Next let $u^0_z$ denote the solution of \eqref{eq:deterministicnonlinearwave} with initial conditions $z=(x,y)\in \e.$ Under Assumptions \ref{Assumption:b}, \ref{Assumption:xstar} and for all $\theta< \min \{\alpha/8, \lambda_b/(4\alpha)\},$ there exists $\rho_0=\rho_0(\ell, M, |b^{(2)}(x^*)|_{\infty},\dots, |b^{(N)}(x^*)|_{\infty})$ such that for all $R<\rho_0, t>0$
 \begin{equation}\label{eq:Nonlinearstabilitybound}
   \sup_{z\in B_{\e}(z^*, R)}\big|Z_z^0(t)-z^*\big|_{\e}=  \sup_{z\in B_{\e}(z^*, R)}\big|(u^0_{z}(t), \partial_tu^0_{z}(t))-z^*\big|_{\e}\leq C_{\rho_0}e^{-\theta t}.
 \end{equation}    
\end{thm}

\begin{rem} This can be seen as a "refined", local version of the stable manifold theorem for Banach-space valued flows; see e.g. \cite[Theorem 1.1]{gallay1993center} for a related result. Indeed, while the latter proves existence of a stable manifold for the nonlinear dynamics, we are interested in finding stable domains that feature uniform attraction, which is a stronger property. For the case of wave equations, and to the best of our knowledge, such domains have only been identified in \cite[Theorem 1]{sattinger1968global} when the initial data is smooth. Theorem \ref{thm:domainofattraction} proves existence of such attractive domains in the much larger phase space $\e.$ Finally, from a look at the proof, it becomes clear that
the arguments are applicable to mild solutions of a larger class of evolution equations provided that a) the semigroup is contractive b) the nonlinearities are polynomial (or at least analytic) c) the solutions are real-valued.
\end{rem}

\begin{proof} Let $R>0$ and consider the ansatz
\begin{equation}\label{eq:ansatz}
    \begin{aligned}
     Z^0_z(t)=\sum_{n=0}^{\infty}R^nZ_n(t),
     \end{aligned}
\end{equation}
where $Z^0_z=(u^0_z,\partial_tu^0_z)$ is the solution with initial data $z\in B_\e(z^*,R)$ and $\{Z_n\}_{n\in\N}$ a sequence of functions in $C([0,\infty);\e)$ to be determined. From the latter, in combination with the mild formulation, we obtain
\begin{equation*}
    \begin{aligned}    \sum_{n=0}^{\infty}R^nZ_n(t)=S_\alpha(t)z^*+RS_\alpha(t)z_0+\int_{0}^{t}S_\alpha(t-s)\bigg(0, b\bigg(\sum_{n=0}^{\infty}R^n\Pi_1Z_n(s)\bigg)\bigg)ds,
     \end{aligned}
\end{equation*}
where we wrote $z=z^*+Rz_0,$ for some $z_0\in B_\e(0,1).$ From a (pointwise for every $x\in(0,\ell)$) Taylor expansion of $b$ around $\Pi_1Z_0(s)$ 
\begin{equation*}
    \begin{aligned}        b\bigg(\sum_{n=0}^{\infty}R^n\Pi_1Z_n(s)\bigg)&=\sum_{k=0}^{\infty}\frac{b^{(k)}(\Pi_1Z_0(s))}{k!}\bigg(\sum_{n=1}^{\infty}R^n\Pi_1Z_n(s)\bigg)^k\\&
    =b(\Pi_1Z_0(s))+\sum_{k=1}^{\infty}\frac{b^{(k)}(\Pi_1Z_0(s))}{k!}\sum_{n=k}^{\infty}R^n\sum_{i_1+\dots+i_k=n}\Pi_1Z_{i_1}(s)\dots \Pi_1Z_{i_k}(s)\\&    =b(\Pi_1Z_0(s))+\sum_{n=1}^{\infty}R^n\sum_{k=1}^{n}\frac{b^{(k)}(\Pi_1Z_0(s))}{k!}\sum_{i_1+\dots+i_k=n}\Pi_1Z_{i_1}(s)\dots \Pi_1Z_{i_k}(s),
    \end{aligned}
\end{equation*}
where the change in the order of summation on the last line will be justified from the absolute convergence of the series.
The last two displays allow to determine the functions $Z_n$ by identifying terms of the same order in $R.$ Indeed, from the zeroth order terms we see that 
$$Z_0(s)=S_\alpha(t)z^*+\int_0^tS_\alpha(t-s)\bigg(0, b(\Pi_1Z_0(s))\bigg)ds.    $$
Since $z^*$ is an equilibrium, it follows that $Z_0(s)=z^*$ for all $s.$ From the first order terms it follows that 
\begin{equation*}\label{eq:n=1}
    \begin{aligned}
        Z_1(s)=S_\alpha(t)z_0+\int_{0}^{t}S(t-s)\bigg(0, b'(\Pi_1Z_0(s))\Pi_1Z_1(s)\bigg)ds,
    \end{aligned}
\end{equation*}
i.e. $Z_1(s)=\bar{S}_\alpha(s) z_0,$ where $\bar{S}_\alpha$ is the semigroup generated by the linearisation of  \eqref{eq:deterministicnonlinearwave} around $x^*.$ In view of Lemma \ref{lem:linearizedsemigroup}(2), we have the estimate
$$|Z_1(s)|_\e\leq Ce^{-\theta s}|z_0|_\e\leq Me^{-\theta s}.$$
Turning to $n-$th order terms for $n\geq 1$ we have
\begin{equation*}
    Z_n(t)=\int_{0}^{t}\bar{S}_\alpha(t-s)\bigg(0, \sum_{k=2}^{N}\frac{b^{(k)}(x^*)}{k!}\sum_{i_1+\dots+i_k=n}\Pi_1Z_{i_1}(s)\dots \Pi_1Z_{i_k}(s)\bigg) ds,
\end{equation*}
where we used Assumption \ref{Assumption:b} which implies that all derivatives $b^{(k)}$ of order $k>N$ vanish.
With $A_1:=M,$ we assume that for each $k=1,\dots,n-1$ there is a constant $A_k>0$ such that 
\begin{equation*}\label{eq:Ak}
   |Z_k(s)|_\e\leq A_ke^{-\theta s}. 
\end{equation*} Notice that $Z_n$ satisfies a non-homogeneous linear equation with a source term that is independent of $Z_k, k\geq n. $ From the induction assumption, and with $\gamma$ from Assumption \ref{Assumption:b}, we thus have

\begin{equation*}
    \begin{aligned}
        |Z_n(t)|_{\e}&\leq A_1\int_{0}^{t} e^{-\theta (t-s)}\sum_{k=2}^{\gamma}\frac{|b^{(k)}(x^*)|_{\infty}}{k!}\sum_{i_1+\dots+i_k=n}|\Pi_1Z_{i_1}(s)\dots \Pi_1Z_{i_k}(s)|_{C^{-1}} ds\\&\leq 
        \ell A_1 \sum_{k=2}^{\gamma}\frac{|b^{(k)}(x^*)|_{\infty}}{k!}\int_{0}^{t}e^{-\theta (t-s)}\sum_{i_1+\dots+i_k=n}|\Pi_1Z_{i_1}(s)\dots \Pi_1Z_{i_k}(s)|_{\infty} ds\\&
        \leq \ell A_1 \sum_{k=2}^{\gamma}\sum_{i_1+\dots+i_k=n}A_{i_1}\dots A_{i_k}\frac{|b^{(k)}(x^*)|_{\infty}}{k!}\int_{0}^{t}e^{-\theta (t-s)}e^{-k\theta s} ds\\&
        \leq \ell A_1 e^{-\theta t}\sum_{k=2}^{\gamma}\sum_{i_1+\dots+i_k=n}A_{i_1}\dots A_{i_k}\frac{|b^{(k)}(x^*)|_{L^\infty}}{k!(k-1)}.
    \end{aligned}
\end{equation*}
\noindent Thus, one can take 
\begin{equation}
    \label{eq:Recursion}
    A_{n}:=\ell A_1 \sum_{k=2}^{\gamma}R_k\sum_{i_1+\dots+i_k=n}A_{i_1}\dots A_{i_k},\;\; R_k:=\frac{|b^{(k)}(x^*)|_{\infty}}{k!(k-1)}.
\end{equation}
We claim that, in order to obtain \eqref{eq:Nonlinearstabilitybound}, it remains to identify a lower bound $\rho_0>0$ for the radius of convergence of the series 
$$  F(\zeta)=\sum_{n=1}^{\infty} A_n \zeta^n.$$
Indeed, from \eqref{eq:ansatz} 
\begin{equation*}
    \begin{aligned}
        |Z^0_z(t)-z^*|_\e=\bigg|\sum_{n=0}^\infty R^nZ_n(t)-Z_0(t)\bigg|_\e\leq \sum_{n=1}^\infty R^n|Z_n(t)|_\e\leq e^{-\theta t} \sum_{n=1}^\infty R^nA_n
    \end{aligned}
\end{equation*}
 i.e. the desired estimate holds provided that $R<\rho_0.$
To find the latter, note that $F$ is invertible around $0$ (since $F'(0)=A_1\neq 0$) and from the recursive relations \eqref{eq:Recursion} its inverse is given by
$$F^{-1}(\zeta)=\frac{\zeta}{A_1}-\ell\sum_{k=2}^{\gamma} R_k\zeta^k.$$ The radius of convergence of $F$ can then be determined by the largest interval  around $0$ where $(F^{-1})'$ is non-zero. Since $F^{-1}$ is an $\gamma-$th degree polynomial with a root at $0$ there must exist $\rho_0>0$ such that $F^{-1}$ does not change monotonicity in $(-\rho_0, \rho_0).$ The proof is complete.
\end{proof}

\begin{example}\label{rem:Xstarexamples}
 A simple and important example of a nonlinearity  that satisfies the assumptions of Theorem \ref{thm:domainofattraction} is given by $b(x)=x-x^3.$ Indeed, Assumption \ref{Assumption:xstar} can be seen to hold from the analysis of \cite[Section 7]{faris1982large}. Explicit expressions for the equilibrium point $x^*$ and eigenvalues of the Schr\"odinger operator $\partial^2_x+b'(x^*)$ for this example as well as for polynomial nonlinearities of higher degree can be found in \cite[Section 5.1]{gasteratos2023importance} and references therein.

For this case, it is also possible to find an explicit expression for the size $\rho_0$ of domains of uniform attraction to $z^*.$ Indeed, by examining the proof of Theorem \ref{thm:domainofattraction}, we have that 
$$F^{-1}(z)=\frac{\zeta}{A_1}-\ell(R_2 \zeta^2+R_3\zeta^3).$$
Since the roots of $F'(\zeta)=A_1^{-1}-2\ell R_2\zeta-3\ell R_3\zeta^2$ are given by $\pm\rho,$
$$\rho_{\pm}:=\frac{2 R_2\pm\sqrt{4R_2^2+12A_1^{-1}\ell^{-1} R_3}}{6 R_3},$$
it suffices to take 
$$\rho_0=\rho_{+}=\frac{2|b''(x^*)|_{\infty}+2\sqrt{|b''(x^*)|_{\infty}^2+A_1^{-1}\ell^{-1}|b^{(3)}(x^*)|_{\infty} }}{|b^{(3)}(x^*)|_{\infty}}.$$
\end{example} 
\begin{rem} Our stability result is stated for nonlinear equations with polynomial nonlinearities of arbitrary degree and "rough" initial data in the sense that the initial position is only continuous. A similar result holds in the more general case where $p$ is only assumed to be the restriction of an analytic function on the real line. A proof of such a statement, in the case where the initial data is continuously differentiable, can be found in \cite{sattinger1968stability}.
\end{rem}

The previous theorem establishes existence of uniformly attracting balls around the asymptotically stable equilibrium $z^*.$ However, as we shall see in Section \ref{Sec:BoundaryPoints} (see also Remarks \ref{rem:invariance},  \ref{rem:cylinderNonInvariance}) such balls need not necessarily be invariant for the solution flow induced by $Z^0$ \eqref{eq:deterministicnonlinearwave}. Nevertheless, we show in the next proposition that invariant domains can be obtained by considering all the orbits issued from such uniformly attracting balls.

\begin{prop}\label{corr:DExistence} Let $Z^0_z=(u^0_z, \partial_tu^0_z)$ denote the solution of \eqref{eq:deterministicnonlinearwave} with initial data $z\in\e.$ Then, there exists an open bounded set $D\subset\e$ such that the following hold:
\begin{enumerate}
    \item (Uniform attraction) $\lim_{t\to\infty}\sup_{z\in D}|Z^0_z(t)-z^*|_\e =0.$
    \item (Invariance) For all $z\in D, t>0:$ $Z^0_z(t)\in D.$
\end{enumerate}    
\end{prop}

\begin{proof} Let $(t,z)\longmapsto\mathcal{Z}(t,z):=Z^0_z(t)$ denote the induced flow map. With $\rho_0$ as in Theorem \ref{thm:domainofattraction}, let
\begin{equation}\label{eq:orbitDomain}
    D:=\bigcup_{t\geq 0} \mathcal{Z}\bigg(t,B_\e(z^*, \rho_0)\bigg).
\end{equation}
  From \eqref{eq:Nonlinearstabilitybound} it is straightforward to deduce that $D$ is bounded and satisfies (1). Indeed, 
$$\sup_{z\in D}|z|_\e=\sup_{t\geq 0, z\in B_\e(z^*,\rho_0)}|Z^0_z(t)|_\e\leq C_{\rho_0}<\infty.$$
Moreover, from the flow property and definition of $D$ it follows that for each $z\in D$ there exist $z'\in B_\e(z^*, \rho_0), t'\geq 0$ such that for each $t\geq 0,$ $$Z^0_z(t)=Z^0_{Z^0_{z'}(t')}(t)=Z^0_{z'}(t+t')\in\mathcal{Z}( t+t', B_\e(z^*, \rho_0))\subset D.$$ It remains to show that $D$ is open. To this end, we have from time-reversibility and continuity of the wave flow that $\mathcal{Z}\big(t,B_\e(z^*, \rho_0)\big)= \mathcal{Z}^{-1}\big((-t),B_\e(z^*, \rho_0)\big),$ where, for each $z\in B_\e(z^*, \rho_0), $ $u=\Pi_1Z^0_z(-t)$ solves 
$$\partial_t^2u-\alpha\partial_tu=\partial^2_xu+b(u), u(0)=\Pi_1z.     $$
As in the proof of Theorem \ref{thm:Linftydecay}, we let $q(t):=e^{-\alpha t/2}u(t)$ and note that $\partial^2_tq=\partial^2_xq-\tfrac{\alpha^2}{4}q+e^{-\alpha t/2}b(e^{\alpha t/2}q).$ Letting $\{z_n\}\subset\e$ such that $z_n\rightarrow z\in\e$ and $T\geq t$ we set $R=\sup_{t\in[0,T]}|u_z(t)|_{\e}\vee \sup_{n\in\N, t\in[0,T]}|u_{z_n}(t)|_{\e}.$ Using the local Lipschitz continuity of $b$ on $B_{\e}(0,R),$  along with Gr\"onwall's inequality, we can then show  
$$  |Z^0_{z_n}(-t)-Z^0_{z}(-t)|_{\e}\leq e^{CT}|z_n-z|_{\e}\rightarrow 0   $$
as $n\to\infty.$ Hence, for each $t
\geq 0$ the time-reversed flow map 
$$ \e\ni z\longmapsto Z^0_{z}(-t)\in\e      $$
is continuous. The latter proves that, for each $t,$ $\mathcal{Z}\big(t,B_\e(z^*, \rho_0)\big)$ is the continuous inverse image of an open set hence an open set itself. The latter is also true for $D$ as an arbitrary union of open sets. The proof is complete.
\end{proof}

\section{Local well-posedness and local uniform large deviations}\label{sec:wellposedness}
\subsection{The controlled equation}\label{SS:ControlledEquation}
With the notation introduced in the previous section, we can recast \eqref{eq:modelsystem} as a stochastic evolution equation on $\e.$ In fact we shall consider here a slightly more general controlled equation which results from \eqref{eq:modelsystem} after replacing $\dot{W}$ by $\dot{W}+h/\epsilon$ for some deterministic control  $h\in L^2([0,T];H).$ Such perturbations arise naturally in the weak convergence approach to large deviations for \eqref{eq:modelsystem}. 
To this end, we define superposition operators $B:\e\rightarrow\e, B(x,y):=(0, b(x))$ and for each $f\in H,$ $\Sigma:\e\rightarrow\mathscr{L}(\e),$ $ \Sigma(x,y)f:=(0, \sigma(x))f.$ The controlled equation then takes the form
\begin{equation}\label{eq:controlledequation}
     \begin{aligned}
        dZ^{\epsilon, h}(t)=A_\alpha Z^{\epsilon, h}(t)dt+B(Z^{\epsilon, h}(t))dt+\Sigma\big(Z^{\epsilon, h}(t)\big)h(t)dt+\epsilon\Sigma\big(Z^{\epsilon, h}(t)\big)dW(t)\;, Z^{\epsilon, h}(0)=z\in\e,
     \end{aligned}
 \end{equation}
 with $A_\alpha$ as in \eqref{eq:dampedwavegenerator}.

 \begin{dfn}[Global mild solutions]\label{dfn:mild solutions} Let $T, \alpha,\epsilon >0, p\geq 1, z\in\e, h\in L^2([0,T];H).$\\ A process $\{Z_z^{\epsilon, h}(t); t\in[0,T]\}$ is a mild solution of \eqref{eq:controlledequation} in $L^p(\Omega;     C([0,T];\e))$ if $$\Pi_1Z_z^{\epsilon, h}\in L^p(\Omega;     C([0,T];C_0(0,\ell))),
  \Pi_2Z_z^{\epsilon, h}\in L^p(\Omega;     C([0,T];C^{-1}(0,\ell)))  $$ and $\pr-$almost surely for all $t\in[0,T]$
 \begin{equation*}\label{eq:MildControlledSolution}
     \begin{aligned}
    Z_z^{\epsilon, h}(t)&=S_\alpha(t)z+\int_0^tS_{\alpha}(t-s)B(Z_z^{\epsilon, h}(s))ds+\int_0^tS_{\alpha}(t-s)\Sigma(Z_z^{\epsilon, h}(s))h(s)ds\\&+\epsilon\int_0^tS_{\alpha}(t-s)\Sigma(Z_z^{\epsilon, h}(s))dW(s).
     \end{aligned}
 \end{equation*}
 A mild solution is global if it is well-defined for any $T>0.$
 \end{dfn}

 Our interest here is focused on local (mild) solutions of \eqref{eq:controlledequation} as defined below:

 \begin{dfn}[Local mild solutions]\label{dfn:local mild solutions}   Let $\epsilon>0, n\in\N, z\in\e,$ 
\begin{equation}\label{eq:bnlocalization}
 b_n(x):=
\begin{cases}
    &  b(n)+(n-x)b'(n)\;\;,\;x> n\\&
    b(x)\;\;,\;|x|\leq n\\&
    b(-n)+(x+n)b'(-n)\;\;,\;x<-n
\end{cases}
 \end{equation}
(resp. $\sigma_n$) be a Lipschitz approximation of $b$ (resp. $\sigma$) and $B_n$ (resp. $\sigma_n$) the corresponding nonlinear operator. If, for each $n\in\N,$ the evolution equation \eqref{eq:controlledequation}, with $B, \Sigma$
replaced by $B_n,\Sigma_n$ admits a unique global mild solution denoted by $Z^{\epsilon, h}_{z,n},$ we define a local mild solution to \eqref{eq:model} as follows:  Consider the family of stopping times
\begin{equation*}\label{eq:LocalizingStoppingTimes}
	\tau^{\epsilon, h}_{z,n}:=\inf\{t>0 : | \Pi_1Z^{\epsilon, h}_{z,n}(t)|_{C_0(0,\ell)}> n\}.
\end{equation*}
By uniqueness, the $Z^{\epsilon,h}_{z,n}$ are consistent in the sense that $$Z_{z,n}^{\epsilon, h}=Z_{z,m}^{\epsilon, h}$$ on $[0, \tau^{\epsilon, h}_{z,m}]$ for each $m<n.$     
A local mild solution to \eqref{eq:model} is then defined  by 
$Z_z^{\epsilon,h}:=Z_{z,n}^{\epsilon, h}\in C([0,T];\e)$ for all $T\leq \tau^{\epsilon, h}_{z,n}. $ The explosion time of a local mild solution is given by 
\begin{equation}\label{eq:explosiontime}   \tau^{\epsilon, h}_{z,\infty}:=\sup_{n\in\N}\tau^{\epsilon, h}_{z,n}.
\end{equation}
 \end{dfn}

This localization procedure requires the use of supremum norms, and cannot easily be accomplished in other topologies. In Proposition \ref{prop:WellPosedness} we will prove local well-posedness by showing that for each $n$, there exists a unique global mild solution $Z^{\epsilon,h}_{z,n}.$

 Before we present a local well-posedness result for the controlled equation, we proceed with an apriori estimate for the stochastic convolution. To this end, let $\{\Psi^\epsilon\}_{\epsilon}\subset L^p(\Omega;     C([0,T];C_0(0,\ell)))$ and consider the process
\begin{equation*}\label{eq:StochasticConvolutionDfn}
     \Gamma_{\Psi^\epsilon}(t):=\int_0^tS_{\alpha}(t-s)(0,\Psi^\epsilon(s))dW(s).
 \end{equation*} 
\begin{lem}\label{lem:StochasticConvolutionApriori} Let $T>0, p\geq 2$ and $\{\Psi_1^\epsilon;\epsilon>0\}, \{\Psi_2^\epsilon;\epsilon>0\}\subset L^p(\Omega; C([0,T];C_0(0,\ell))).$ Under Assumption \ref{Assumption:sigma} there exists $C>0$ such that  \begin{equation}\label{eq:stochasticconvolutionbound}
     \begin{aligned}
        \ex| \Gamma_{\Psi_1^\epsilon}-\Gamma_{\Psi_2^\epsilon}|_{C([0,T];\e)}^p\leq C\ex|\Psi_1^\epsilon-\Psi_2^\epsilon|^p_{ C([0,T];C_0(0,\ell))}.
     \end{aligned}
 \end{equation}
 \end{lem}
 \noindent The proof is postponed to Appendix \ref{App:StochConv}.

 \begin{rem} Estimates similar to \eqref{eq:stochasticconvolutionbound}, with $\e$ replaced by the Hilbert spaces $\h_\delta$  can be found in several places in the literature (see e.g. \cite[proof of Theorem 4.2]{cerrai2006smoluchowski}). The difference here is that we work on a stronger topology and hence we require stronger estimates. For this reason we present a self-contained proof for the previous lemma and the next proposition.
     
 \end{rem}
 \begin{prop}[Local well-posedness of controlled equation] \label{prop:WellPosedness} Under Assumptions \ref{Assumption:b}, \ref{Assumption:sigma},  \eqref{eq:controlledequation} admits, for each $\epsilon>0, z\in\e,$ and adapted $ h\in L^2([0,T];H)$ a unique local mild solution on $[0,T \wedge \tau^{\epsilon, h}_{z,\infty}).$ 
 \end{prop}

 \begin{proof}   
  It suffices to prove well-posedness for $\epsilon=1.$ Moreover, from Definition \ref{dfn:local mild solutions} it suffices to prove global well-posedness for $Z^{1,h}_{z,n}$ for each $n\in\N.$ The local solution to \eqref{eq:controlledequation} is then defined by patching up such global solutions up to the explosion time.  In turn, since we are working on $\e,$ this reduces to the case when $b, \sigma$ are bounded and Lipschitz continuous.

 To this end we consider the map $ \mathscr{S}:\e\times L^p(\Omega;     C([0,T]; \e))\rightarrow  L^p(\Omega;     C([0,T]; \e))$
 \begin{equation*}
     \begin{aligned}
         \mathscr{S}(z, Z)(t)&=S(t)z+\int_0^tS_\alpha(t-s) B(Z(s)\big)ds+\int_0^tS_\alpha(t-s) \Sigma(Z(s)\big)h(s)ds+\int_0^tS_\alpha(t-s) \Sigma(Z(s)\big)dW(s)\\&
         =:S(t)z+\rho_{Z}(t)+H_{Z}(t)+\Gamma_Z(t).
     \end{aligned}
 \end{equation*}
 For $Z_1, Z_2\in  L^p(\Omega;     C([0,T]; \e)),$ Theorem \ref{thm:Linftydecay} yields the almost sure estimates
 \begin{equation*}
 	\begin{aligned}
 	|\rho_{Z_1}(t)-  \rho_{Z_2}(t) |_{\e}& \leq \int_0^t \big|S_\alpha(t-s)\big( B(Z_1(s)\big)-B(Z_2(s)\big)|_\e ds\\&
 	\leq \int_0^t e^{-\theta (t-s)} |b(\Pi_1 Z_1(s))- b(\Pi_1 Z_2(s))|_{C^{-1}}ds\\&
 	\leq C\int_0^t  |b(\Pi_1 Z_1(s))- b(\Pi_1 Z_2(s))|_{\infty}ds\\&
 	\leq C|b|_{Lip}T \sup_{t\in[0,T]}|Z_1(t)-Z_2(t)|_{\e},
 	 \end{aligned}
 \end{equation*}

 \begin{equation*}
 	\begin{aligned}
 			|H_{Z_1}(t)-H_{Z_2}(t)|_\e&\leq
 		\int_0^t e^{-\theta (t-s)} \big|\big(\sigma(\Pi_1 Z_1(s))- \sigma(\Pi_1 Z_2(s))\big)h(s)\big|_{C^{-1}}ds\\&
 		\leq C \int_0^t \big|\big(\sigma(\Pi_1 Z_1(s))- \sigma(\Pi_1 Z_2(s))\big)h(s)\big|_{H}ds\\&
 		\leq  C|\sigma|_{Lip} \sqrt{T}\sup_{t\in[0,T]}|Z_1(t)-Z_2(t)|_{\e}|h|_{L^2([0,T];H)}.
 	\end{aligned}
 \end{equation*}
 In view of Lemma \ref{lem:StochasticConvolutionApriori} we also have 
 $$  \ex| \Gamma_{Z_1}-\Gamma_{Z_2}|_{C([0,T];\e)}^p\leq C_T\ex|Z_1-Z_2|^p_{ C([0,T];C_0(0,\ell))}   $$
 for a constant $C=C_T$ that is increasing in $T.$ From these estimates we deduce that 
 $$\big| \mathscr{S}(z, Z_1)-\mathscr{S}(z, Z_2)\big|_{ L^p(\Omega;     C([0,T]; \e)}    \leq C_T| Z_1-Z_2|_{ L^p(\Omega;     C([0,T]; \e)} .$$
In particular, for $T\leq T_0$ sufficiently small the map $\mathscr{S}(z,\cdot)$ is a contraction. From a classical fixed-point argument it follows that, for each $\epsilon>0,$ \eqref{eq:controlledequation} has a unique mild solution  $Z^{\epsilon,h}\in L^p(\Omega;     C([0,T_0]; \e)).$ The restriction on $T$ can then be lifted by repeating this argument on time-intervals of the form $[kT_0,  (k+1)T_0], k\in\N$. Thus for each $n\in\N$ $Z^{\epsilon, h}_{z,n}$ is globally well-posed and the proof is complete.
    	   \end{proof}

           \begin{rem}\label{rem:LocalSolutions} The previous proof establishes existence and uniqueness of local solutions to \eqref{eq:controlledequation} i.e. solutions defined up to explosion time. Local well-posedness is a consequence of working in the topology of continuous functions. Our exit problem asymptotics in Section \ref{Sec:Metastability} are valid for exits of such local solutions from bounded neighborhoods of a stable equilibrium. Moreover they provide asymptotic lower bounds for mean explosion times as $\epsilon\to 0$  (Remark \ref{rem:explosiontimes}). Nevertheless,
           as we explain in Section \ref{sec:GlobalSolutions}, it is also possible to obtain global well-posedness on $\e$ in the case where $b$ is a nonlinearity with at most cubic growth.
           \end{rem}
    	   
\subsection{Continuity properties of the skeleton equation}\label{sec:skeletoncontinuity} From this point on, $Z^\epsilon_z=(u^\epsilon_z,\partial_tu^\epsilon_z), Z^{u}_z:=Z^{0,u}_z$ will denote the local mild solutions of \eqref{eq:model} and \eqref{eq:controlledequation} with initial data $Z^\epsilon_z(0)=Z^u_z(0)=z\in\e$ and control $u$  respectively. This section contains a number of useful lemmas regarding the so-called skeleton equation $Z_z^{u},$ where $u\in L^2([0,T];H)$ is a deterministic control. These properties will be useful for the exit problem analysis of Section \ref{Sec:Metastability}.

\begin{lem}\label{lem:Zaprioricompactnesslem}
	For each $n\in\N,$ let $\Sigma_n$ denote the Lipschitz approximation of the operator $\Sigma$ per Definition \ref{dfn:local mild solutions}. For $n
    \in\N, T,\alpha>0, z\in\mathcal{E},$ $u\in L^2([0,T]; H, Z\in C([0,T];\e)$,
 under Assumption \ref{Assumption:sigma}, the linear operator 
	$$L^2([0,T];H)\ni u\longmapsto   K(u)(t):= \int_{0}^{t}S_\alpha(t-s)\Sigma_n(Z(s))u(s)ds\in C([0,T];\e) $$
	is compact.		
\end{lem}
\begin{proof}

Let $\mathcal{U}\subset L^2([0,T];H)$ be a bounded set. The statement follows by showing that $K(\mathcal{U})$ is relatively compact in $C([0,T];\e).$ First note that, from Lemma \ref{lem:StochasticConvolutionApriori} (with $\epsilon=1, \Psi^\epsilon_2=0$ and $\Psi^\epsilon_1=\Sigma(Z)$) we obtain the bound
	\begin{equation}\label{Kcomp}
	\begin{aligned}
	\sup_{u\in\mathcal{U}}\sup_{t\in[0,T]}\bigg|\int_{0}^{t}S_\alpha(t-s)\Sigma_n(Z(s))u(s)ds\bigg|_{\h_1}\leq C_n,
	\end{aligned}
	\end{equation}	
	for some constant $C_n$ that depends on $T, |\sigma_n|_\infty$ and the diameter of $\mathcal{U}.$	We claim now that $\h_1$ is relatively compact in $\mathcal{E}.$ Indeed, the inclusion $H^1\subset C_0$ is compact by Sobolev embedding. Turning to the inclusion $L^2\subset C^{-1}$ we have , for any $\phi\in H,$
	$$|\phi|_{C^{-1}}=\big|\hat{\phi}\big|_{\infty}:=\sup_{r\in(0,\ell)}\bigg|\int_{0}^{r}\phi(\xi)d\xi\bigg|.$$
	Hence, for any bounded sequence $\{\phi_m\}_{m\in\N}\subset{H},$ the sequence of anti-derivatives $\{\hat{\phi}_m\}_{m\in\N}\subset C_0(0,\ell)$ is uniformly bounded and uniformly $1/2-$H\"older continuous. From the classical Arzel\`a-Ascoli theorem it follows that, up to a subsequence, $\hat{\phi}_m$ converges uniformly to a function $\hat{\phi}\in C_0.$ Passing if necessary to a further subsequence and without changing the notation, the weak $L^2-$compactness of $\{\phi_m\}_m$ and uniqueness of the limit imply that $\hat{\phi}$ has a weak derivative $\hat{\phi}'\in H.$ From the latter we deduce that $\lim_{m\to\infty}\phi_m=\hat{\phi}'$ in $C^{-1},$ which allows us to conclude our initial claim.
	
	It remains to show that $K(\mathcal{U})\subset C([0,T];\e)$ is uniformly equicontinuous. To this end, note that for any $s<t\in[0,T], u\in\mathcal{U}$ we have 
	\begin{equation*}
	K(u)(t)-K(u)(s)=\int_{s}^{t}S_\alpha(t-r)\Sigma_n(Z(r))u(r)dr+\int_{0}^{s}[S_\alpha(t-r)-S_\alpha(s-r)]\Sigma_n(Z(r))u(r)dr.
	\end{equation*}
	From continuity of the semigroup $S$ in $\h_1$ (\ref{prop:L2-decay}) and the local boundedness of $\sigma, u$ we obtain 
	\begin{equation*}
	\begin{aligned}
	\big|K(u)(t)-K(u)(s)\big|_{\e}&\leq C\int_{s}^{t}|\Sigma_n(Z(r))u(r)|_{\e}dr\\&+C\int_{0}^{s}|S_\alpha(s-r)[S_\alpha(t-s)-I]|_{\mathscr{L}(\h_1)}|\sigma_n(\Pi_1Z(r))u(r)|_{H}dr\\&
	\leq C_{\sigma_n}(t-s)^{1/2}|u|_{L^2([0,T];H)}+C_n\int_{0}^{s}e^{-\omega(s-r)}|S_\alpha(t-s)-I|_{\mathscr{L}(\h_1)}|u(r)|_{H}dr\\&
	\leq C_{\sigma_n, \mathcal{U}}(t-s)^{1/2}+C_{n,T,\mathcal{U}}|S_\alpha(t-s)-I|_{\mathscr{L}(\h_1)},
	\end{aligned}
	\end{equation*}
	where the latter holds uniformly over $u\in\mathcal{U}$ and $s,t\in[0,T].$
	Next, let $\epsilon>0.$ From the strong continuity of $S_\alpha,$  there exists $\delta_1=\delta_1(T,\mathcal{U})>0$ such that for all $s,t$ with $|t-s|<\delta_1$ the second term in the last display is below $\epsilon/2.$ Thus, the choice $\delta_n:=\min\{\delta_1, (\epsilon/2C_{\sigma_n, \mathcal{U}})^2\}$ satisfies the $\epsilon-$challenge for uniform equicontinuity of $K(\mathcal{U}).$ The proof is complete in view of the latter, \eqref{Kcomp} and an infinite-dimensional version of Arzel\`a-Ascoli theorem.	
\end{proof}
\noindent In the next lemma we show that localizations of the solution to the skeleton equation are continuous  with respect to controls and initial conditions. Then we prove a quantitative estimate for the difference between controlled and uncontrolled trajectories in terms of the $L^2$ norm of the controls.
\begin{lem}\label{lem:controliccontinuity}
\begin{enumerate} Let $Z^u_{n,z}, Z^0_{n,z}$ as in Definition \ref{dfn:mild solutions}. Under Assumptions \ref{Assumption:b}, \ref{Assumption:sigma} the following hold:
    \item Let $T, N>0$ and $\mathcal{U}_N\subset L^2([0,T];H)$ be the centered ball of radius $N,$ endowed with the weak $L^2$ topology. For each $n\in\N$ the map 
	$$\e\times \mathcal{U}_N \ni (z, u)\longmapsto Z^u_{n,z}\in C([0,T];\e)$$
is continuous.
\item Let $T>0, n\in\N, u\in L^2([0,T];\h).$ There exists a smooth, non-decreasing function $\Lambda_n:[0,\infty)\rightarrow [0,\infty)$ with $\Lambda(0)=0, \lim_{x\to\infty}\Lambda(x)=\infty$ such that for all $z\in\e$ $$| Z_{n,z}^{u}-Z_{n,z}^0|_{C([0,T];\e)}\leq \Lambda_n\big(|u|_{L^2([0,T];H)}\big). $$
\end{enumerate}
\end{lem}
\begin{proof}
Throughout the proof and for the sake of lighter notation we shall drop the subcript $n$ and work instead in the case of bounded, Lipschitz continuous coefficients $b,\sigma.$
\begin{enumerate}
    \item \noindent Since the weak topology on bounded sets is metrizable, it suffices to consider a sequence $\{(z_m, u_m)\}_{m\in\N}\subset\mathcal{U}_N\times\e$ that converges weakly, as $m\to\infty,$ to a pair $(z, u)\in\mathcal{U}_N\times\e.$ For $t\in[0,T]$ we have 
\begin{equation*}
\begin{aligned}
|Z^{u_m}_{z_m}(t)-Z^{u}_{z}(t)\big|_{\e}&\leq \big| S_\alpha(t)(z_m-z)\big|_{\e}+\int_{0}^{t}\bigg| S_\alpha(t-s)\big[B(Z^{u_m}_{z_m}(s))- B(Z^{u}_{z}(s))\big]\bigg|_{\e}ds\\&
+\bigg|\int_{0}^{t} S_\alpha(t-s)\big[\Sigma(Z^{u_m}_{z_m}(s))u_m(s)- \Sigma(Z^{u}_{z}(s))u(s)\big]ds\bigg|_{\e}
\end{aligned}
\end{equation*}
From Assumption \ref{Assumption:b}, Assumption \ref{Assumption:sigma} and Theorem \ref{thm:Linftydecay} we obtain 
\begin{equation*}\label{Zcontbnd1}
\begin{aligned}
|Z^{u_m}_{z_m}(t)-Z^{u}_{z}(t)\big|_{\e}&\leq C| z_m-z|_{\e}+C\int_{0}^{t} \big|b(\Pi_1Z^{u_m}_{z_m}(s))- b(\Pi_1Z^{u}_{z}(s))\big|_{H}ds\\&
+C\int_{0}^{t}\big|\big[\sigma(\Pi_1Z^{u_m}_{z_m}(s))-\sigma(\Pi_1Z^{u}_{z}(s))\big]u_m(s)\big|_{H}ds\\&+C\bigg|\int_{0}^{t}S_\alpha(t-s) \Sigma(Z^{u}_{z}(s))\big[u_m(s)-u(s)\big]ds\bigg|_{\e}\\&
\leq C| z_m-z|_{\e}+C|b|_{Lip}\int_{0}^{t}\big|\Pi_1Z^{u_m}_{z_m}(s)-\Pi_1Z^{u}_{z}(s)|_{H}ds\\&
+|\sigma|_{Lip}\int_{0}^{t}\big|\Pi_1Z^{u_m}_{z_m}(s)-\Pi_1Z^{u}_{z}(s)|_{C_0(0,\ell)}|u_m(s)|_{H}ds\\&
+C\sup_{t\in[0,T]}\bigg|\int_{0}^{t}S_\alpha(t-s) \Sigma(Z^{u}_{z}(s))\big[u_m(s)-u(s)\big]ds\bigg|_{\e}.
\end{aligned}
\end{equation*}
Thus, from an application of the Cauchy-Schwarz inequality for the control term we obtain
\begin{equation*}
\begin{aligned}
|Z^{u_m}_{z_m}(t)-Z^{u}_{z}(t)\big|^2_{\e}&
\leq C| z_m-z|^2_{\e}+CT^{1/2}\int_{0}^{T}\big|Z^{u_m}_{z_m}(s)-Z^{u}_{z}(s)|^2_{\e}ds\\&+\tilde{C}T^{1/2}\int_{0}^{T} |Z^{u_m}_{z_m}(s)-Z^{u}_{z}(s)\big|^2_{\e}ds\\&
+|\sigma|_{Lip}\sup_{m\in\N}|u_m|^2_{L^2([0,T];H)}\int_{0}^{T}\big|Z^{u_m}_{z_m}(s)-Z^{u}_{z}(s)|^2_{\e}ds\\&
+C\sup_{t\in[0,T]}\bigg|\int_{0}^{t}S_\alpha(t-s) \Sigma(Z^{u}_{z}(s))\big[u_m(s)-u(s)\big]ds\bigg|^2_{\e}.
\end{aligned}
\end{equation*}
Gr\"onwall's inequality then furnishes
\begin{equation*}
\begin{aligned}
\sup_{t\in[0,T]}|Z^{u_m}_{z_m}(t)-Z^{u}_{z}(t)\big|^2_{\e}&
\leq Ce^{C_{\sigma, b, N, T}}\bigg(| z_m-z|^2_{\e}+\sup_{t\in[0,T]}\bigg|\int_{0}^{t}S_\alpha(t-s) \Sigma(Z^{u}_{z}(s))\big[u_m(s)-u(s)\big]ds\bigg|^2_{\e}\bigg).
\end{aligned}
\end{equation*}
In view of Lemma \ref{lem:Zaprioricompactnesslem} with $Z=Z^{u}_{z},$ the right-hands side vanishes as $n\to\infty$ and the conclusion follows.
\item 
For $t\leq T$ we have 
\begin{equation*}
   Z_z^u(t)-Z_z^0(t)=\int_{0}^tS_\alpha(t-s)\big[b(  \Pi_1Z_z^u(s))- b(\Pi_1  Z_z^0(s))\big]ds+\int_{0}^tS_\alpha(t-s)\sigma( \Pi_1 Z_z^u(s))u(s)ds.
\end{equation*}
From the local boundedness of $\sigma$, the second term on the right hand side satisfies 

\begin{equation*}
    \begin{aligned}
        \bigg|\int_{0}^tS_\alpha(t-s)\sigma(  Z_z^u(t))u(t)dt\bigg|_{\e}&\leq C \bigg|\int_{0}^tS_\alpha(t-s)\sigma( Z_z^u(t))u(t)dt\bigg|_{\h_1}\\&\leq C|\sigma|_{\infty}\int_{0}^{t}e^{-\theta(t-s)} |u(s)|_{H}ds\leq C|\sigma|_{\infty} T^{1/2}|u|_{L^2([0,T];H)}.
    \end{aligned}
\end{equation*}
As for the first term, we use the (local) Lipschitz continuity of $b$ to obtain
\begin{equation*}
\begin{aligned}
    \bigg|\int_{0}^{t}S_\alpha(t-s)\big[b(  Z_z^u(s))&- b(  Z_z^0(s))\big]ds\bigg|_{\e}\leq C\int_{0}^{T}\sup_{r\leq s}|Z^{u}_{z}(r)-Z^{0}_{z}(r)|_{\e} ds.
\end{aligned}  
\end{equation*}
From the last two estimates along with Gr\"onwall's inequality we may conclude
\begin{equation*}
 | Z_z^u-Z_z^0|_{C([0,T];\e)}\leq C'T^{1/2}|u|_{L^2([0,T];H)}\exp(CT)
\end{equation*}
for constants $C, C'>0.$  The proof is complete upon observing that the function $$\Lambda(x)=C'T^{1/2}\exp(CT)x, x\geq 0$$ satisfies the desired properties.\end{enumerate}\end{proof}

  \subsection{Local Uniform Large Deviations}\label{sec:LULDP}   
        We conclude this section by formulating and proving a local  Uniform Large Deviations Principle (LULDP), per Definition \ref{Def:LocalULDP}, for the local solutions $\{Z_{z,n}^\epsilon; \epsilon\in(0,1)\}$ over bounded sets of initial data. This is an important ingredient in the study of the exit problem for local solutions of \eqref{eq:model}.
        
        To do so, we first state the definition of a global Uniform Large Deviations Principle (GULDP) in Definition \ref{Def:GlobalULDP}. The idea that we exploit is that given that we are interested in the exit from a  non-empty bounded set  $D\subset\e$ with $z^*\in D$, proving a local ULDP for the family $\{Z_z^\epsilon; \epsilon\in(0,1)\}$ is the same as proving a global ULDP for the family $\{Z_{z,n}^\epsilon; \epsilon\in(0,1)\}$ where the process $Z_{z,n}^\epsilon$ is a localized version of $Z_z^\epsilon$. 
        
       \begin{rem} As a matter of fact, as it is implied by the context and we will more specifically elaborate in Section \ref{sec:Controllability}, we are mainly interested in the process $\{Z_{z,n_{D}}^\epsilon; \epsilon\in(0,1)\}$, i.e., for a specific value of $n=n_D$. In particular, $n_D$ is fixed to be large enough so that a ball with radius $n_D$ contains $D$, i.e., $n_D:=\inf\big\{ n\in\N: D\subset \bar{B}_\e(0, n)\}$, see also  (\ref{eq:nDdefinition}).
        \end{rem}

        \begin{dfn}[Global ULDP]\label{Def:GlobalULDP} Let $D\subset\e$ be a bounded set of initial data. 
        \begin{enumerate}
                  \item[1.] 
                Let  $z\in\e,T>0.$
        A lower semicontinuous function $I_{z,T}:C([0,T];\e) \to [0,+\infty]$ is called a rate function. For each $s\geq 0$ we define the sublevel sets       \begin{equation}
        \label{eq:SublevelSets}
             \Phi_{z,T}(s) := \left\{\phi \in C([0,T];\e): I_{z,T}(\phi) \leq s \right\}.
         \end{equation}
         \item[2.] A family $\{Z_z^\epsilon; \epsilon\in(0,1)\}$ of $C([0,T];\e)-$valued random elements is said to satisfy a (global) ULDP with respect to $D$ if for any $T>0$, $s_0>0$ and $\delta>0$ the following hold: \begin{enumerate}
                  \item 
                  \begin{equation} 
                  \label{eq:LULDP-lower}
                      \liminf_{\epsilon \to 0} \inf_{z \in D} \inf_{\phi \in \Phi_{z,T}(s_0)} \bigg\{\epsilon^2 \log \pr\bigg(|Z^{\epsilon}_z - \varphi|_{C([0,T];\e)}<\delta\bigg) + I_{z,T}(\varphi) \bigg\} \geq 0,
                  \end{equation}
                  \item 
                  \begin{equation} 
                  \label{eq:LULDP-upper}
                      \limsup_{\epsilon \to 0} \sup_{z \in D} \sup_{s \in [0,s_0]} \bigg\{\epsilon^2\log \pr \bigg(\textnormal{dist}_{C([0,T];\e)}\big(Z^\epsilon_z, \Phi_{z,T}(s) \big)\geq\delta\bigg) + s \bigg\} \leq 0.
                  \end{equation}
              \end{enumerate}

\end{enumerate}
      
             \end{dfn}

Next, we define the local ULDP. 
       \begin{dfn}[Local ULDP] \label{Def:LocalULDP} Let $T>0, z\in\e$ and $\{I^n_{z,T};n\in\N\}$ be a countable collection of good rate functions on $C([0,T];\e).$ The family $\{Z_z^\epsilon; \epsilon\in(0,1)\}$ of $C([0,T];\e)-$valued local mild solutions to \eqref{eq:model}  is said to satisfy a Local Uniform Large Deviations Principle (LULDP) over bounded sets of initial data if, for all $n\in\N,$ the family of localized (global) mild solutions  $\{Z^{\epsilon}_{z,n}; \epsilon\in(0,1)\}$ (per Definition \ref{dfn:local mild solutions}) satisfies a Global ULDP, per Definition \ref{Def:GlobalULDP}, over bounded sets with rate function $I^n_{z,T}.$ The level sets of the latter will be denoted by $\Phi^n_{z,T}(s)$ for $s\geq 0.$
\end{dfn}

    \begin{prop}\label{prop:LULDP} Let 
    $Z_z^u$ be the local mild solution of the skeleton equation starting at $z\in\e$. Under Assumptions \ref{Assumption:b}, \ref{Assumption:sigma}, local solutions of \eqref{eq:model} satisfy, per Definition \ref{Def:LocalULDP}, a LULDP over bounded sets of initial data with good rate functions             \begin{equation}\label{eq:LRateFunctionWave}
             I^n_{z,T}(\phi) := \inf_{u\in L^2([0,T];H)} \left\{ \frac{1}{2}\int_0^T \int_0^\ell |u(s,y)|^2 dyds: \phi=Z^{u}_{z,n}  \right\},\;n\in\N
         \end{equation}
          for all $T>0$ and $\phi\in C([0,T];\e),$ where the convention that $\inf\varnothing=\infty$ is understood.
    \end{prop}

              \begin{proof} Let $T>0, n\in\N$   
              and $Z^\epsilon_{z,n}$ be a localized mild solution, per Definition \ref{dfn:local mild solutions}, with control $h=0.$ We shall rely on the equivalence of the ULDP and Equicontinuous Uniform Laplace Principle (EULP); see e.g. \cite[Theorem 2.9]{salins2019equivalences}).  In view of  \cite[Theorem 2.12]{salins2019equivalences}, the latter holds over bounded subsets of initial data, provided that for any $\delta>0,$ $D\subset\e$ bounded and $M>0,$
    	   \begin{equation}
    	       \label{eq:EULPsufficient}
            \lim_{\epsilon\to 0}\sup_{z\in D}\sup_{u\in\mathcal{P}_2^M}\pr\bigg[\sup_{t\in[0,T]}\big|Z_{z, n}^{\epsilon,u}(t)-Z_{z, n}^{u}(t)\big|_{\e}\geq \delta \bigg]=0.
    	   \end{equation}
  Here, $\mathcal{P}_2^M$ is the collection of adapted, $H-$valued stochastic controls $u$ such that $\pr( |u|^2_{L^2([0,T];H)}\leq M)=1.$ Moreover, $Z^{\epsilon,u}_{z,n}$ is a localized mild solution, per Definition \ref{dfn:local mild solutions}, with control $h=u.$

   \noindent We can now estimate the latter using the local Lipschitz continuity of $b,\sigma$ and Gr\"onwall's inequality. Indeed,     	   
    	   \begin{equation*}
    	   \begin{aligned}
\big|Z_{z,n}^{\epsilon,u}(t)-Z_{z,n}^{0,u}(t)\big|_{\e}&\leq \bigg|\int_{0}^{t}S_{\alpha}(t-s)\big[B_n(Z_{z,n}^{\epsilon,u}(s))-B_n(Z_{z,n}^{0,u}(s))\big]ds\bigg|_{\e}\\&+ \bigg|\int_{0}^{t}S_{\alpha}(t-s)\big[\Sigma_n(Z_{z,n}^{\epsilon,u}(s))-\Sigma_n(Z_{z,n}^{0,u}(s))\big]u(s)ds\bigg|_{\e}\\&
+\epsilon\bigg|\int_{0}^{t}S_{\alpha}(t-s)\Sigma_n(Z_{z,n}^{\epsilon,u}(s))dW(s)\bigg|_{\e}\\&\leq C\int_{0}^{t}\big|b_{n}(\Pi_1Z_{z,n}^{\epsilon,u}(s))-b_{n}(\Pi_1Z_{z,n}^{0,u}(s))\big|_{H}ds\\&
+C\int_{0}^{t}\big|\big[\sigma_n(\Pi_1Z_{z,n}^{\epsilon,u}(s))-\sigma_n(\Pi_1Z_{z,n}^{0,u}(s))\big]u(s)\big|_{H}ds+2C\epsilon\sup_{t\in[0,T]}\big|\Gamma_{n}^\epsilon(t)\big|_\e
    	   \end{aligned}
    	   \end{equation*}
where $\Gamma^\epsilon_{n}(t)$ is the stochastic convolution term.
 From Assumption \ref{Assumption:b}, \ref{Assumption:sigma}, there exists $L_{n}$ such that 
\begin{equation*}
\begin{aligned}
\big|Z_{z,n}^{\epsilon,u}(t)-Z_{z,n}^{0,u}(t)\big|_{\e}&\leq L_n\int_{0}^{t}\big|\Pi_1Z_{z,n}^{\epsilon,u}(s))-\Pi_1Z_{z,n}^{0,u}(s))\big|_{\infty}ds\\&
+C_{\sigma_n}\int_{0}^{t}\big|\Pi_1Z_{z,n}^{\epsilon,u}(s)-\Pi_1Z_{z,n}^{0,u}(s)\big|_{\infty}|u(s)|_{H}ds+C\epsilon\sup_{t\in[0,T]}|\Gamma_{n}^\epsilon(t)|_\e.
\end{aligned}
\end{equation*}

\noindent From the $L^2-$bound on $u$ and Gr\"onwall's inequality we get

\begin{equation*}
\begin{aligned}
\ex\sup_{t\in[0,T]}\big|Z_{z,n}^{\epsilon,u}(t)-Z_{z,n}^{0,u}(t)\big|^2_{\e}\leq C \epsilon^2e^{C_{n}T}\sup_{\epsilon\in(0,1), n\in\N}\ex\sup_{t\in[0,T]}\big|\Gamma_{n}^\epsilon(t)\big|^2_\e\leq C\epsilon^2e^{C_{n}T},
\end{aligned}
\end{equation*}
 \noindent where the last estimate follows by applying Lemma \ref{lem:StochasticConvolutionApriori} with $\Psi_2=0, p=2$ and using the boundedness of $\sigma_n.$  Taking $\epsilon\to 0$ we see that \eqref{eq:EULPsufficient} holds true for any $\delta, T>0$ and any bounded set $D\subset\e$ of initial data. 

 From the previous argument we deduce that the family of random elements  $\{Z^{\epsilon}_{z, n}; \epsilon>0\}$ satisfies a ULDP with rate function given in variational form by 
$$ \inf_{u\in L^2([0,T];H)} \left\{ \frac{1}{2}\int_0^T \int_0^\ell |u(s,y)|^2 dyds: \phi=Z^{u}_{z, n}  \right\},$$
where $Z^u_{z,n}$ is a localization of the skeleton equation. 
Finally, $I^n_{z,T}$ is a good rate function since its sublevel sets $\Phi^{n}_{z,T}(s_0)$  
are compact as a consequence of lower semicontinuity. The proof is complete.
\end{proof}

 \section{Exact controllability and regularity of the quasipotential}\label{sec:Controllability}

 Let $z^* \in \e$ be the asymptotically stable equilibrium from Assumption \ref{Assumption:xstar}. An observation that is important for the subsequent analysis is that such a point enjoys higher regularity. Indeed, $z^*$ satisfies 
\begin{align*}
\begin{cases}
    \Pi_2 z^*& = 0, \\
    \partial^2_x \Pi_1z^* + b(\Pi_1 z^*) &= 0.\end{cases}
\end{align*}
By standard elliptic regularity theory, 
$\Pi_1 z^* \in H^2$ because $b(\Pi_1 z^*) \in L^\infty \subset H$. Hence, $z^* \in \h_2\subset \h_1$.

The goal of this section is to show that we can connect any state $z\in\h_1$ to $z^*$ with a (controlled) path of the local solution of the skeleton equation $Z_{z^*}^{u}:=Z_{z^*}^{0,u}$ \eqref{eq:controlledequation}. Moreover, we show that this path has arbitrarily small energy, provided that $z$ is sufficiently close to $z^*$. In turn, this controllability property implies a continuity property of the quasipotential \eqref{eq:QuasipotentialIntro} which we prove in Lemma \ref{lem:innerregularity} below. 

Fix a non-empty bounded set  $D\subset\e$ with $z^*\in D$, 
and set
\begin{equation}\label{eq:nDdefinition}
    n_D:=\inf\big\{ n\in\N: D\subset \bar{B}_\e(0, n)       \big\}.
\end{equation}

Next, we formally define the notion of two quasipotentials that will be of central interest in the study of exit time and exit place asymptotics in Section \ref{Sec:Metastability}. In this section we study some of their properties.

\begin{dfn}[Quasipotentials $V_D$ and $V_{\bar{D}}$]\label{dfn:quasipotentialVd} Let $T>0, z\in\e$ and fix a set $N\subset\ \bar{D}$. Let $I^{n_{D}}_{z,T}: C([0,T];\e)\rightarrow [0,\infty]$ be the LULDP rate function corresponding to the localized process $\{Z^\epsilon_{z,n_D}\}$ from Proposition \ref{prop:LULDP} (see also  \eqref{eq:LRateFunctionWave}). The quasipotential $V_D$ from $z\in\e$ to $N$ defined with respect to paths that are bound to stay in $D$ is given by 
\begin{equation}
    \label{eq:VDquasipotential}
    V_D(z, N):=\inf\bigg\{ I^{n_D}_{z,T}(\phi)\;;\;T>0, \phi\in C([0,T];\e): \phi(0)=z\;,\;\phi(T)\in N, \forall t\in[0,T)\; \phi(t)\in D\; \bigg\}.
\end{equation}
Analogously, if we require that paths are bound to stay in $\bar{D}$ we define
\begin{equation}
    \label{eq:VDbarquasipotential}
    V_{\bar{D}}(z, N):=\inf\bigg\{ I^{n_D}_{z,T}(\phi)\;;\;T>0, \phi\in C([0,T];\e): \phi(0)=z\;,\;\phi(T)\in N, \forall t\in[0,T)\; \phi(t)\in \bar{D}\; \bigg\}.
\end{equation}
\end{dfn}

For $z'\in\e$ we shall also use the notation $V_D(z, z'):=V_D(z, \{z'\})$ and $V_{\bar{D}}(z, z'):=V_{\bar{D}}(z, \{z'\})$. These functionals are important in the study of exit shape asymptotics because there we need to make sure that paths do not spend time on the exterior of $D$ before hitting $N$; see Remark \ref{rem:VD} for a more detailed discussion.

In the case where paths are allowed to exit $D$ before hitting the target set $N\subset\ \bar{D}$, we drop the subscripts $D$ and $\bar{D}$ from $V_D$ and $V_{\bar{D}}$ respectively and we simply define $V$ via Definition \ref{dfn:quasipotential}.
\begin{dfn}[Quasipotential $V$]\label{dfn:quasipotential} 

Let $T>0, z\in\e$ and fix a set $N\subset\ \bar{D}$. Let $I^{n_{D}}_{z,T}: C([0,T];\e)\rightarrow [0,\infty]$ be the LULDP rate function corresponding to the localized process $\{Z^\epsilon_{z,n_D}\}$ from Proposition \ref{prop:LULDP} (see also  \eqref{eq:LRateFunctionWave}). The quasipotential $V$ from $z\in\e$ to $N$  is given by 
$$ V(z, N):=\inf\big\{ I^{n_D}_{z,T}(\phi); , T>0, \phi\in C([0,T];\e): \phi(0)=z\;,\;  \phi(T)\in N\big\}.       $$    

In general, we may drop the superscript $n_D$ from the rate function  $I_{z,T}: C([0,T];\e)\rightarrow [0,\infty]$  and for a given set $G\subset\e$, we may set for the quasipotential $V$ from $z$ to $G$  
     $$ V(z, G):=\inf\big\{ I_{z,T}(\phi);, T>0, \phi(0)=z\;,\; \phi\in C([0,T];\e):   \phi(T)\in G\big\}.       $$      
     \end{dfn}

In some particular cases, $V_D$ and $V_{\bar{D}}$ (from Definition \ref{dfn:quasipotentialVd}) are related to each other and to the quasipotential $V$ (from Definition \ref{dfn:quasipotential}) as shown in Lemma \ref{lem:VDequality} below.

\begin{lem}\label{lem:VDequality} For any $\varnothing\neq D\subset\e$ and $z\in D$ we have 
$$V_{\bar{D}}(z, \partial D)=V_D(z, \partial D)= V(z, \partial D).$$    
\end{lem}
\begin{proof} The equality $V_{\bar{D}}(z, \partial D)=V_D(z, \partial D)$ is clear, so we focus on showing $V_D(z, \partial D)= V(z, \partial D)$. On the one hand, the inequality $V(z, \partial D)\leq V_D(z, \partial D)$ is trivial since $V_D$ is an infimum over a smaller set than the one used for $V.$ For the reverse inequality, let $\eta>0,$ $T>0, z'\in\partial D, u\in L^2([0,T];H)$ and $\phi=Z^u_z$ such that $Z^u_z(T)=z'$ and 
$$  \frac{1}{2}|u|^2_{L^2([0,T];H)}<V(z, \partial D)+\eta.$$
Next let $\tau=\inf\{ t>0 : \phi(t)\in\partial D \}$ be the first time that $\phi$ hits the boundary and note that $\tau\leq T.$ Since the restriction of $\phi$ on $[0,
\tau]$ is in  $C([0,\tau];\e)$ and $\phi(t)\in D$ for any $t<\tau$ we conclude that 
$$  V_D(z, \partial D)\leq \frac{1}{2}|u|^2_{L^2([0,\tau];H)}<V(z, \partial D)+\eta.$$
Since $\eta$ is arbitrary, the proof is complete.
\end{proof}

Before we proceed to the main body of this section we make the following important remark on notation:

\begin{rem}\label{rem:notation1} Throughout the rest of this section, we shall omit the index $n_D$ \eqref{eq:nDdefinition} from both the rate function $I^{n_D}_{z,T}$ and the local solutions $Z^u_{z, n_D}$ of the skeleton equation. Hence, we shall abuse notation and write $Z^u_{z, n_D}\equiv Z^u_{z}$ and $I^{n_D}_{z,T}\equiv I_{z,T}$ for the sake of lighter notation. This simplified notation will also be used for the nonlinear terms $B, \Sigma.$ In other words, from this point on, we shall write $B, \Sigma$ instead of $B_{n_D}, \Sigma_{n_D}$ \eqref{eq:bnlocalization}.  The localization index $n_D$ is important for the study of the exit problem for local solutions. We defer a more detailed discussion to the introduction of Section \ref{Sec:Metastability}.
\end{rem}

\begin{lem}[Exact controllability from $z^*$ to $H^1 \times H$] Under Assumptions \label{lem:exact-control} \ref{Assumption:b}-\ref{Assumption:sigmaNondeg} the following holds:

  There exists $T_0>0$ such that for any $\eta>0$, there exists $\delta>0$ such that whenever $|z-z^*|_{\h_1 }< \delta$, there exists a control $u \in L^2([0,T_0];H)$ such that
  \begin{equation*}
      \frac{1}{2}\int_0^{T_0} \int_0^\ell |u(s,x)|^2 dxds < \eta
  \end{equation*}
  and
  \begin{equation*}
      Z^{u}_{z^*}(T_0) = z.
  \end{equation*}
\end{lem}

\begin{proof} 
    For $T_0>0$ to be chosen later, define the linear operator $L_{T_0}:L^2([0,T_0];H) \to \h_1$ by
    \begin{equation*}
        L_{T_0} u := \int_0^{T_0} S_{\alpha}(T_0 - s)\begin{pmatrix}0\\ u(s) \end{pmatrix}ds.
    \end{equation*}
    
    This operator maps $L^2([0,T_0]; H)$ \textit{onto} $\h_1$ when $T_0$ is sufficiently large. Indeed, the adjoint operator 
    $L^\star_{T_0}: \h_1\to L^2([0,T];H)$ is given by
    \begin{equation*}
        [L^\star_{T_0} z](t,x) = [\Pi_2 S_\alpha^\star(T_0-t) z](x).
    \end{equation*}
     In view of \cite[Proposition 2.3]{cerrai2006PTRFsmoluchowski1} and Proposition \ref{prop:L2-decay} it follows that the operator $S_\alpha^\star$ is of negative type with the same exponent as $S.$ From direct calculations (see e.g. \cite[Lemma 3.2]{cerrai2014smoluchowski}) we have
    \begin{align*}
        &|L^\star_{T_0} z|_{L^2([0,T_0];H)}^2 \nonumber= \int_0^{T_0} \int_0^\ell |[\Pi_2 S_\alpha^\star(T_0-t) z](x)|^2 dxdt \nonumber= \frac{1}{2 \alpha} \left( |z|_{\h_1}^2 - |S_\alpha^\star(T_0) z|_{\h_1}^2 \right)
        \geq \frac{1}{2 \alpha}(1-M^2e^{-2\theta t})|z|^2_{\h_1}.
    \end{align*}
    Choosing $T_0$ big enough so that \begin{equation*}\label{eq:T0choice}
        M^2e^{-2\theta T_0} = \frac{1}{2}
    \end{equation*}
    yields    
    \begin{equation*}
        |z|^2_{\h_1} \leq 4\alpha |L^\star_{T_0} z|^2_{L^2([0,T_0];H)}=4\alpha |(L_{T_0}L^\star_{T_0})^{1/2}z|^2_{\h_1}.
    \end{equation*}
    By \cite[Corollary B7]{da2014stochastic}, this proves that the image
    \begin{equation*}
        \left\{ z \in H^1 \times H : |z|_{\h_1 }\leq 1\right\} \subset \left\{L_{T_0} u : |u|_{L^2([0,T_0];H)} \leq 2 \sqrt{\alpha} \right\}.
    \end{equation*}
  In particular, this proves that for any $z \in\h_1$, there exists a control $u \in L^2([0,T_0];H)$ such that 
    \begin{equation*}\label{eq:lin-opt-contrl-1}
        z = \int_0^{T_0} S_\alpha(T_0-t)\begin{pmatrix} 0 \\u(t) \end{pmatrix}dt
    \end{equation*}
    and
    \begin{equation*} \label{eq:lin-opt-contrl-2}
        |u|_{L^2([0,T_0];H)} \leq 2 \sqrt{\alpha} |z|_{\h_1}.
    \end{equation*}
   Now let $\delta>0$  to be specified later. Given any $z \in\h_1$ such that
    $|z-z^*|_{\h_1} < \delta$, let $|u_1|_{L^2([0,T_0];H)} \leq 2 \sqrt{\alpha} \delta$ be a control such that
    $L_{T_0} u_1 = z-z^*$. Next, 
    define $Z(t) := z^* + \int_0^t S_\alpha(t-s) \begin{pmatrix} 0 \\u_1(s) \end{pmatrix}ds.$ Since $z^*$ is an equilbrium and $\sigma$ is non-degenerate we can write
    \begin{equation*}
        Z(t) = S_\alpha(t)z^* + \int_0^t S_\alpha(t-s) B(Z(s))ds  + \int_0^t S_\alpha(t-s) \begin{pmatrix} 0 \\\sigma(\Pi_1 Z(s))u(s) \end{pmatrix}ds,
    \end{equation*}
    where 
    \begin{equation*}
        u(s,x) := \frac{1}{\sigma(\Pi_1 Z(s,x))} \bigg(u_1(s,x) + b(\Pi_1 z^*(x)) - b( \Pi_1 Z(s,x)) \bigg),\; (s,x)\in[0,t]\times(0,\ell).
    \end{equation*}
   
    In particular, $Z=Z_{z^*}^{u}$ and moreover $Z(T_0)=z.$ Notice that the path $Z$ exists for all times even though $Z_{z^*}^{u}$ is understood as a local solution.

    It remains to estimate the $L^2$ norm of $u.$ By contractivity of the semigroup in $\h_1$ (Proposition \ref{prop:L2-decay}) and choice of $u_1$ we have
    \begin{equation} \label{eq:Z-bounded}
        \sup_{t \in [0,T_0]} |Z(t) - z^*|_{\h_1 } \leq C  \delta.
    \end{equation}
    The latter, along with Assumption \ref{Assumption:sigmaNondeg} and the choice of $u_1$ yield  
    \begin{equation*}
        \int_0^{T_0} \int_0^\ell |u(s,x)|^2 ds
        \leq C' \sup_{t \in [0,T_0]} |b( \Pi_1 z^*(x)) - b(\Pi_1 Z(s))|^2_{\infty} + C \delta^2.
    \end{equation*}
    By \eqref{eq:Z-bounded}, the Sobolev embedding $ H^1\subset C_0$ and the continuity of $b$, we can take $\delta$ small enough so that the right-hand side is less than $\eta.$ The proof is complete. 
\end{proof}

We are now ready to prove an important continuity property of $V_{\bar{D}}$ that is needed for our proof of exit time lower bounds and exit shape asymptotics; Theorems \ref{thm:ExitTimeLowerBnd}, \ref{thm:exitshapeldp}. This property, also known in the literature as \textit{inner regularity}, is listed as an assumption in \cite[Section 12.5.1]{da2014stochastic} (see in particular the definition of the "lower exit rate" $\underline{e}$ therein) and \cite{salins2019uniform} (see Assumption 8.2). It is proven via controllability theorems in \cite[Section 5.4; see also Theorem 5.7]{salins2021metastability}. Those controllability arguments from the parabolic case do not work in the setting of wave equations. This is mainly attributed to the fact that, due to the lack of smoothing, arbitrary controlled trajectories of the underlying dynamics cannot necessarily be connected. A version of this lemma was proved in \cite[Section 6]{cerrai2016smoluchowski} under much stronger assumptions. In particular, the authors worked in the case of additive noise under the assumption that there exists a unique stable equilibrium and assumed some extra smoothness on $b$ (see Assumption 2 therein). We are interested in the case of multiplicative noise and multiple stable equilibria, so we need such a result in higher generality.

\begin{lem}[Inner regularity of $V_{\bar{D}}$]\label{lem:innerregularity} Let $D\subset\e$ be a bounded set that contains the asymptotically stable equilibrium $z^*$ and $V_{\bar{D}}$ as in \eqref{eq:VDbarquasipotential}. For any closed $N \subset \partial D,$ let $B_\e(N, \delta)=\{ z\in\e: \textnormal{dist}_\e (z, N)<\delta\}$
    and 
    \begin{equation*}
        V_{\bar{D}}\big(B_\e(z^*,\rho), B_\e(N,\delta)\big): = \inf\bigg\{V_{\bar{D}}(z_1,z_2): z_1 \in B(z^*,\rho), z_2 \in B(N,\delta) \bigg\}.
    \end{equation*}
    If $V_{\bar{D}}(z^*, N)<\infty$ then, under Assumptions \ref{Assumption:b}- \ref{Assumption:sigmaNondeg}
    \begin{equation*}\label{eq:innerregularity}
        \lim_{(\rho, \delta)\rightarrow (0,0)} V_{\bar{D}}\big(B_\e(z^*,\rho), B_\e(N,\delta)\big) = V_{\bar{D}}(z^*,N).
    \end{equation*}
\end{lem}

\begin{proof}
Assume by contradiction that 
\[\lim_{(\rho, \delta)\rightarrow (0,0)} V_{\bar{D}}\big(B_\e(z^*,\rho), B_\e(N,\delta)\big) < V_{\bar{D}}(z^*,N).\]
This means that there exist a sequence of initial data $\{z_n;n\in\N\}\subset \e$ such that $\lim_n|z_n - z^*|_\e =0,$ a sequence of time horizons $T_n>0 $, a sequence of controls  $u_n \in L^2([0,T_n]; H)$, and a distance $\eta>0$ such that
\begin{equation}\label{eq:VDcontradiction}
    \frac{1}{2}|u_n|_{L^2([0,T_n]; H)}^2 \leq V_{\bar{D}}(z^*,N) - \eta,
\end{equation}
such that the controlled processes $Z^{u_n}_{z_n}(t)$ have the properties that $ Z^{u_n}_{z_n}(t) \in \bar{D}$  for all $t \in [0,T_n)$
and   $\textnormal{dist}(Z^{u_n}(T_n) ,N) \to 0.$ Since $Z^{u_n}_{z_n}(t) \in \bar{D}$ for all $t,$ there exists $N_0=N_0(D)\in\N$ (independent of $n$) such that the local $Z^{u_n}_{z_n}$ coincides with the process $Z^{u_n}_{z_n,N_0}$ (see Definition \ref{dfn:local mild solutions}). Hence, without loss of generality, we shall assume that $b, \sigma$ are bounded, Lipschitz continuous. Moreover, we drop the subscript $N$ and write $Z^{u_n}_{z_n}$ instead of $Z^{u_n}_{z_n,N_0}$ for the sake of lighter notation.

We separate the analysis into two cases; first where there exists a bounded subsequence of time horizons $\{T_n; n\in\N\}$ and second where $T_n \uparrow +\infty$.

\noindent \underline{\textbf{Case 1}}: (A subsequence of $\{T_n; n\in\N\}$ is upper bounded). Notice that \eqref{eq:VDcontradiction} provides a uniform $L^2$ bound for the controls, hence the sequence $\{u_n; n\in\N\}$ is weakly pre-compact. Thus, there exists a further subsequence (relabeled by $n$) such that $T_n \to T$ and $u_n$ converges weakly in $L^2([0,T];H)$ to $u$. Then by Lemma \ref{lem:controliccontinuity}(1), $$
    |Z^{u_n}_{z_n}(T_n) - Z^{u}_{z^*}(T)|_\e\rightarrow 0.$$
 Because $\textnormal{dist}(Z^{u_n}_{z_n}(T_n) ,N) \to 0$ and $N$ is closed this proves that $Z^{u}_{z^*}(T) \in N.$ But
\[    V_{\bar{D}}(z^*,N) \leq \frac{1}{2}|u|_{L^2([0,T];H)}^2 \leq V_{\bar{D}}(z^*,N) - \eta\]
which is a contradiction.

\noindent \underline{\textbf{Case 2}}: ($T_n\uparrow \infty$). We want to show that we also have a kind of convergence in this case. To guarantee convergence, we translate these processes to the negative half line $(-\infty,0]$. Define for $t<0$
\begin{equation*}
    \varphi_n(t) = \begin{cases}
        Z^{u_n}_{z_n}(T_n + t) & \text{ for } t \in [-T_n,0] \\
        z_n & \text{ for } t < -T_n
    \end{cases}
\end{equation*}
and
\begin{equation}\label{eq:hnDefinition}
    h_n(t) = 
    \begin{cases} 
        u_n(t + T_n) &\text{ for } t \in [-T_n,0], \\
        0 &\text{ for } t<-T_n.
    \end{cases}
\end{equation}
For $t \in [-T_n,0]$, $\varphi_n(t)$ solves
\begin{align*}
    \varphi_n(t) = &S_\alpha(T_n + t)z_n + \int_{-T_n}^t S_\alpha(t-s) \begin{pmatrix}
        0 \\ b(\Pi_1 \varphi_n(s))
    \end{pmatrix}
    ds \nonumber+ \int_{-T_n}^t S_\alpha(t-s) \begin{pmatrix}  0 \\ \sigma(\Pi_1 \varphi_n(s))h_n(s) \end{pmatrix}ds \nonumber\\
    = &S_\alpha(T_n + t)(z_n-z^*) + z^* + \int_{-T_n}^t S_\alpha(t-s) \begin{pmatrix}
        0 \\ (b(\Pi_1 \varphi_n(s)) - b(\Pi_1 z^*))
    \end{pmatrix}
    ds \nonumber\\
    &+ \int_{-T_n}^t S_\alpha(t-s) \begin{pmatrix}  0 \\ \sigma(\Pi_1 \varphi_n(s))h_n(s) \end{pmatrix} ds\nonumber\\
    =:& z^* +  S_\alpha(T_n + t) (z_n - z^*) + I_n(t) + Y_n(t).
\end{align*}
The second to last line is a consequence of the assumption that $z^*$ is invariant,  implying that
\[z^* = S_\alpha(t)z^* + \int_0^t S_\alpha(t-s) \begin{pmatrix} 0 \\ b(\Pi_1 z^*) \end{pmatrix} ds.\]
We shall proceed to the asymptotic analysis of the paths $\{\phi_n(t);t\in[-\infty, 0]\}$.

\noindent{\textit{Step 1}: \textit{Pre-compactness in} $C([-T,0];\e)$}. We shall use a functional analytic version of Arzel\`a-Ascoli to prove that a subsequence $I_n(t) + Y_n(t)$ converges in $\e$ for all $t \in (-\infty,0]$.
The fact that $\varphi_n(t) \in \bar{D}$ for all $t <0$ guarantees that $|b(\Pi_1 \varphi(s)) - b(\Pi_1z^*)|_H$ is uniformly bounded in $s<0$. In particular, from Proposition \ref{prop:L2-decay} there exist $M,\theta>0$ such that
\begin{equation*}
    |I_n(t)|_{\h_1} \leq  M\int_{-T_n}^t   e^{-\theta(t-s)} ds \leq \frac{M}{\theta}.
\end{equation*}
Because $\h_1$ is compactly embedded into $\e$ (as was shown e.g. in the proof of Lemma \ref{lem:Zaprioricompactnesslem}), this proves that for fixed $t<0$, the family $\{I_n(t)\}_n$ is compact.
As for the control term, \eqref{eq:VDcontradiction} furnishes
\begin{equation}\label{eq:hnbound}
    |Y_n(t)|_{\h_1} \leq C |h_n|_{L^2([-T_n,0];H)} \leq C \sqrt{V_{\bar{D}}(z^*,N) - \eta}.
\end{equation}
We turn to temporal equicontinuity estimates for $I_n, Y_n.$
A direct consequence of the compact inclusion $\h_1\subset\e$ is that
\begin{equation*}
    \lim_{t \rightarrow 0}|S_\alpha(t) - I|_{\mathscr{L}(\h_1; \e)} = 0.
\end{equation*}

Now for $s<t<0$,
\begin{align*}
    &\left|I_n(t) - I_n(s) \right|_\e \nonumber\\
    &\leq \int_{-T_n}^s \left| (S_\alpha(t-r)-S_\alpha(s-r)) 
    \begin{pmatrix} 0 \\ (b(\Pi_1 \varphi_n(r)) - b(\Pi_1 z^*)) \end{pmatrix}\right|_\e dr \nonumber\\
    &\quad+ \int_{s}^t \left|S_\alpha(t-r) \begin{pmatrix} 0 \\ (b(\Pi_1 \varphi_n(r)) - b(\Pi_1 z^*)) \end{pmatrix}\right|_\e dr\nonumber\\
    &\leq \int_{-T_n}^s M e^{-\theta(s-r)} |S_\alpha(t-s) - I|_{\mathscr{L}(\h_1;\e)} | (b(\Pi_1 \varphi_n(r)) - b(\Pi_1 z^*))|_H dr\nonumber\\
    &\quad+ \int_s^t M e^{-\theta (t-r)} | (b(\Pi_1 \varphi_n(r)) - b(\Pi_1 z^*))|_H dr \nonumber\\
    &\leq C |S_\alpha(t-s) - I|_{\mathscr{L}(\h_1; \e)} + C(t-s),
\end{align*}
which proves equicontinuity for $I_n$. The constant $C$ depends on $M$, $\theta$ and $\sup_{z \in \bar{D}} |b(z) - b(z^*)|_H.$

Equicontinuity for the $Y_n(t)$ terms is similar. For $s<t<0$,
\begin{align*}
    &\left|Y_n(t) - Y_n(s) \right|_\e \nonumber\\
    &\leq \int_{-T_n}^s \left| (S_\alpha(t-r)-S_\alpha(s-r)) 
    \begin{pmatrix} 0 \\ \sigma(\Pi_1 \varphi_n(r))h_n(r) \end{pmatrix}\right|_\e dr \nonumber\\
    &\quad+ \int_{s}^t \left|S_\alpha(t-r) \begin{pmatrix} 0 \\ \sigma(\Pi_1 \varphi_n(r))h_n(r) \end{pmatrix}\right|_\e dr\nonumber\\
    &\leq \int_{-T_n}^s M e^{-\theta(s-r)} |S_\alpha(t-s) - I|_{\mathscr{L}(H^1\times H, \e)} | \sigma(\Pi_1 \varphi_n(r))h_n(r)|_H dr\nonumber\\
    &\quad+ \int_s^t M e^{-\theta (t-r)} | \sigma(\Pi_1 \varphi_n(r))h_n(r)|_H dr.
\end{align*}
Now we use the fact that $\sup_{z \in \bar{D}} |\sigma(\Pi_1 z)|_{\infty}<+\infty$ and a Cauchy-Schwarz inequality in time to bound the above expression by
\begin{align*}
    \left|Y_n(t) - Y_n(s) \right|_\e \nonumber
    &\leq |S_\alpha(t-s) - I|_{\mathscr{L}(\h_1; \e)} C | h_n|_{L^2([-T_n,0];H)} +  \sqrt{t-s} | h_n|_{L^2([-T_n,0];H)}\\&
    \leq C \sqrt{V_{\bar{D}}(z^*,N) - \eta}
    \bigg(|S_\alpha(t-s) - I|_{\mathscr{L}(\h_1; \e)}+   \sqrt{t-s} \bigg),
\end{align*}
where the last line follows from \eqref{eq:hnbound}.
This proves equicontinuity for  the $Y_n$ terms.

\noindent{\textit{Step 2}: \textit{Identification of limiting paths}}. By the Arzel\`a-Ascoli Theorem (see e.g. \cite[Theorem 47.1]{Munkres}, there exists a subsequence (relabeled as $Y_n$, $I_n$) and limits $I(t)$ and $Y(t)$ such that for any $-T<0$
\begin{equation*}
    \lim_{n \to \infty} \sup_{t \in [-T,0]} |I_n(t) - I(t)|_\e =\lim_{n \to \infty} \sup_{t \in [-T,0]} |Y_n(t) - Y(t)|_\e=0,
\end{equation*}
i.e. we have uniform convergence on compact time intervals.

Because $T_n \uparrow \infty$ and $|z_n - z^*|_\e \to 0$, it follows that
\[\lim_{n \to \infty}\sup_{t\in [-T_n,0]} |S_\alpha(t + T_n)(z_n - z^*)|_\e = 0.\]
Therefore, a subsequence of  $\varphi_n(t)$ converges  to a limit $\varphi(t)$ uniformly on compact time intervals. In order to charactherize $\varphi$ we use the dominated convergence theorem and uniqueness of the limit to conclude that
\begin{equation*}
    I_n(t) \longrightarrow \int_{-\infty}^t S_\alpha(t-s) \begin{pmatrix}
        0 \\ b(\Pi_2 \varphi(s)) - b(z^*))
    \end{pmatrix}ds=:I(t),\;\;n\to\infty.
\end{equation*}
In view of \eqref{eq:VDcontradiction}, \eqref{eq:hnDefinition}, the sequence $\{h_n\mathds{1}_{[-T_n,t]};n\in\N\}\subset L^2((-\infty, 0);H)$ is weakly pre-compact. Passing if necessary to a further subsequence so that  $h_n\mathds{1}_{[-T_n,t]} \rightarrow h\mathds{1}_{[-\infty,t]}$ weakly, it follows that 
\begin{equation*}
    Y_n(t) \longrightarrow \int_{-\infty}^t S_\alpha(t-s) \begin{pmatrix}
        0 \\ \sigma( \Pi_1 \varphi(s))h(s)
    \end{pmatrix}=:Y(t),\;\;n\to\infty.
\end{equation*}

Hence, for $t<0$ the limiting path $\varphi$ satisfies
\begin{align}\label{eq:varphi-infinite-integrals}
    \varphi(t) = &z^* + \int_{-\infty}^t S_\alpha(t-s) \begin{pmatrix} 0 \\ (b(\Pi_1\varphi(s)) - b(\Pi_1 z^*)) \end{pmatrix} ds \nonumber\\&+ \int_{-\infty}^t S_\alpha(t-s) \begin{pmatrix} 0 \\ \sigma(\Pi_1 \varphi(s))h(s) \end{pmatrix} ds.
\end{align}

\noindent{\textit{Step 3}: \textit{Long-time behaviour of the limiting path in $\e$}}. The next step is to prove that $\lim_{t \downarrow -\infty} |\varphi(t)-z^*|_\e = 0$.
For any $T<t<0$, a representation for $\varphi(t)$ is
\begin{align*}
    \varphi(t) = &S_\alpha(t-T)\varphi(T)  + \int_T^t S_\alpha(t-s)\begin{pmatrix} 0 \\ (b(\Pi_1\varphi(s)) ) \end{pmatrix} ds \nonumber+ \int_{T}^t S_\alpha(t-s) \begin{pmatrix} 0 \\ \sigma(\Pi_1 \varphi(s))h(s) \end{pmatrix} ds.
\end{align*}
By the proof of Lemma \ref{lem:controliccontinuity}(2), given $\rho>0$, there exists $T_{\rho}>0$ and $\alpha_\rho>0$ such that 
\begin{equation*}
    |Z^{u}_z(T_\rho) -z^*|_\e < \rho
\end{equation*}
if $z \in D$ and $|u|_{L^2([0,T_
\rho];H)}\leq \alpha_\rho$.
Because $h \in L^2((-\infty,0);H)$, there exists $-T'_\rho<0$ such that 
\begin{equation*}
    |h|_{L^2((-\infty,-T'_\rho);H)}< \alpha_\rho.
\end{equation*}
Then for any $t<-T'_\rho$, $\varphi(t)$ can be written as
\begin{align*}
    \varphi(t) = &S_\alpha(T'_\rho) \varphi(t-T'_\rho)
    + \int_{t-T'_\rho}^t S_\alpha(t-s)\begin{pmatrix} 0 \\ (b(\Pi_1\varphi(s)) ) \end{pmatrix} ds \nonumber\\
    &+ \int_{t-T'_\rho}^t S_\alpha(t-s) \begin{pmatrix} 0 \\ \sigma(\Pi_1 \varphi(s))h(s) \end{pmatrix} ds.
\end{align*}
Therefore, because $|h|_{L^2((t-T'_\rho,t);H)}< \alpha_\rho$, it follow that $|\varphi(t)-z^*|_\e< \rho$. This is true for all $t<-T'_\rho$. We can do this argument for arbitrarily small $\rho>0$, proving that
\begin{equation*}
    \lim_{t \downarrow -\infty} |\varphi(t) - z^*|_\e = 0.
\end{equation*}

\noindent{\textit{Step 4}: \textit{Long-time behaviour of the limiting path in $\h_1.$}} Next, we argue that, because we know that $|\varphi(t) - z^*|_\e \to 0$ as $t \downarrow -\infty$, the stronger $\h_1$ convergence holds
\begin{equation*}
    \lim_{t \downarrow -\infty} |\varphi(t) - z^*|_{\h_1}  =0.
\end{equation*}
For this we use the representation \eqref{eq:varphi-infinite-integrals}.  From Proposition \ref{prop:L2-decay}, for any $t<0$,
\begin{align*}
    &\left|\int_{-\infty}^t S_\alpha(t-s) \begin{pmatrix} 0 \\ (b(\Pi_1\varphi(s))- b(\Pi_1 z^*)) \end{pmatrix}ds \right|_{\h_1}\nonumber\leq \frac{M}{\theta} \sup_{s \leq t} |(b(\Pi_1\varphi(s))- b(\Pi_1 z^*))|_H.
\end{align*}
This converges to 0 as $t \downarrow -\infty$ because $\varphi(s) \to z^*$ in $\e$.

For the second term,
\begin{align*}
    &\left|\int_{-\infty}^t S_\alpha(t-s) \begin{pmatrix} 0 \\ \sigma(\varphi(s))h(s) \end{pmatrix}ds \right|_{\h_1}
    \nonumber\leq \int_{-\infty}^t Me^{-\omega(t-s)} |\sigma(\varphi(s))h(s)|_Hds \nonumber\leq \frac{|\sigma|_\infty M}{\sqrt{2\omega}} |h|_{L^2((-\infty,t);H)},
\end{align*}
where the final inequality is the Cauchy-Schwarz inequality with respect to the time variable. This expression converges to $0$ as $t \downarrow -\infty$ because $h \in L^2((-\infty,0);H)$.

\noindent{\textit{Step 4}: \textit{Exact controllability}}. Now that $\varphi(t) \to z^*$ in $\h_1$, we can use exact controllability Lemma \ref{lem:exact-control}.
Given $\eta>0$, there exists $T_0>0$ and $\delta>0$ such that when $|z-z^*|_{\h_1}< \delta_0$ there exists a controlled path that connects $z^*$ to $z$ with action less than $\frac{\eta}{2}$. Because $\varphi(t) \to z^*$ in $\h_1$, we can find $-t_0<0$ such that $|\varphi(-t_0) - z^*|_{\h_1}<\delta_0$. Therefore, we can find a controlled path $Z^{u_0}_{z^*}$ such that
\begin{equation*}
    Z^{u_0}_{z^*}(T_0) = \varphi(-t_0)
\end{equation*}
and
\begin{equation*}
    \frac{1}{2}|u_0|_{L^2([0,T_0];H)}^2 < \frac{\eta}{2}.
\end{equation*}
Finally, we build a path by concatenation for $t>0$
\begin{equation*}
    \varphi_1(t) = 
    \begin{cases}
        Z^{u_0} (t), & \text{ for } t \in [0,T_0],\\
        \varphi(-t_0 + t - T_0), & \text{ for } t \in [T_0, T_0  + t_0].
    \end{cases}
\end{equation*}
We have that $\varphi_1 = Z^{u}_{z^*}$ is a path controlled by
\begin{equation*}
    u(t) = 
    \begin{cases}
        u_0(t), & \text{ for } t \in [0,T_0],\\
        h(-t_0 + t - T_0),  &\text{ for } t \in [T_0, T_0 + t_0]
    \end{cases}
\end{equation*}
 such that
\begin{equation*}
    \frac{1}{2}|u|_{L^2([0,T_0 + t_0];H)}^2 \leq \frac{1}{2}|h|_{L^2((-\infty,0);H)}^2 + \frac{\eta}{2} \leq V_{\bar{D}}(z^*, N) - \frac{\eta}{2}
\end{equation*}
and $Z^{u}_{z^*}(T_0 + t_0) \in N.$
This is a contradiction. Therefore, in this case we must have
\begin{equation*}
    \lim_{(\rho, \delta)\rightarrow (0,0)} V_{\bar{D}}(B(z^*,\rho),N) = V_{\bar{D}}(z^*, N).
\end{equation*}
The proof is complete.
\end{proof}

 \section{The problem of exit from a domain of attraction}\label{Sec:Metastability} Throughout this section, we fix  a bounded subset $D\subset\e$ of the phase space that contains the asymptotically stable equilibrium $z^*.$ We are concerned with the problem of exit of local solutions $Z^{\epsilon}=(u^\epsilon, \partial u^\epsilon)$ to \eqref{eq:model} from $D.$  In particular, we shall prove logarithmic asymptotics for the distribution of \textit{exit times} 
\begin{equation}\label{eq:exittimes}
     \tau_D^{\epsilon, z}:=\inf\big\{ t>0: Z^\epsilon_z(t)\notin D     \big\}
 \end{equation}
 and \textit{exit shapes} (or exit places) 
 \begin{equation}\label{eq:exitshape}
     Z_z^\epsilon\big( \tau_D^{\epsilon, z}  \big)
 \end{equation}
 as $\epsilon\to 0.$ We only work with local mild solutions (Definition \ref{dfn:local mild solutions}) of the stochastic wave equations \eqref{eq:model}. As we have  explained above, this definition heavily relies on the supremum-norm topology that we are using and, since $D$ is bounded, leads to the following important observation.  With $n_D=\inf\{n\in\N: D\subset\bar{B}_{\e}(0,n)\},$ as in \eqref{eq:nDdefinition}, we have
 \begin{equation*}
     \tau_D^{\epsilon, z}:=\inf\big\{ t>0: Z^\epsilon_{z, n_D}(t)\notin D     \big\}
 \end{equation*}
 and
\begin{equation*}
     Z_{z}^\epsilon\big( \tau_D^{\epsilon, z}  \big)= Z_{z, n_D}^\epsilon\big( \tau_D^{\epsilon, z}  \big).
 \end{equation*}
 
 In other words, to study exits of our local mild solutions from the bounded domain $D\subset\e,$  it suffices to consider the localized process $Z_{z, n_D}^\epsilon$ which evolves according to a wave equation with globally Lipschitz coefficients.

 Local mild solutions are well-defined in the topology of $\e$  and satisfy LULDPs over bounded sets of initial data (Proposition \ref{prop:LULDP}). Moreover, the aforementioned asymptotics are meaningful for local solutions since exit times from $D$ occur before explosion times, i.e. for each $\epsilon>0, z\in D$ we have
 $$ \tau_D^{\epsilon, z}<\tau^\epsilon_{z,\infty}.$$
 In fact, our exit time analysis yields novel asymptotic lower bounds for explosion times that are of order $\exp(1/\epsilon^2)$ (a more detailed discussion is deferred to Remark \ref{rem:explosiontimes}). 
 
 We remind to the reader that Remark \ref{rem:notation1} of Section \ref{sec:Controllability} is in place. In particular, for the duration of this section, even though we work with the localized solution $Z_{z,n_{D}}^\epsilon$, we omit the subscript $n_D$ and we simply write $Z_z^\epsilon$. The same index omission is in place for the level sets $\Phi^{n_D}_{T,z}$ which will be simply denoted by $\Phi_{T,z}.$

 Intuitively, it is clear that, as $\epsilon\to 0,$ $\tau_D^{\epsilon, z}$ diverges to $+\infty$  since the deterministic dynamics $Z_z^0$ are attracted to $z^*$ and never exit $D.$  The exponential growth rate of exit times and concentration of the exit place distribution are characterized by the quasipotentials, $V_D$, $V_{\bar{D}}$ and $V$ whose definition are given in Definitions  
  \ref{dfn:quasipotentialVd} and \ref{dfn:quasipotential} respectively. In particular, the quasipotentials $V_D$ and $V_{\bar{D}}$ are important for the study of exit shape asymptotics (Section \ref{Sec:ExitShapes}).

  Starting from the classical work of Freidlin and Wentzell \cite[Chapters 4, 6.5]{freidlin1998random}, it is typically assumed that the domain $D$ is open, bounded, uniformly attractive and invariant for the deterministic dynamics $Z^{0}_z.$ In Proposition \ref{corr:DExistence}, we proved the existence of such a domain. Nevertheless, we choose to present our main results in a way that clarifies precisely what conditions on $D$ are required for each of the  asymptotic bounds to hold. 

The rest of this section is organized as follows: In Sections \ref{Sec:ExitTimesUB}, \ref{Sec:ExitTimesLB} we present our analysis for the exit time upper and lower bounds, Theorems \ref{thm:ExitTimeUpperBnd}, \ref{thm:ExitTimeLowerBnd} respectively. Then, Section \ref{Sec:ExitShapes} is devoted to the typical behavior and large deviation asymptotics of exit shapes, Theorem \ref{thm:exitshapeldp}. In Section \ref{Sec:BoundaryPoints} we investigate conditions under which our exit time and exit shape upper and lower bounds are equal. This question is closely related to a notion of "regular" boundary points that we introduce in Definition \ref{def:regularpoints}. In the same section, we illustrate this definition by providing examples of both regular and irregular boundary points for two concrete domains of attraction. Throughout this section we work with local solutions of \eqref{eq:model} (Definition \ref{dfn:local mild solutions}). In Section \ref{sec:GlobalSolutions}, we briefly explain how to adapt all the results of this section in the setting of global solutions.

\subsection{Logarithmic exit time upper bound}\label{Sec:ExitTimesUB} Let us recall here the notation $\bar{D}^c:=\e\setminus\bar{D}$. The main result of this section is the following:

\begin{thm}[Exit time upper bound]\label{thm:ExitTimeUpperBnd} Let $D\subset\e$ be a bounded, uniformly attracting set that contains the stable equilibrium $z^*$. Under Assumptions \ref{Assumption:b}-\ref{Assumption:xstar} and for any $z\in D$ we have
\begin{enumerate}
    \item $$\limsup_{\epsilon\to 0}\epsilon^2\log\ex\big[ \tau_D^{\epsilon,z}\big]\leq V(z^*, \bar{D}^c).$$
    \item For any $\delta>0$ $$\lim_{\epsilon\to 0}\pr\bigg[ \epsilon^2\log\tau_D^{\epsilon,z}>V(z^*, \bar{D}^c)+\delta \bigg]=0.$$
\end{enumerate}
\end{thm}
\noindent For the proof we shall need the following auxiliary lemma:  
\begin{lem}\label{lem:preupperbound} With $D\subset\e$ and the same assumptions as in Theorem \ref{thm:ExitTimeUpperBnd} we have: For any $\delta>0$ there exists $T>0$ such that 
\begin{equation*}
    \liminf_{\epsilon\to 0}\inf_{z\in D}\epsilon^2\log\pr\big[ \tau_D^{\epsilon, z}\leq T  \big]>-V(z^*,  \bar{D}^c)-\delta.
\end{equation*} 
\end{lem}

\begin{proof} Without loss of generality we assume that $V(z^*,  \bar{D}^c)<\infty$ (otherwise there is nothing to prove). By definition of the quasipotential and rate function there exists $T_1>0$, $y \in \bar{D}^c$ and a control $v\in L^2([0,T_1];H)$ such that  $Z_{z^*}^v(T_1)=y$ and $\tfrac{1}{2}|v|^2_{L^2([0,T_1];H)}<V(z^*, \bar{D}^c)+\delta.$ Letting $d:=\textnormal{dist}_{\e}(y, \partial D)$ and using the flow continuity of the skeleton equation (Lemma \ref{lem:controliccontinuity}(1)), there exists $\rho>0$ such that for all initial data $z\in\e$ with $|z-z^*|_\e<\rho$ we have $|Z_{z}^v-Z_{z^*}^v   |_{C([0,T_1];\e)}<d/2.$ By uniform attraction to $x^*$ we can find a time $T_2>0$ such that $\sup_{z\in D}|Z_{z}^v(T_2)-z^*    |_{\e}<\rho.$ Next we define a control $u$ by
$$  u(t)=\begin{cases} 0,&\;\; t\in [0, T_2]\\
                       v(t-T_2),&\;\; t\in [T_2, T_1+T_2].   
\end{cases}     $$
Letting $T:=T_1+T_2$ we see that $$\frac{1}{2}|u|^2_{L^2([0,T];H)}=\frac{1}{2}|v|^2_{L^2([0,T_1];H)}<V(z^*, \bar{D}^c)+\delta.$$ Moreover, a direct application of the reverse triangle inequality yields 
$$ \textnormal{dist}_{\e}(Z^u_{z}(T), \partial D)\geq  \textnormal{dist}_{\e}(Z^u_{z^*}(T), \partial D)-|Z_{z}^v-Z_{z^*}^v   |_{C([0,T_1];\e)}>d-d/2=d/2.      $$

The latter, along with another triangle inequality, furnishes the inclusion $$\big\{ |Z_{z}^{\epsilon}-Z_{z}^u   |_{C([0,T];\e)}<d/4   \big\}\subset \big\{\textnormal{dist}_{\e}(Z_z^{\epsilon}(T), \partial D)>d/4\big\}.$$ Hence 
$$\pr\bigg[  |Z_{z}^{\epsilon}-Z_{z}^{u}   |_{C([0,T];\e)}<d/4    \bigg]\leq \pr[  \tau_D^{\epsilon, z}\leq T   ].   $$
Since $Z_z^u$ has finite energy and remains bounded in $\e$ uniformly over $t\in[0,T],$ there exists $s_0>0$ such that $Z_z^u\in\Phi_{z,T}(s_0)$ (with the sublevel set notation introduced in \eqref{eq:SublevelSets}).
From the ULDP lower bound \eqref{eq:LULDP-lower}, we conclude that
\begin{equation*}
    \begin{aligned}
       \liminf_{\epsilon\to 0}\inf_{z\in D}\epsilon^2\log\pr\big[ \tau_D^{\epsilon, z}\leq T  \big]&\geq-\frac{1}{2}|u|^2_{L^2([0,T];H)} >-V(z^*, \bar{D}^c)-\delta.  
\end{aligned}
\end{equation*}\end{proof}

\noindent With Lemma \ref{lem:preupperbound} at hand, Theorem \ref{thm:ExitTimeUpperBnd} follows from the arguments in the proof of \cite[Theorem 5.7.11(a)]{dembo2009large}. We include the proof below for completeness and readers' convenience.

\begin{proof}{\textit{(Of Theorem \ref{thm:ExitTimeUpperBnd})}}

\begin{enumerate}
    \item Let $\delta>0$ and take $T>0$ from Lemma \ref{lem:preupperbound}.
By virtue of the Markov property of $Z^\epsilon_z$, for any $k \in \mathbb{N},$ we have
\begin{equation*}
  \sup_{z \in D} \pr(\tau^{\epsilon,z}_D > kT) \leq \left(\sup_{z \in D} \pr(\tau^{\epsilon,z}_D > T) \right)^k \leq \left( 
 1 - \inf_{z \in D} \pr(\tau^{\epsilon,z}_D \leq T)\right)^k.
\end{equation*}

Thus, by the tail probability formula
$$\ex[ \tau_D^{\epsilon,z}]\leq T\bigg(      1+   \sum_{k=1}^{\infty}\pr\big( \tau_D^{\epsilon,z}\geq kT  \big)\bigg).    $$
From the combination of the last two displays, along with the geometric series formula, we have 
$$ \sup_{z \in D}\ex[ \tau_D^{\epsilon,z}]\leq\frac{T}{1-\left(1 - \inf_{z \in D} \pr(\tau^{\epsilon,z}_D \leq T) \right)} = \frac{T}{ \inf_{z \in D} \pr(\tau^{\epsilon,z}_D \leq T)}.$$

From Lemma \ref{lem:preupperbound} it follows that 
$$   \limsup_{\epsilon\to 0} \sup_{z \in D}\epsilon^2\log\ex\big[ \tau_D^{\epsilon,z}\big]\leq    -\liminf_{\epsilon \to 0} \inf_{z \in D} \epsilon^2\log\pr( \tau_D^{\epsilon,z}\leq T)\leq  V(z^*, \bar{D}^c)+\delta.  $$
Since $\delta>0$ was arbitrary, this concludes the proof of the first assertion. 
\item We have by Chebyshev's inequality
\begin{equation*}
    \begin{aligned}
        \pr\bigg[ \epsilon^2\log\tau_D^{\epsilon,z}>V_D(z^*, \bar{D}^c)+\delta \bigg]&=  \pr\bigg[ \tau_D^{\epsilon,z}>\exp\bigg(\frac{V(z^*, \bar{D}^c)+\delta}{\epsilon^2}\bigg)                              \bigg]\\&\leq \ex\big[ \tau_D^{\epsilon,z}\big]\exp\bigg(-\frac{V(z^*, \bar{D}^c)+\delta}{\epsilon^2}\bigg).  
    \end{aligned}
\end{equation*}  
By (1), the right-hand side converges to $0$ as $\epsilon\to 0$ and the proof is complete.\end{enumerate}\end{proof}

\subsection{Logarithmic exit time lower bound}\label{Sec:ExitTimesLB}We turn to the proof of lower bounds for the exit times \eqref{eq:exittimes}. To do so, we assume that the interior of $D$ contains a ball around the equilibrium $z^*.$ 

\begin{customthm}{5}\label{Assumption:Dinterior}  There exists $R>0$ such that 
 $B_\e(z^*, R)\subset \textnormal{int}(D).$
     \end{customthm}
It is clear that this is satisfied e.g. if $D$ is an open subset of $\e;$ in particular, both an open ball around $z^*$ and the set $D$ that we identified in Proposition \ref{corr:DExistence} satisfy this assumption. In this section we shall prove the following:

\begin{thm}[Exit time lower bound]\label{thm:ExitTimeLowerBnd} Let $D\subset\e$ be a bounded, invariant, uniformly attractive set that contains the asymptotically stable equilibrium $z^*.$ Under Assumptions \ref{Assumption:b}- \ref{Assumption:Dinterior} the following hold:
\begin{enumerate} 
\item  $$\lim_{\epsilon\to 0}\pr\bigg[ \epsilon^2\log\tau_D^{\epsilon,z}\leq V(z^*, \partial D)-\delta \bigg]=0.$$
    \item $$\liminf_{\epsilon\to 0}\epsilon^2\log\ex\big[ \tau_D^{\epsilon,z}\big]\geq V(z^*, \partial D).$$ 
\end{enumerate}
\end{thm}

\begin{rem}\label{rem:exittimematching} Notice that, a priori, the asymptotic upper and lower bounds from Theorems \ref{thm:ExitTimeUpperBnd}(1) \ref{thm:ExitTimeLowerBnd}(2) do not match. This will only be the case for domains $D$ that satisfy $V(z^*, \partial D)=V(z^*, \bar{D}^c).$ Such domains also play an important role for our exit shape asymptotics (see Remark \ref{rem:exitshapeldp}). Necessary conditions on $D$ will be discussed in Section \ref{Sec:BoundaryPoints}.    
\end{rem}

\begin{rem}[Explosion time lower bounds]\label{rem:explosiontimes}
An immediate corollary of Theorem \ref{thm:ExitTimeLowerBnd} is that, when perturbed by noise of sufficiently small intensity, local solutions to \eqref{eq:model} exist, on average, for exponentially long time periods before exploding. Indeed, since for each $\epsilon>0, z\in D$ it holds that $\tau_D^{\epsilon, z}<\tau^\epsilon_{z,\infty},$ we can use  Theorem \ref{thm:ExitTimeLowerBnd}(2) to derive the asymptotic lower bound\eqref{eq:explosiontime}
    \begin{equation*}
        \liminf_{\epsilon\to 0}\epsilon^2\log\ex[\tau^\epsilon_{z,\infty}]\geq \liminf_{\epsilon\to 0}\epsilon^2\log\ex[\tau^{\epsilon,z}_{D}]= V(z^*,\partial D).
    \end{equation*}
    Delayed blowup for solutions of reaction-diffusion equations perturbed by small transport noise has been recently observed in \cite{agresti2024delayed}. Roughly speaking, the lower bound above provides a quantitative description of the same phenomenon in the setting of wave equations.
\end{rem}

Theorem \ref{thm:ExitTimeLowerBnd} relies on a number of preliminary lemmas given below. Conditionally on these lemmas, the proof is identical to that of \cite[Theorem 5.7.11(a)]{dembo2009large}. For completeness and readers' convenience we present it at the end of this subsection.

Let us now fix some notation: 

From Assumption \ref{Assumption:Dinterior}, we can find $\rho>0$ small enough such that 

\begin{equation}\label{eq:gammarhodef}
    \gamma_\rho:= \bar{B}_\e(z^*, \rho)\subset D\;,\gamma_\rho\cap\partial D=\varnothing.
\end{equation} 
For each $z\in D , \epsilon$ we consider the stopping times 
\begin{equation}\label{eq:taurhodef}
    \tau^{\epsilon, z}_{\rho}:=\inf\bigg\{ t>0: Z_z^\epsilon(t)\in\gamma_\rho\cup\partial D     \bigg\}
\end{equation}
in which the random trajectories hit either the small ball $\gamma_\rho$ or the boundary $\partial D.$ Our first auxiliary statement shows that, with overwhelmingly large probability, the dynamics cannot "meander" for an arbitrarily long time without hitting either $\gamma_\rho$ or $\partial D.$

\begin{lem}\label{lem:meandering} Let $D$ be a bounded and uniformly attractive set that contains the asymptotically stable equilibrium $z^*.$ Under Assumption \ref{Assumption:Dinterior} there exists $R>0$ such that for all $\rho<R$
\begin{equation*}    \limsup_{t\to\infty}\limsup_{\epsilon\to 0}\sup_{z\in D}\epsilon^2\log\pr\bigg[ \tau^{\epsilon, z}_{\rho}>t   \bigg]=-\infty.
\end{equation*} 
\end{lem}

\begin{proof} 
Let $\rho$ be small enough for \eqref{eq:gammarhodef} to hold. By the uniform attraction shown in Theorem \ref{thm:domainofattraction} there exists $T_0=T_0(\rho)>0$ such that $\sup_{z\in D}|Z_z^0(T_0)-z^*|_\e<\rho/2.$ In view of the latter, Lemma \ref{lem:controliccontinuity}(2)  and a triangle inequality imply that, if there exists a control $u$ such that $|Z^u_z(T_0)-z^*|_{\e}> 3\rho/4$ 
then 
$$  \frac{\rho}{4}< |Z^{u}_z(T_0)-Z^0_z(T_0)|_{\e}\leq \Lambda\big(|u|_{L^2([0,T_0];H)}\big).$$ Since $\Lambda$ is invertible with a non-decreasing inverse it follows that

$$  0<\frac{1}{2}\bigg(\Lambda^{-1}(\tfrac{\rho}{4})\bigg)^2\leq \frac{1}{2}|u|^2_{L^2([0,T_0];H)}.      $$

Hence there exists $s_0>0$ such that the set of trajectories $\phi\in C([0,T_0];\e)$ such that 1) $\phi(0)=z,$ $2)$ $\phi(t)\in D$ for all $t\in[0,T_0],$ and 3) $|\phi(T_0)-z^*|_{\e}> 3\rho/4$ is disjoint from the sublevel set $\Phi_{ T_0, z}(s_0)$ (recall the notation \eqref{eq:SublevelSets}). Hence all trajectories that satisfy 1), 2) and $\phi(t)\in D\setminus (\gamma_\rho\cup\partial D)$ for all $t\in [0,T_0]$ must also satisfy $$ \textnormal{dist}_{C([0,T_0];\e)}\big(\phi; \Phi_{T_0, z}(s_0)\big)\geq \rho/4. $$ 

From the ULDP upper bound 
it follows that 
\begin{equation*}
\begin{aligned}
  &\limsup_{\epsilon\to 0}\sup_{z\in D}\epsilon^2\log\pr\bigg[ \tau^{\epsilon, z}_{\rho}>T_0  \bigg]\leq    \limsup_{\epsilon\to 0}\sup_{z\in D}\epsilon^2\log\pr\bigg[\forall t\in [0, T_0]\;   Z_z^{\epsilon}(t)\in D\setminus(\gamma_\rho\cup\partial D )\bigg]\\&
  \leq  \limsup_{\epsilon\to 0}\sup_{z\in D}\epsilon^2\log\pr\bigg[ Z^\epsilon_z(t) \in \mathcal{D}, t \in [0,T_0] \text{ and }  \textnormal{dist}_{C([0,T_0];\e)}(Z_z^{\epsilon}; \Phi_{T_0, z}(s_0))\geq \rho/4\bigg]
  \\& \leq -s_0.
\end{aligned}   
\end{equation*}
This bound can then be bootstrapped via the Markov property so that for all $k\in\N$
\[\sup_{z\in D}\pr\bigg[ \tau^{\epsilon, z}_{\rho}>kT_0   \bigg] \leq \left( \sup_{z\in D}\pr\bigg[ \tau^{\epsilon, z}_{\rho}>T_0   \bigg]\right)^k.\]
We conclude that 
\begin{equation*}
    \limsup_{\epsilon\to 0}\sup_{z\in D}\epsilon^2\log\pr\bigg[ \tau^{\epsilon, z}_{\rho}>kT_0   \bigg]\leq -ks_0
\end{equation*}
and the proof is complete upon taking $k\to\infty.$
\end{proof}

Next, we let 
\begin{equation}\label{eq:Mdefinition}
    M:=1\vee \sup_{t\geq 0}|S_\alpha(t)|_{\mathscr{L}(\e)}. 
\end{equation}
Taking $\rho>0$ to be small enough, we consider a sphere
\begin{equation}\label{eq:Gammarhodef}
    \Gamma_\rho:=\partial B_\e(z^*, 2M\rho), \Gamma_\rho\cap\partial D=\varnothing 
\end{equation} 
of radius $2M\rho>\rho.$ The following lemma proves that "fast excursions" to $\Gamma_\rho$ of random trajectories issued from $\gamma_\rho$ are exponentially unlikely.
\begin{lem}\label{lem:fastexcursions} Let $D\subset\e$ be bounded and such that Assumption \ref{Assumption:Dinterior} holds.
    For any $\rho>0$ and with $\gamma_\rho, 
    \Gamma_\rho$ as in \eqref{eq:gammarhodef}, \eqref{eq:Gammarhodef} respectively we have
    \begin{equation*}
        \limsup_{T \to 0} \limsup_{\epsilon \to 0} \sup_{z \in \gamma_\rho} \epsilon^2 \log \pr\bigg[Z^{\epsilon}_z(t) \in \Gamma_\rho \text{ for some } t \in [0,T] \bigg] = -\infty.
    \end{equation*}
\end{lem}

\begin{proof}
    For any $z \in \gamma_\rho,$ $u \in L^2([0,T];H)$ and $Z_z^u$ a controlled path we have
    \begin{align*}
        |Z^{u}_z(t) - z^*|_\e \leq &|S_\alpha(t)(z - z^*)|_\e + \left| \int_0^t S_\alpha(t-s) ( B(Z^{u}_z(s)) - B(z^*))ds\right|_\e \nonumber\\
        &+ \left| \int_0^t S_\alpha(t-s) \Sigma(Z^{u}_z(s))u(s)ds \right| \nonumber\\
        &\leq M \rho + Ct + C \sqrt{t}|u|_{L^2([0,t];H)},
    \end{align*}
    where $M$ as in \eqref{eq:Mdefinition}, $C$ a constant independent of $z,T$ and we used the local Lipschitz continuity of $b, \sigma$.
    Notice that if $Z^u_z(t)\notin B_\e(z^*, 3M\rho/2)$ then 
    \begin{equation*}
        |u|_{L^2([0,t];H)} \geq \frac{\frac{ M\rho}{2} - Ct}{C\sqrt{t}} =: E(t). 
    \end{equation*}
     In other words, recalling the sublevel set notation \eqref{eq:SublevelSets}, we have the inclusion     $$\Phi_{z,T}\big(\tfrac{1}{2}E^2(T)\big)\subset\bigg\{ \phi\in C([0,T];\e): \phi(t)\in    B_\e(z^*, 3M\rho/2)\;\textnormal{for some}\;t\in[0,T]\bigg\}.$$
   Therefore, because $\Gamma_\rho$ is a sphere of radius $2M\rho$
    \begin{equation*}
        \sup_{z \in \gamma_\rho} \pr \bigg[ Z^\epsilon_z(t) \in \Gamma_\rho \text{ for some } t \in [0,T] \bigg] 
        \leq \sup_{z \in \gamma_\rho} \pr  \bigg[\textnormal{dist}\big(Z^\epsilon_z, \Phi_{z,T}(E^2(T)/2)\big) \geq \frac{M\rho}{2} \bigg] .
    \end{equation*}
    
    By the ULDP upper bound \eqref{eq:LULDP-upper},
    \begin{equation*}
        \limsup_{\epsilon \to 0} \sup_{z \in \gamma_\rho} \epsilon^2 \log \pr\bigg[Z^{\epsilon}_z(t) \in \Gamma_\rho \text{ for some } t \in [0,T] \bigg]  \leq -\frac{1}{2}E^2(T).
    \end{equation*}
    The conclusion follows because $E(T) \to \infty$ as $T \to 0$.
\end{proof}

The following lemma is concerned with the "typical behavior" of the random trajectories in an invariant, uniformly attracting set $D$. In particular it says that, with probability converging to $1$ as $\epsilon\to 0,$ they will hit the small ball $\gamma_\rho$ before hitting the boundary $\partial D.$ As we shall discuss in Remark \ref{rem:invariance} the assumption on invariance of the set $D$ is key for this property to hold.

\begin{lem}\label{lem:smallballlimit} Let $D$ be an invariant and uniformly attracting set that contains $z^*$. For $\rho$ small enough for \eqref{eq:gammarhodef} to hold, $\tau^{\epsilon,z}_\rho$ as in \eqref{eq:taurhodef} and all $z\in D$ we have
\begin{equation*}
      \lim_{\epsilon\to 0} \pr\bigg[ Z^{\epsilon}_z(\tau_{\rho}^{\epsilon, z}  )\in\gamma_\rho \bigg]=1. 
\end{equation*}  
\end{lem}
\begin{proof} This follows from the convergence in probability  $Z^{\epsilon}_z\rightarrow Z_z^0, \epsilon\to 0,$ which holds uniformly over compact time intervals, and the asymptotic stability of $z^*.$

In particular, for $z\in\gamma_\rho$ there is nothing to prove. Thus, we fix $z\in D\setminus\gamma_\rho.$ By virtue of asymptotic stability, the deterministic hitting time $\tau^z:=\inf\{ t>0 : |Z^0_z(t)-z^*|_\e=\rho/2  \}$
is finite. By invariance of $D$ (due to Proposition \ref{corr:DExistence}) and path continuity of $Z_z^0$ we have $$d=\textnormal{dist}_{\e}\big(\{Z_z^0(t); t\in[0,\tau^z]\}, \partial D    \big)>0.        $$
An application of the triangle inequality and the aforementioned convergence then yield
$$\pr\bigg[ Z^{\epsilon}_z(\tau_{\rho}^{\epsilon, z}  )\in\gamma_\rho \bigg]\geq \pr\bigg[ \big|Z^\epsilon_z-Z^0_z\big|_{C([0,\tau^z];\e)}  <\frac{d\wedge\rho}{2} \bigg]\longrightarrow 1\;,\epsilon\to 0.    $$
The proof is complete.
\end{proof}
\noindent From the preceding analysis it becomes clear that the probability of any trajectory, issued from $\Gamma_\rho,$ hitting a subset of $\partial D$ before $\gamma_\rho$ is (exponentially) small.
The next lemma provides an asymptotic upper bound for this probability. The proof crucially relies on the inner regularity of the quasipotential $V_D$ (Lemma \ref{lem:innerregularity}).
\begin{lem}\label{lem:rhoupperbound}
    Let $D\subset\e$ be a bounded and uniformly attracting set that contains $z^*$ and satisfies Assumption \ref{Assumption:Dinterior}. With  $V_{\bar{D}}, \tau^{\epsilon,z}_\rho, \Gamma_\rho $ as in \eqref{eq:VDbarquasipotential}, \eqref{eq:taurhodef}, \eqref{eq:Gammarhodef} respectively and for any closed $N \subset \partial D$
    \begin{equation*}
        \limsup_{\rho \to 0} \limsup_{\epsilon \to 0} \sup_{z \in \Gamma_\rho} \epsilon^2 \log \pr\bigg[ Z^\epsilon_z(\tau^{\epsilon,z}_\rho) \in N\bigg] \leq - V_{\bar{D}}(z^*,N).
    \end{equation*}
\end{lem}

\begin{proof}  Let $s_0<V_{\bar{D}}(z^*, N).$ For $T>0, \rho>0, z\in\Gamma_\rho$ we have
\begin{equation*}
    \begin{aligned}       \pr\bigg[Z^\epsilon_z(\tau^{\epsilon,z}_\rho) \in N\bigg]=\pr\bigg[ Z^\epsilon_z(\tau^{\epsilon,z}_\rho) \in N, \tau^{\epsilon,z}_\rho\leq T\bigg]+\pr\bigg[ Z^\epsilon_z(\tau^{\epsilon,z}_\rho) \in N, \tau^{\epsilon,z}_\rho>T\bigg].
    \end{aligned}
\end{equation*}
From Lemma \ref{lem:innerregularity}, there exist $\rho_0, \delta_0>0$ such that for all $\rho<\rho_0, \delta<\delta_0$ we have 
\begin{equation}\label{eq:lsc}
    \begin{aligned}
        V_{\bar{D}}(B(z^*,\rho), B(N, \delta) )>s_0.
    \end{aligned}
\end{equation}
By choosing a possibly smaller $\rho$, Lemma \ref{lem:meandering} furnishes a sufficiently large $T_0$ such that
\begin{equation}\label{eq:largetime}
    \begin{aligned}
       \limsup_{\epsilon\to 0}\sup_{z\in\Gamma_\rho}\epsilon^2\log\pr\bigg[ Z^\epsilon_z(\tau^{\epsilon,z}_\rho) \in N, \tau^{\epsilon,z}_\rho>T_0\bigg]\leq \limsup_{\epsilon\to 0}\sup_{z\in\Gamma_\rho}\epsilon^2\log\pr\bigg[\tau^{\epsilon, z}_\rho>T_0\bigg]\leq -s_0.
    \end{aligned}
\end{equation}
In view of \eqref{eq:lsc},  any path $\psi\in C ([0,T_0];\e)$ with $\psi(0)\in\Gamma_{\rho}$ and $\textnormal{dist}_{\e}(\psi(T_0), N)=\delta$ will satisfy
\begin{equation*}
    \begin{aligned}
 \textnormal{dist}_{C([0,T_0]; \e)}(\psi, \Phi_{z, T_0}(s_0))>0.
    \end{aligned}
\end{equation*}
The latter, along with the triangle inequality, implies that for all paths $\phi$ that exit $D$ through $N$ by time $T_0$
\begin{equation}\label{eq:contradiction}
    \begin{aligned}
 \textnormal{dist}_{C([0,T_0]; \e)}(\phi, \Phi_{z, T_0}(s_0))\geq |\phi-\psi|_{C([0,T_0]; \e)}+\textnormal{dist}_{C([0,T_0]; \e)}(\psi, \Phi_{z, T_0}(s_0))\geq \delta.
    \end{aligned}
\end{equation}

From \eqref{eq:contradiction} and the ULDP upper bound \eqref{eq:LULDP-upper} it follows that
\begin{equation*}\label{eq:smalltime}
    \begin{aligned}
       \limsup_{\epsilon\to 0}\sup_{z\in\Gamma_\rho}\epsilon^2\log\pr\bigg[ Z^\epsilon_z(\tau^{\epsilon,z}_\rho) \in N, \tau^{\epsilon,z}_\rho\leq T_0\bigg]\leq \limsup_{\epsilon\to 0}\sup_{z\in\Gamma_\rho}\epsilon^2\log\pr\bigg[   \textnormal{dist}_{C([0,T_0]; \e)}(Z^{\epsilon}_z, \Phi_{z, T_0}(s_0))>
        \delta      \bigg]\leq -s_0.
    \end{aligned}
\end{equation*}
\noindent From the latter, along with \eqref{eq:largetime}, we conclude that 
\begin{equation*}
\begin{aligned}
    \limsup_{\rho \to 0} &\limsup_{\epsilon \to 0} \sup_{z \in \Gamma_\rho} \epsilon^2\log \pr\bigg[Z^\epsilon_z(\tau^{\epsilon,z}_\rho) \in N\bigg]\\& \leq
         \limsup_{\rho \to 0} \limsup_{\epsilon \to 0} \sup_{z \in \Gamma_\rho}\max\bigg\{   \epsilon^2\log\pr\bigg[\tau^{\epsilon, z}_\rho>T_0\bigg], \epsilon^2\log\pr\bigg[   \textnormal{dist}_{C([0,T_0]; \e)}(Z^{\epsilon}_z, \Phi_{z, T_0}(s_0))>
        \delta      \bigg]  \bigg\}\leq -s_0.
\end{aligned}   
        \end{equation*}
Since $s_0< V_{\bar{D}}(z^*, N)$ was arbitrary, the proof is complete.\end{proof}

At this point we have all the necessary ingredients to prove Theorem \ref{thm:ExitTimeLowerBnd}. Before doing so, we introduce some additional notation:

For $n\in\N$ and $ z\in D$ we define the stopping times 
\begin{equation}\label{eq:donutStoppingtimes}
    \begin{aligned}
     &\theta^{\epsilon,z}_0:=0\\&     \tau^{\epsilon,z}_0:=\tau^{\epsilon,z}_\rho=\inf\big\{ t\geq \theta^{\epsilon,z}_0: Z_z^{\epsilon}(t)\in\gamma_\rho\cup \partial D\big\}\\&
        \quad\vdots
         \\&
        \theta^{\epsilon,z}_{n+1}:=\inf\big\{ t>\tau^{\epsilon,z}_n: Z_z^{\epsilon}(t)\in\Gamma_\rho\big\}\\&     \tau^{\epsilon,z}_{n+1}=\inf\big\{ t\geq \theta^{\epsilon,z}_{n+1}: Z_z^{\epsilon}(t)\in\gamma_\rho\cup \partial D\big\}.    
    \end{aligned}
\end{equation}
The lower bound on $\tau_D^{\epsilon, z}$ \eqref{eq:exittimes} rely on the asymptotic behavior of the Markov chain \begin{equation}\label{eq:Markovchain}
Z^{\epsilon,z}_n:=Z^{\epsilon}_z(\tau^{\epsilon,z}_{n})
\end{equation}
whose law is supported on $\gamma_\rho\cup \partial D.$

\begin{proof}{\textit{(Of Theorem \ref{thm:ExitTimeLowerBnd})}}
\begin{enumerate}
    \item Let $\delta>0, m\geq 1.$ From Lemma \ref{lem:rhoupperbound} with $N=\partial D$ and Lemma \ref{lem:VDequality} there exists $\rho_0$ such that for all $\rho<\rho_0$
    \begin{equation*}
    \begin{aligned}
         \limsup_{\epsilon \to 0} \sup_{z \in \Gamma_\rho} \epsilon^2 \log \pr\left( Z^\epsilon_z(\tau^{\epsilon,z}_0) \in \partial D\right) &\leq -  V_{\bar{D}}(z^*,\partial D)+\frac{\delta}{2}
         =-  V(z^*,\partial D)+\frac{\delta}{2}.
    \end{aligned}       
    \end{equation*}
The latter and the Markov property of $Z^\epsilon_z$ then furnish 
\begin{equation}\label{eq:lbaux1}
    \begin{aligned}
        \sup_{z\in D}\pr[\tau_D^{\epsilon,z}=\tau^{\epsilon,z}_{m}]&\leq \sup_{z\in D}\pr\bigg[\bigcap_{k=0}^{m-1}\tau_D^{\epsilon,z}\neq\tau^{\epsilon,z}_{k}\bigg]\sup_{z\in \Gamma_\rho}\pr\big[ Z^\epsilon_z(\tau^{\epsilon,z}_{0})\in\partial D\big]\\&
        \leq \sup_{z\in \Gamma_\rho}\pr\big[ Z^\epsilon_z(\tau^{\epsilon,z}_{0})\in\partial D\big]\leq \exp\bigg\{-\frac{1}{\epsilon^2}\bigg( V(z^*,\partial D)-\frac{\delta}{2}\bigg)\bigg\},
    \end{aligned}
\end{equation}
which holds for $\epsilon$ sufficiently small. Next, from Lemma \ref{lem:fastexcursions}, we choose $T_0$ small enough such that
\begin{equation*}
         \limsup_{\epsilon \to 0} \sup_{z \in \gamma_\rho} \epsilon^2 \log \pr\bigg[Z^{\epsilon}_z(t) \in \Gamma_\rho \text{ for some } t \in [0,T_0]\bigg]\leq-V(z^*,\partial D).
    \end{equation*}
The latter, along with the Markov property, implies that for each $m\in\N$
\begin{equation}\label{eq:lbaux2}
         \limsup_{\epsilon \to 0} \sup_{z \in G} \epsilon^2 \log \pr\left(\theta^{\epsilon,z}_{m+1}-\tau^{\epsilon,z}_{m} \leq T_0\right)\leq-V(z^*,\partial D).
    \end{equation}
Next notice that for each $k\in\N, m\leq k, z\in D$
\begin{equation*}
    \begin{aligned}
        \{\tau_D^{\epsilon,z}\leq kT_0 \}\subset\bigcup_{m=0}^{k}\{\tau_D^{\epsilon,z}=\tau^{\epsilon,z}_{m}   \}\cup\bigcup_{m=0}^{k}\{ \tau^{\epsilon,z}_{m+1}-\tau^{\epsilon,z}_{m} \leq T_0  \}
    \end{aligned}
\end{equation*}
( see e.g. the proof of \cite[Theorem 5.7.11(a)]{dembo2009large} for a similar argument). Thus, 
\begin{equation*}
    \begin{aligned}
        \pr\big[  \tau_D^{\epsilon,z}\leq kT_0   \big]&\leq \sum_{m=0}^{k}\bigg(\pr\big[\tau_D^{\epsilon,z}=\tau^{\epsilon,z}_{m}      \big]+\pr\big[ \theta^{\epsilon,z}_{m+1}-\tau^{\epsilon,z}_{m} \leq T_0  \big]    \bigg)\\&
        \leq \pr\big[\tau_D^{\epsilon,z}=\tau^{\epsilon,z}_{0}      \big]+ 2k\exp\bigg\{-\frac{1}{\epsilon^2}\bigg( V(z^*,\partial D)-\frac{\delta}{2}\bigg)\bigg\}\\&
        \leq \pr\big[ Z^{\epsilon}_z(\tau_{\rho}^{\epsilon, z}  )\in\partial D\big]+2k\exp\bigg\{-\frac{1}{\epsilon^2}\bigg( V(z^*,\partial D)-\frac{\delta}{2}\bigg)\bigg\},
    \end{aligned}
\end{equation*}
where we used \eqref{eq:lbaux1}, \eqref{eq:lbaux2} and the fact that $\theta^{\epsilon,z}_{m+1}\leq\tau^{\epsilon,z}_{m+1} $ almost surely. Letting $$k=k_0(\epsilon):=\bigg[\frac{1}{T_0}\exp\bigg\{ \frac{V( z^*,\partial D)-\delta      }{\epsilon^2} \bigg\}  \bigg]+1,$$
where $[\cdot]$ here is the integer part,
we obtain the estimate 
\begin{equation*}
   \pr\big[  \tau_D^{\epsilon,z}\leq e^{\frac{1}{\epsilon^2}(V( z^*,\partial D)-\delta  )     }  \big]\leq \pr\big[  \tau_D^{\epsilon,z}\leq k_0T_0   \big]\leq \pr\big[ Z^{\epsilon}_z(\tau_{\rho}^{\epsilon, z}  )\in\partial D\big]+Ce^{-\frac{\delta}{2\epsilon^2}}
\end{equation*}
which holds for $\epsilon$ sufficiently small. From an application of Lemma \ref{lem:smallballlimit} we conclude that the right-hand side converges to $0$ as $\epsilon\to 0.$ The argument is complete.
\item Let $\delta>0.$ By Chebyshev's inequality we have 
\begin{equation*}
  \epsilon^2\log\pr\bigg[  \tau_D^{\epsilon,z}>e^{\frac{1}{\epsilon^2}(V( z^*,\partial D)-\delta  )     }  \bigg] +V( z^*,\partial D)-\delta \leq \epsilon^2\log\ex[ \tau_D^{\epsilon,z}].
   \end{equation*}
   The statement follows by an application of part (1).\end{enumerate}\end{proof}

\begin{rem}[On the role of invariance]\label{rem:invariance} On the one hand, the assumption that $D$ is invariant is not necessary for the validity of the exit time upper bound (Theorem \ref{thm:ExitTimeUpperBnd}(1)). On the other hand, invariance is crucial for the exit time lower bound (Theorem \ref{thm:ExitTimeLowerBnd}(2)). In particular, it is used in Lemma \ref{lem:smallballlimit} which states that, with overwhelming probability, random trajectories hit a small ball around $z^*$ before hitting $\partial D$. Without this assumption, one cannot exclude cases in which, on average, trajectories exit from $D$ faster than $Ce^{V(z^*,\partial D)/\epsilon^2}$ as $\epsilon\to0.$ 
\end{rem}

 \subsection{Exit shape asymptotics}\label{Sec:ExitShapes} This section is devoted to the asymptotic distribution of the exit shapes \eqref{eq:exitshape}. Our arguments are largely based on the auxiliary Lemmas \ref{lem:meandering}-\ref{lem:rhoupperbound} that were used for the proof of Theorem \ref{thm:ExitTimeLowerBnd}. For this reason, we shall frequently use notation that was introduced in Section \ref{Sec:ExitTimesLB}. First, we prove a "Law of Large Numbers" (LLN) which shows that $ Z_z^\epsilon\big( \tau_D^{\epsilon, z}  \big)$ "almost" concentrates on minimizers of the quasipotential $$\partial D\ni z\longmapsto V_{\bar{D}}(z^*, z)\in[0,\infty]$$
 with probability converging to $1$ as $\epsilon\to 0$ (for a justification of the quotes see Theorem \ref{thm:exitshapeldp}(1) and Remark \ref{rem:exitshapeldp} below). Furthermore, we prove Large Deviations upper and lower bounds formulated respectively in terms of the functionals \begin{equation}\label{eq:J1definition}
     \partial D\ni z\longmapsto J_1(z):=V_{\bar{D}}(z^*, z)-V(z^*,\bar{D}^c)
 \end{equation}
and
 \begin{equation}\label{eq:J2definition}
     \partial D\ni z\longmapsto J_2(z):=V_D(z^*, z)-V(z^*,\partial D)\in [0,\infty],
 \end{equation}
  where $V_D$ and $V_{\bar{D}}$ are given by Definition \ref{dfn:quasipotentialVd}.
 Before we present our main result, Theorem \ref{thm:exitshapeldp}, we introduce a notion of "regular" boundary points, defined exclusively in terms of minimal energy. This non-geometric definition is particularly well-suited for infinite-dimensional settings such as the one considered here.

 \begin{dfn}[Regular boundary points]\label{def:regularpoints}
   Let $D\subset\e$ be a bounded set and $I_{z,T_0}$ the LULDP rate function \eqref{eq:LRateFunctionWave}. A point $z\in\partial D$ is called \textbf{regular} if for any $\delta>0$ there exists $T_0>0$ and a path $\phi\in C([0,T_0];\e)$ such that $\phi(0)=z,$ $I_{z,T_0}(\phi)<\delta$
   and for all $t\in(0, T_0],$ $\phi(t)\in\bar{D}^c.$
\end{dfn}

\begin{rem}\label{Rem:RegularPoints} In finite dimensions, all boundary points of a domain $D\subset\R^d$ with a smooth boundary are regular (per Definition \ref{def:regularpoints}), provided that the stochastic dynamics are governed by an It\^o diffusion with uniformly elliptic diffusion coefficient. Apart from avoiding questions of boundary smoothness in infinite dimensions, our definition is also agnostic to the degeneracy of $\sigma.$ This last aspect is particularly important for wave equations which, in view of \eqref{eq:modelsystem}, do not satisfy such ellipticity conditions. The reader is referred to the discussion preceding Assumption \ref{Assumption:MinimizingRegularPoints}, Section \ref{Sec:BoundaryPoints} for more details on geometric conditions and relevant references to the literature. 
\end{rem}

\begin{thm}[Exit shape asymptotics]\label{thm:exitshapeldp} Let $D\subset\e$ be a bounded, invariant, uniformly attracting domain that contains  the asymptotically stable equilibrium $z^*$ and satisfies Assumption \ref{Assumption:Dinterior}. With $V_D, V_{\bar{D}}, J_1, J_2$ as in \eqref{eq:VDquasipotential},\eqref{eq:VDbarquasipotential}, \eqref{eq:J1definition}, \eqref{eq:J2definition} respectively and under Assumptions \ref{Assumption:b}-\ref{Assumption:sigmaNondeg}, we have that for all $z\in D$ the following hold:
\begin{enumerate}
\item (\textnormal{LLN}) For any closed $N\subset\partial D$ such that $V_{\bar{D}}(z^*,N)>V(z^*,  \bar{D}^c)$ we have $$\lim_{\epsilon\to 0}\pr\bigg[  Z^\epsilon_z(\tau_D^{\epsilon, z})\in N\bigg]=0.$$
    \item(\textnormal{Large deviations upper bound}) For any closed  $N\subset \partial D$   \begin{equation}\label{eq:exitshapeLDPub}
    \limsup_{\epsilon\to 0}\epsilon^2\log\pr\bigg[Z^\epsilon_z(\tau^{\epsilon, z}_{D})\in N\bigg]\leq -\inf_{z\in N}J_1(z).    
    \end{equation}
    \item (\textnormal{Large deviations lower bound}) Assume moreover that $D$ is open. For all $\eta>0$ and any regular point $y\in\partial D$
    \begin{equation}\label{eq:exitshapelb}
        \liminf_{\epsilon\to 0}\epsilon^2\log\pr\bigg[ |Z^\epsilon_z(\tau^{\epsilon, z}_{D})-y|_{\e}<\eta \bigg]\geq -J_2(y).
    \end{equation}
    \end{enumerate}   
\end{thm}

Before we proceed to the proof, we shall attempt to explain the role of the quasipotentials  $V_D, V_{\bar{D}}$ for our exit shape asymptotics.

\begin{rem}[On the roles of $V_D$ and $V_{\bar{D}}$]\label{rem:VD} 
Note that $V_D$ and $V_{\bar{D}}$ appear in the exit shape asymptotics of Theorem \ref{thm:exitshapeldp}, and the reason for this is fundamental. The proof of Theorem \ref{thm:exitshapeldp}(3) requires us to construct a path that exits $D$ from a neighborhood $N$ of a regular boundary point $y\in\partial D$. To do so, we need to make sure that such a path does not "wander" on the exterior of $D$ before it hits $N.$ Moreover, the proofs of Theorem \ref{thm:exitshapeldp}(1), (2) make use of the inner regularity result of the quasipotential $V_{\bar{D}}$ Lemma \ref{lem:innerregularity}.
Hence, $V_D$ and $V_{\bar{D}}$ are the relevant quantities for meaningful lower and upper bounds. 
\end{rem}

\begin{proof}{\textit{(Of Theorem \ref{thm:exitshapeldp})}}
Let $\tau^{\epsilon,z}_k$ be defined by \eqref{eq:donutStoppingtimes}
\begin{enumerate}
       \item Given Lemmas \ref{lem:fastexcursions}-\ref{lem:rhoupperbound}, the proof is analogous to that of \cite[Theorem 5.7.11(b)]{dembo2009large}. For $z\in D, k\in\N, T>0$ we have 
\begin{equation}\label{eq:exitshapemainestimates}
    \begin{aligned}
        \pr\bigg[  Z^\epsilon_z(\tau_D^{\epsilon, z})&\in N   \bigg]= \pr\bigg[  Z^\epsilon_z(\tau_D^{\epsilon, z})\in N, \tau_D^{\epsilon, z}>\tau^{\epsilon, z}_k\bigg]\\&+\pr\bigg[  Z^\epsilon_z(\tau_D^{\epsilon, z})\in N, \tau_D^{\epsilon, z}=\tau_{m}^{\epsilon, z}
        \;\text{for some}\;m\in\{0,\dots, k
        \}\bigg]\\&
        \leq \pr\bigg[ \tau_D^{\epsilon, z}>k T\bigg]+\pr\bigg[ \tau_k^{\epsilon, z}\leq k T \bigg]+\pr\bigg[ Z^\epsilon_z(\tau_{0}^{\epsilon, z})\in N   \bigg]\\&+\sum_{m=1}^{k}\pr\big[\tau_D^{\epsilon, z}> \tau^{\epsilon, z}_{m-1}    \big]\pr\bigg[ Z^\epsilon_z(\tau_{m}^{\epsilon, z})\in N  \bigg|\tau_D^{\epsilon, z}> \tau^{\epsilon, z}_{m-1}    \bigg]\\&
        \leq \frac{1}{kT}\ex\big[ \tau_D^{\epsilon, z}\big]+ \pr\bigg[ Z^\epsilon_z(\tau_{0}^{\epsilon, z})\in N \bigg]\\&+k\sup_{z\in\gamma_{\rho}}\pr\bigg[ Z^{\epsilon}_z(t) \in \Gamma_\rho \text{ for some } t \in [0,T]   \bigg]+k\sup_{z\in\Gamma_\rho}\pr\bigg[Z^\epsilon_z(\tau_{0}^{\epsilon, z})\in N   \bigg],
    \end{aligned}
\end{equation}
where we used Chebyshev's inequality and the Markov property on the last line.
Next let $\delta>0.$ From Lemma \ref{lem:fastexcursions}, we can find $T_0$ such that
\begin{equation*}
        \limsup_{\epsilon \to 0} \sup_{z \in \gamma_\rho} \epsilon^2 \log \pr\bigg[Z^{\epsilon}_z(t) \in \Gamma_\rho \text{ for some } t \in [0,T_0] \bigg]\leq -V_{\bar{D}}( z^*,N)+\delta.
    \end{equation*}
    Moreover, from Lemma \ref{lem:rhoupperbound} we can find $\rho$ such that 
    \begin{equation*}
        \limsup_{\epsilon \to 0} \sup_{z \in \Gamma_\rho} \epsilon^2 \log \pr\bigg[Z^\epsilon_z(\tau^{\epsilon,z}_0) \in N\bigg] \leq - V_{\bar{D}}(z^*,N)+\delta
    \end{equation*}
   (recall that $\tau^{\epsilon,z}_\rho=\tau^{\epsilon,z}_0).$  A combination of the last three displays, along with the exit time upper bound from Theorem \ref{thm:ExitTimeUpperBnd}(1)
    implies 
\begin{equation*}
    \begin{aligned}
        \pr\bigg[  Z^\epsilon_z(\tau_D^{\epsilon, z})\in N   \bigg]&\leq \frac{1}{kT_0}e^{(V(z^*, \bar{D}^c)+\delta   )/\epsilon^2}  \\&  +    2k\exp\bigg\{-\frac{1}{\epsilon^2}\bigg( V_{\bar{D}}(z^*,N)-\delta\bigg)\bigg\}+ \pr\bigg[ Z^\epsilon_z(\tau_{0}^{\epsilon, z})\in N \bigg].
    \end{aligned}
\end{equation*}
Choosing $k=[e^{(V(z^*,\bar{D}^c)+2\delta   )/\epsilon^2}],$ (with $[\cdot]$ being the integer part), $$ 
0<\delta<\frac{V_{\bar{D}}(z^*, N)- V(z^*,\bar{D}^c)}{3}$$ and invoking Lemma \ref{lem:smallballlimit} we conclude that 
   \begin{equation*}
    \begin{aligned}
        \limsup_{\epsilon\to 0} \pr\bigg[  Z^\epsilon_z(\tau_D^{\epsilon, z})\in N   \bigg]&\leq \limsup_{\epsilon\to 0}\pr\bigg[ Z^\epsilon_z(\tau_{0}^{\epsilon, z})\in N \bigg]\\&+\limsup_{\epsilon\to 0}\exp\bigg(-\frac{1}{\epsilon^2}(V_{\bar{D}}(z^*, N)-V(z^*,\bar{D}^c)-3\delta   )\bigg)=0.
     \end{aligned}
\end{equation*}
The proof is complete.

    \item The upper bound essentially follows from the estimates in  \eqref{eq:exitshapemainestimates}. In particular, for any initial condition $z\in\gamma_\rho$ we have $\pr[Z_z^\epsilon(\tau_1^{\epsilon, z})\in N]=0.$ Hence, using the exact same arguments as in the proof of part (1) we get for any $\delta>0$
    $$\limsup_{\epsilon\to 0}\epsilon^2\log\sup_{z\in\gamma_\rho}\pr\bigg[Z^\epsilon_z(\tau^{\epsilon, z}_{D})\in N\bigg]\leq -V_{\bar{D}}(z^*,N)+V(z^*,\bar{D}^{c})+3\delta=
    -\inf_{y\in N}J_1(y)+3\delta.$$
   For $z\in D\setminus\gamma_\rho,$ Lemma \ref{lem:smallballlimit} and the strong Markov property furnish the asymptotic estimates
    $$ \lim_{\epsilon\to 0}\pr\bigg[Z^\epsilon_z(\tau^{\epsilon, z}_{\rho})\in \gamma_\rho   \bigg]=1$$
    and 
    \begin{equation*}
        \begin{aligned}
\limsup_{\epsilon\to 0}\epsilon^2\log\pr\bigg[Z^\epsilon_z(\tau^{\epsilon, z}_{D})\in N\bigg]&\leq \limsup_{\epsilon\to 0}\epsilon^2\log\bigg(\pr\bigg[Z^\epsilon_z(\tau^{\epsilon, z}_{\rho})\in \gamma_\rho   \bigg]\sup_{y\in\gamma_\rho}\pr\bigg[Z^\epsilon_y(\tau^{\epsilon, z}_{D})\in N\bigg]\bigg)
\\&\leq -\inf_{y\in N}J_1(y)+3\delta.
        \end{aligned}
    \end{equation*}
Since $\delta$ is arbitrary, the argument is complete.
\item 
\noindent We break the proof in two steps. In \textbf{Step 1} we discuss the case of initial data $z\in\gamma_\rho$, while in \textbf{Step 2} we consider initial data $z\in D\setminus\gamma_\rho$. 

\noindent \textbf{Step 1:} We start by showing that \eqref{eq:exitshapelb} holds uniformly over initial data $z\in\gamma_\rho$ and $\rho$ sufficiently small. Recalling the Markov chain \eqref{eq:Markovchain}, let $$\zeta^{\epsilon, z}=\inf\big\{ n\in\N: Z^{\epsilon,z}_n\notin\gamma_\rho \big\}$$
and note that 
\begin{equation*}
    \begin{aligned}
\pr\bigg[\big|Z^\epsilon_z(\tau^{\epsilon, z}_{D})-y\big|_{\e}<\eta\bigg]=\pr\bigg[Z^{\epsilon, z}_{\zeta^{\epsilon,z}}\in B_{\partial D}(y, \eta)\bigg],
    \end{aligned}
\end{equation*}
where the subscript $\partial D$ indicates that the ball is taken in the subspace topology of $\partial D.$ By virtue of the strong Markov property we have (with $Z^{\epsilon, z}_{0}\equiv  Z_z^{\epsilon, z}(\tau^{\epsilon,z}_\rho)$)
\begin{equation*}
    \begin{aligned}       \pr\bigg[\big|Z^\epsilon_z(\tau^{\epsilon, z}_{D})-y\big|_{\e}<\eta\bigg]&=\sum_{k=0}^{\infty}\pr\bigg[\zeta^{\epsilon,z}=k,\;Z^{\epsilon, z}_{\zeta^{\epsilon,z}}\in B_{\partial D}(y, \eta)\bigg]\\&
    \geq \sum_{k=1}^{\infty}\inf_{z\in\gamma_\rho}\pr\bigg[ Z^{\epsilon, z}_{0}\in \gamma_\rho \bigg]^{k-1}\inf_{z\in\gamma_\rho}\pr\bigg[Z^{\epsilon, z}_{0}\in B_{\partial D}(y, \eta) \bigg] \\&=\frac{\inf_{z\in\gamma_\rho}\pr\bigg[Z^{\epsilon, z}_{0}\in B_{\partial D}(y, \eta)  \bigg]}{1-\inf_{z\in\gamma_\rho}\pr\bigg[ Z^{\epsilon, z}_{0}\in \gamma_\rho \bigg]} \\&
 \geq \frac{\inf_{z\in\gamma_\rho}\pr\bigg[Z^{\epsilon, z}_{0}\in B_{\partial D}(y, \eta)  \bigg]}{\sup_{z\in\gamma_\rho}\pr\bigg[ Z^{\epsilon, z}_0\notin \gamma_\rho \bigg]}.
\end{aligned}
\end{equation*}

An upper bound for the denominator follows from Lemma \ref{lem:rhoupperbound} with $N=\partial D.$ Indeed, for $\delta>0,$ there exist $\rho, \epsilon$ sufficiently small such that  
\begin{equation}\label{eq:exitshapeLB:denominator}
    \begin{aligned}
       \sup_{z\in\gamma_\rho}\pr\bigg[ Z^{\epsilon, z}_{0}\notin \gamma_\rho \bigg]=\sup_{z\in\gamma_\rho}\pr\bigg[ Z^{\epsilon, z}_{0}\in\partial D \bigg]=\sup_{z\in\gamma_\rho}\pr\bigg[ Z^{\epsilon}_z(\tau^{\epsilon, z}_\rho)\in\partial D \bigg]\leq e^{ -\big(V(z^*, \partial D)+\delta\big)/\epsilon^2}.
    \end{aligned}
\end{equation}
\noindent In order to obtain a lower bound for the numerator we make use of the following:\\
\noindent \textbf{Claim:} For any $\delta>0$, there exist $T_0>0, r>0$, and $\rho>0$ such that for any $ z \in \gamma_\rho$ there exists a controlled path $\phi_z \in C([0,T_0];\e)$ such that  $\phi_z(0)=z$, 
$$I_{z,T_0}(\phi_z)<V_D(z^*, y)+\delta$$  and $$\big\{|Z^{\epsilon}_z-\phi_z|_{C([0,T_0];\e)}<r\big\}\subset \big\{Z^{\epsilon, z}_{0}\in B_{\partial D}(y, \eta)   \big\}.  $$
\noindent The latter, along with the LULDP lower bound \eqref{eq:LULDP-lower}, immediately yield
\begin{equation}\label{eq:exitshapeLB:numerator}
    \begin{aligned}
        \liminf_{\epsilon\to 0}\inf_{z\in\gamma_\rho}\bigg\{\epsilon^2\log\pr\bigg[Z^{\epsilon, z}_{0}\in B_{\partial D}(y, \eta)  \bigg]+I_{z,T_0}(\phi_z)\bigg\} \geq 0.
    \end{aligned}
\end{equation}
Then, the combination of \eqref{eq:exitshapeLB:numerator} and \eqref{eq:exitshapeLB:denominator} furnishes
\begin{equation*}
    \begin{aligned}       \liminf_{\epsilon\to 0}\inf_{z\in\gamma_\rho}\epsilon^2\log \pr\bigg[\big|Z^\epsilon_z(\tau^{\epsilon, z}_{D})-y\big|_{\e}<\eta\bigg]&\geq V(z^*, \partial D)-  V_D(z^*, y)   -    2\delta=-J_2(y)-2\delta
\end{aligned}
\end{equation*}
which concludes the proof of Step 1 since $\delta$ can be taken to be arbitrarily small.

 \textit{Proof of Claim:} Turning to the construction of $\phi_z,$ we assume without loss of generality that $V_D(z^*, y)<\infty$ (otherwise the lower bound holds trivially). Thus, there exists $T>0$ a control $v_1$ with  \begin{equation}\label{eq:v1action}
     \frac{1}{2}|v_1|^2_{L^2([0, T];H)}<V_D(z^*, y)+\delta/2
 \end{equation}  and a controlled path $\phi_1:=Z^{v_1}_{z^*}\in C([0,T];\e)$ with  $\phi_1(0)=z^*, \phi_1(T)=y$ and $\phi_1(t)\in D$ for all $t\in[0, T ).$ From the regularity of the boundary point $y,$ there exists a controlled path $\phi_2:=Z_y^{v_2}$ such that $\{\phi_2(t);\;t\in[0,t_0]\}\subset\bar{D}^c$ and
\begin{equation}\label{eq:v2action}
    \frac{1}{2}|v_2|^2_{L^2([0,t_0];H)}<\frac{\delta}{2}.
\end{equation}

Because $\phi_2(t)$ is continuous in $\e$, there exists $T_1 \in (0,t_0)$ such that for all $t \in (0,T_1)$, 
$|\phi_2(t)-y|_{\e}<\eta/3.     $
 Define $\phi_{z^*}\in C([0,T+T_1];\e),$ as the concatenation of $\phi_1$ and  $\phi_2$. This trajectory has the properties that $\phi_{z^*}(0)=z^*, \phi_{z^*}(T+T_1)=\phi_2(T_1)=:y'\in \bar{D}^c,$  
\begin{equation}\label{eq:phiydistance}
    \begin{aligned}
        |\phi_{z^*}(t)-y|_{\e}<\eta/3,\;\; \text{ for all } t\in[T, T+T_1]
    \end{aligned}
\end{equation}
and in view of \eqref{eq:v1action}, \eqref{eq:v2action}
\begin{equation}\label{eq:phiaction}
    I_{z^*, T+t_0}(\phi_{z^*})<V_D(z^*,  y)+\delta.
\end{equation}

Because $\phi_{z^*}(t) \in D$ for all $t \in [0,T)$,  $|\phi_{z^*}(t) -y|< \frac{\eta}{3}$ for $t \in [T, T+T_1]$, and $\phi_{z*}(T+T_1) \in \bar{D}^c$ there exists a small $\tilde{\eta} \in \left(0,\frac{\eta}{3} \wedge \textnormal{dist}_\e(y',D) \right)$ such that any trajectory $\psi \in C([0,T + T_1];\e)$ with the property that
\begin{equation} \label{eq:exit-tube}
    \sup_{t \in [0,T + T_1]}|\phi_{z^*}(t) - \psi(t)|_\e< \tilde{\eta},
\end{equation}
will have the properties that $\psi$ exits $D$ and its exit location is within $\eta$ of $y$.

Define a control
\begin{equation}\label{eq:udef}
    u(t)=\begin{cases}
           v_1(t),&\;\;t\in[0,T]\\
           v_2(t-T),&\;\;t\in[T,T+T_1]
    \end{cases}
\end{equation}
so that $\phi_{z^*}=Z_{z^*}^u.$ We use the continuity of the skeleton equation with respect to initial conditions, Lemma \ref{lem:controliccontinuity}, to deduce that there exists $\rho>0$ such that for any $z\in\gamma_{\rho}$ 
\begin{equation}\label{eq:phiZvdistance}
      \sup_{t\in[0,T+T_1]}|Z^u_{z^*}(t)-Z^u_{z}(t)|_{\e}= \sup_{t\in[0,T+T_1]}|\phi_{z^*}(t)-Z^u_{z}(t)|_{\e}<\frac{\tilde\eta}{3}.
\end{equation}
Then $T_0 := T + T_1$, $r := \frac{\tilde \eta}{3}$, and $\phi_z:= Z^u_z$ are the quantities that we claimed existed.
From \eqref{eq:phiaction}, $I_{z,T_0}(\phi_z)< V_D(z^*, y) + \delta$.
If a random trajectory $Z^\epsilon_z$ has the properties that 
\begin{equation}
    \sup_{t \in [0,T_0]} | Z^\epsilon_z - \phi_z|_\e<r,
\end{equation}
then by the triangle inequality,
\begin{equation}
   \sup_{t \in [0,T_0]} |Z^\epsilon_z(t) - \phi_{z^*}(T)|_\e < \sup_{t \in [0,T_0]} |Z^\epsilon_z(t) - \phi_{z}(t)|_\e + \sup_{t \in [0,T_0]} |\phi_z(t) - \phi_{z^*}(t)|_\e < \tilde{\eta}.
\end{equation}
Therefore, by \eqref{eq:exit-tube}, the random path must exit $D$ and its exit location must be within $\eta$ of $y$.

\noindent \textbf{Step 2:} It remains to prove \eqref{eq:exitshapelb} for initial data $z\in D\setminus\gamma_\rho.$ This follows by uniform attraction to $z^*,$  the strong Markov property and the previous step. Indeed, from Lemma \ref{lem:smallballlimit}, for $\rho$ sufficiently small we have 
$$\lim_{\epsilon\to 0}\pr\bigg[Z^\epsilon_0\in \gamma_\rho   \bigg]= \lim_{\epsilon\to 0}\pr\bigg[Z^\epsilon_z(\tau^{\epsilon, z}_{\rho})\in \gamma_\rho   \bigg]=1$$
    and 
    \begin{equation*}
    \begin{aligned}
         \liminf_{\epsilon\to 0}\epsilon^2\log\pr\bigg[ |Z^\epsilon_z(\tau^{\epsilon, z}_{D})-y|_{\e}<\eta \bigg]
        &\geq \liminf_{\epsilon\to 0}\epsilon^2\log\bigg(\pr\bigg[Z^\epsilon_0\in \gamma_\rho   \bigg]\inf_{z\in\gamma_\rho}\pr\bigg[ |Z^\epsilon_z(\tau^{\epsilon, z}_{D})-y|_{\e}<\eta \bigg] \bigg)\\&
        \geq -J_2(y)-2\delta.
    \end{aligned}  
    \end{equation*}
    \end{enumerate} 
    The proof is complete.  
\end{proof}

We collect a few observations for our exit shape asymptotics in the following remark:
\begin{rem}\label{rem:exitshapeldp} The LLN, Theorem \ref{thm:exitshapeldp}(1), says that the probability of exiting from any part of the boundary $N$ with $V_{\bar{D}}(z^*,N)>V(z^*,\bar{D}^c)$ converges to $0.$ If the quantities $V(z^*,\bar{D}^c)$ and $V(z^*,\partial D)=V_{\bar{D}}(z^*,\partial D)$ (see also Lemma \ref{lem:VDequality}) happen to be equal, we can conclude that, with probability converging to $1,$ exits take place near minimizers of $V_{\bar{D}}(z^*,\cdot).$

On the level of exit shape large deviations, if $D$ is such that
\begin{equation}\label{eq:BoundaryComplementCondition}
    V(z^*,\bar{D}^c)=V(z^*,\partial D),
\end{equation}
and if Lemma \ref{lem:innerregularity} on inner regularity can be shown to hold for $V_D$ in place of $V_{\bar{D}}$,
then the upper and lower bounds from Theorem \ref{thm:exitshapeldp}(1), (2) hold with matching (non-negative) rate function
$$J_1(y)=J_2(y)=V_D(z^*,y)-V_D(z^*,\partial D), \;y\in\partial D.$$
Of course, even if $D$ satisfies \eqref{eq:BoundaryComplementCondition}, the lower bound only holds for regular points $y\in\partial D.$ Thus if, in addition to \eqref{eq:BoundaryComplementCondition}, we assume that $\partial D$ only consists of regular points then Theorem \ref{thm:exitshapeldp}(2) is truly equivalent to the LDP lower bound
\begin{equation*}
        \liminf_{\epsilon\to 0}\epsilon^2\log\pr[ Z^\epsilon_z(\tau^{\epsilon, z}_{D})\in G ]\geq -\inf_{y\in G}J_2(y)
    \end{equation*}
 which holds for any open $G\subset\partial D.$ Together with the upper bound \eqref{eq:exitshapeLDPub}, the latter results to a full LDP for the exit shape $\{Z_z^{\epsilon}(\tau^{\epsilon,z}_D)\;;\epsilon>0 \}$ with rate function $J=J_1=J_2.$ Examples of regular points and necessary conditions for \eqref{eq:BoundaryComplementCondition} to hold are discussed in the next subsection.
\end{rem}

 \subsection{Regular boundary points and matching bounds}\label{Sec:BoundaryPoints}  In this section we are concerned with the following questions:
\begin{enumerate}
    \item \textit{What further assumptions on $D$ imply \eqref{eq:BoundaryComplementCondition}?} 
    \item \textit{What are some examples of regular and irregular boundary points for some typical choices of $D$ ?}
    \end{enumerate}
We remind the reader, once again, that Remark \ref{rem:notation1} of Section \ref{sec:Controllability} is in place. In particular, for the duration of this section, even though we work with the localized solution $Z_{z,n_{D}}^\epsilon$ and corresponding skeleton equation $Z^u_{z,n_{D}}$ , we omit the subscript $n_D$ and we simply write $Z_z^\epsilon$. The same index omission is in place for the LULDP rate function $I^{n_D}_{T,z}$ which will be simply denoted by $I_{T,z}.$

    \subsubsection{\textbf{On $\mathbf{V(z^*, \bar{D}^c)=V(z^*, \partial D)}$}} In the finite-dimensional setting of \cite{dembo2009large, freidlin1998random}, this equality is a consequence of smoothness of the boundary and uniform ellipticity of the diffusion coefficient. The former implies an \textit{exterior ball condition} i.e. that for any boundary point $z\in\partial D$ one can find $r>0$ and $y\in \bar{D}^c$ such that $B(y,r)\cap D=\varnothing$ and $z\in\partial B(y,r).$ The latter implies that any point in $\bar{D}^c\cap B(y,r),$ that is arbitrarily close to $z,$ can be connected to $z$ via a linear path of arbitrarily small energy (see e.g. Assumption A.4, Exercise 5.7.29 in \cite{dembo2009large}).

    Exterior ball conditions on $D$ are suitable for infinite dimensional settings where boundary smoothness/geometric conditions are not easy to formulate. Indeed, in the setting of parabolic SPDEs, the authors of \cite{salins2021metastability} work on the state space $\e:=C_0(0,\ell)$ and assume a similar condition that for all $z\in\partial D, \delta>0$
    $$B_\e(z, \delta)\cap\bar{D}^c\neq\varnothing. $$ 
    Exploiting the smoothing properties of the heat semigroup, they proceed in showing that any two states in $\e$ can be connected via controlled paths of arbitrarily small energy, provided that they are sufficiently close. The latter along with the exterior ball condition are sufficient to guarantee the desired equality \eqref{eq:BoundaryComplementCondition}.

      As explained in the introduction, 
      such good controllability properties are no longer true in the setting of hyperbolic equations with superlinear nonlinearities (see also the counterexamples from \cite[Theorem 2]{Zuazua1993}). The authors of \cite{cerrai2016smoluchowski} work in the small-mass, single-equilibrium,  additive-noise setting and consider exits from cylinders of the form $D=G\times H^{-1},$ $G\subset L^2.$ Similar to our setting, they are only able to prove controllability from $x^*=\Pi_1z^*$ to any point in $H^1.$ Despite the lack of controllability, it turns out that, in their setting,  the following "exterior line segment condition" is sufficient to prove \eqref{eq:BoundaryComplementCondition}:
      \\ \\
      \noindent \textbf{Exterior line segment condition}( \cite{cerrai2016smoluchowski}, Hypothesis 3])   
     
      For any $x\in\partial D\cap H^1$ there exists $x'\in H^1\cap\bar{D}^c$ such that 
      $$\ell_{x,x'}:=\big\{ (1-t)x+tx'; t\in(0,1]    \big\}\subset\bar{D}^c.$$
   
    By virtue of the Hahn-Banach theorem and the density of $H^1$ in $L^2,$ this condition is satisfied for any convex $D\subset L^2.$

       In contrast to \cite{cerrai2016smoluchowski}, where exits happen only from directions of the position component, we consider general bounded domains of the full phase space $\e.$ As a consequence, exits can occur due to both the position and the velocity components and questions of boundary regularity become more subtle. In fact, even though the "exterior line segment condition" is satisfied for typical convex subsets of $\e$ (such as cylinders and balls) it is still possible to find boundary points from which instant exits to nearby exterior points require "large energy"; see Example \ref{ex:irregularPointscylinders} below.  Using the notion of regular boundary points introduced in Definition \ref{def:regularpoints}, we are able to formulate a sufficient condition for \eqref{eq:BoundaryComplementCondition} that is more appropriate for our setting. Examples of such points in the case where $D$ is spherical or cylindrical are presented in the next section.

       \begin{customthm}{6}\label{Assumption:MinimizingRegularPoints} 
           The collection $\mathcal{R}\subset\partial D$ of regular boundary points is non-empty and $$V(z^*,\partial D)= V(z^*,\partial D\cap \mathcal{R}).         $$
     \end{customthm}

       \begin{lem} Under Assumption \ref{Assumption:MinimizingRegularPoints} we have $$V(z^*, \bar{D}^c)=V(z^*, \partial D).$$
      \end{lem}
      \begin{proof} By continuity of controlled paths and monotonicity of $I_{z^*, T}$ in $T$ it is clear that $V(z^*, \bar{D}^c)\geq V(z^*, \partial D).$ In order to show the reverse inequality
      $V(z^*, \bar{D}^c)\leq V(z^*, \partial D)$
      let $\eta>0.$ By Assumption \ref{Assumption:MinimizingRegularPoints} there exist $T>0,$ a regular boundary point $ z\in\partial D\cap\mathcal{R}$ and $\phi_1\in C([0,T];\e)$ such that $\phi_1(0)=z^*, \phi_1(T)=z$ such that 
      $$  I_{z^*,T}(\phi_1)\leq V(z^*, \partial D)+\eta/3.     $$
      In turn, there exists  $u\in L^2([0,T];H)$ such that $\phi_1=Z_{z^*}^u$ and 
      $$ \frac{1}{2} |u|^2_{L^2([0,T];H)}<V(z^*, \partial D)+2\eta/3.$$
      Since $z^*\in\h_1,$ it follows by regularity properties of the skeleton equation that $z=\phi_1(T)=Z_{z^*}^u(T)\in\h_1\cap\partial D.$ In order to conclude, we use the regularity of $z$ to find an exterior point $z''\in\bar{D}^c$ and a controlled path $\phi_2=Z_z^v$ and $T_1>0$ with $\phi_2(T_1)=Z_z^v(T_1)=z''$ and $\frac{1}{2}|v|_{L^2([0,T];H)}^2<\eta/3.$ Therefore, letting $\phi\in C([0,T+T_1];\e)$ be the concatenation of $\phi_1, \phi_2$ we see that 
      $\phi(0)=z^*, \phi(T+T_1)=z''\in\bar{D}^c$ and
           $$  V(z^*, \bar{D}^c)\leq I_{0,T+T_1}(\phi)< \frac{1}{2}|v|_{L^2([0,T];H)}^2+V(z^*, \partial D)+2\eta/3<  V(z^*, \partial D)+\eta.$$       \end{proof}    
      \subsubsection{\textbf{On regular boundary points}}

      As we have seen, the notion of regular boundary points plays an important role in the form of our exit time and shape asymptotics. To wit, if we fix a bounded, invariant, uniformly attracting set $D$ that satisfies Assumption \ref{Assumption:Dinterior} (e.g. the set of orbits exhibited in \eqref{eq:orbitDomain}), then  Theorems \ref{thm:ExitTimeUpperBnd}(1), \ref{thm:ExitTimeLowerBnd}(2) hold with non-matching upper bounds. If, however, there exists at least one regular boundary point that minimizes the quasipotential to the boundary (so that Assumption 
      \ref{Assumption:MinimizingRegularPoints} holds), then we obtain the asymptotic
      $$\lim_{\epsilon\to 0}\epsilon^2\log\ex[\tau^{\epsilon,z}_D]=V(z^*,\partial D)      $$
      which agrees with the classical results by Freidlin-Wentzell in finite dimensions. Under the same assumption of minimizing regular boundary points, the exit shape concentrates on minimizers of $V_{\bar{D}}$ with high probability. Moreover, as mentioned in Remark \ref{rem:exitshapeldp}, if Lemma \ref{lem:innerregularity} can be shown to hold for $V_D$ in place of $V_{\bar{D}},$ the "rate functions" $J_1,J_2$ (Theorem \ref{thm:exitshapeldp}) are equal.
      
      In this section we attempt to shed more light to this notion. Before we proceed to the main body of our analysis, we emphasize that none of the results in this section rely on the fact that $b$ is a polynomial, Assumption \ref{Assumption:b}. As we shall see below, the analysis here applies to any nonlinearity $b$ that guarantees the existence of bounded uniformly attracting domains $D\subset\e$.

      First, let us introduce a few useful definitions and facts related to differentials of the supremum norm.

\begin{dfn} Let $x_0\in C(0,\ell).$
\begin{enumerate}
   \item The subdifferential $\partial |x_0|_\infty$ of the mapping $C(0,\ell)\ni x\mapsto|x|_\infty\in[0,\infty ), $ at the point $x_0,$ is defined by
$$ \partial |x_0|_\infty:=\big\{ \mu\in \mathcal{M}(0,\ell):  \langle x_0, \mu\rangle_{C(0,\ell), \mathcal{M}(0,\ell) }=|x_0|_{\infty},\; |\mu|_{\mathcal{M}(0,\ell)}=1     \big\}.$$
\item Let $f: \R\rightarrow \R$ be a continuous function.  The right and left  Dini derivatives of $f$ at $x\in\R$ are respectively given by
\begin{equation*}
    \begin{aligned}
        \frac{d^+}{dx}f(x):=\limsup_{h\to 0}\frac{f(x+h)-f(x)}{h},\;\; \frac{d^-}{dx}f(x):=\liminf_{h\to 0}\frac{f(x)-f(x-h)}{h}.
    \end{aligned}
\end{equation*}
\end{enumerate}
\end{dfn}

The following provides a useful characterization for the subdifferential of $|\cdot|_\infty$ and 
chain-rule estimates for Dini derivatives of supremum norms.
 \begin{prop}[\cite{salins2021existence}, Propositions 3.3, 3.5]\label{prop:subdifferential} Let $x_0\in C(0,\ell).$ The following hold:
 \begin{enumerate}
     \item $\mu\in \partial |x_0|_\infty$ if and only if there exists a probability measure $\tilde{\mu},$ supported on $\arg\max\{ |x(\xi)| ; \xi\in (0,\ell)\},$ such that for all $\phi\in C(0,\ell)$
 $$  \langle \mu, \phi \rangle_{C(0,\ell), \mathcal{M}(0,\ell) } = \langle \tilde{\mu}, \phi\cdot\textnormal{sgn}(x_0)\rangle_{C(0,\ell), \mathcal{M}(0,\ell) }.$$    
 \item Let $f:[0, T]\times (0, \ell)\rightarrow\R$ be continuously differentiable in the first argument and such that $f(t,\cdot)\in C(0,\ell) $ for any $t\in[0,T].$ For any $t^-\in(0, T], t^+\in[0, T)$ and $\mu^-\in\partial|f(t^-, \cdot )|_{\infty}$  we have 
  \begin{equation*}
     \begin{aligned}
        \frac{d^-}{dt}|f(t,\cdot)|_\infty\bigg|_{t=t^-}\leq\bigg\langle \frac{\partial f(t,\cdot)}{\partial t} \bigg|_{t=t^-}, \mu^-\bigg\rangle_{C(0,\ell), \mathcal{M}(0,\ell)}
     \end{aligned}
 \end{equation*}
 and 
 \begin{equation}\label{eq:rightDinibound}
     \begin{aligned}
        \frac{d^+}{dt}|f(t,\cdot)|_\infty\bigg|_{t=t^+}=\max_{\mu\in\partial|f(t^+, \cdot )|_\infty}\bigg\langle \frac{\partial f(t,\cdot)}{\partial t} \bigg|_{t=t^+}, \mu\bigg\rangle_{C(0,\ell), \mathcal{M}(0,\ell)}.
     \end{aligned}
 \end{equation}
 \end{enumerate}
 \end{prop}

 The last inequality is a consequence of \cite[Proposition D.4, Equation (D.5)]{da2014stochastic}. In turn, the latter follows from Equation (D.2) of the same reference which is stated without proof. For the sake of completeness, we provide a proof in Appendix \ref{App:Subdifferential}.

 For our first example, we consider cylindrical subsets of $\e.$ Roughly speaking, when the initial velocity $v$ and the initial position $x$ have the same sign, at points where the latter is maximized, there exists a path of zero energy that connects $(x,v)$ to the exterior of the domain. Thus, all boundary points with this property are regular.\\

\begin{example}[Regular boundary points of cylindrical sets: outward pointing velocity]\label{ex:regularculinders}
 Let $R>0, \rho>0$ and consider the domain \begin{equation}\label{eq:cylindricalset}
    \begin{aligned}
      D=B(x^*, R)\times B(0, \rho)=\bigg\{ (x,v)\in C_0(0,\ell)\times C^{-1}(0,\ell):   |x-x^*|_{\infty}< R, |v|_{C^{-1}}<\rho\bigg\}.  
    \end{aligned}
\end{equation}

\noindent The boundary points 
$$\mathcal{B}_{out}:=\bigg\{(x,v)\in C^1_0(0,\ell)\times C_0(0,\ell)\bigg| \;|x-x^*|_{\infty}=R, \exists\mu\in\partial|x-x^*|_{\infty}:\langle v, \mu\rangle>0\bigg\}\subset \partial D $$
with "outward pointing" velocity are regular. To this end let $z=(x,v)\in\mathcal{B}_{out}.$

From Proposition \ref{prop:subdifferential}, there exists a probability measure $\tilde{\mu}$ supported on $\arg\max|x-x^*|_\infty$ such that $$\int_{\arg\max|x-x^*|} v(\xi)\textnormal{sgn}\big(x(\xi)-x^*(\xi)\big)\;d\tilde{\mu}(\xi)>0.$$
Hence, there must exist $v_0>0$ and $\xi_0\in \arg\max|x-x^*|$ such that $$v(\xi_0)\textnormal{sgn}\big(x(\xi_0)-x^*(\xi_0)\big)>v_0.$$
Noting that, for each $z\in \mathcal{B}_{out},$ $Z^0_z\in C([0,\infty);C^1_0(0,\ell)\times C_0(0,\ell)) $ we let
$$\zeta:=\sup_{ t>0}|\Pi_1Z^0_z(t)|_{\infty}+\sup_{t>0}|\Pi_2Z^0_z(t)|_{\infty}<\infty    $$
and 
$\zeta_b:=\sup_{|x|\leq \zeta}|b(x)|+\alpha\zeta. $ Here, we are using that $D$ is uniformly  attracting in $\e$ which follows e.g. from Theorem \ref{thm:domainofattraction} (at least for $\rho, R>0$ small enough). The fact that the $C_0-$norm of the velocity is also uniformly bounded for all $t\geq 0$ is straightforward to show from a similar argument.

 From the continuity of $v$ at $\xi_0,$ there exists $t_0>0$ such that for all $|y-\xi_0|<t_0$ 
 \begin{equation}\label{eq:vcontinuity}
 |v(y)-v(\xi_0)|<v(\xi_0)/2.
 \end{equation}
For
$t\leq T_0:=\min\{t_0, v_0/2\zeta_b, 
 \xi_0, \ell-\xi_0\}$ and from a first-order Taylor approximation, there exist   $\delta,  \delta'\in(0, T_0)$ such that 
\begin{equation*}\label{eq:1storderTaylor}
\begin{aligned}     &x(\xi_0+t)=x(\xi_0)+x'(\xi_0+\delta)t, \\&
     x(\xi_0-t)=x(\xi_0)-x'(\xi_0-\delta')t.
\end{aligned}
\end{equation*}

\noindent From the latter and the choice of $\xi_0$ we obtain 
\begin{equation*}
\begin{aligned}
     \Pi_1Z^0_z(\xi_0,t)-x^*(\xi_0)&=\frac{1}{2}\bigg(x(\xi_0+t) + x(\xi_0-t)    \bigg)-x^*(\xi_0)+\frac{1}{2}  \int_{\xi_0-t}^{\xi_0+t}v(y)dy\\&+\frac{1}{2}\int_0^{t}\int_{\xi_0-t+s}^{\xi_0+t-s}[b(\Pi_1Z^0_z(y, s))+\alpha\Pi_2Z^0_z(y,s)]dyds \\&=   |x-x^*|_\infty\textnormal{sgn}\big(x(\xi_0)-x^*(\xi_0)\big)+\frac{t}{2}\bigg(x'(\xi_0+\delta)-x'(\xi_0-\delta) \bigg)                    \\&+\frac{1}{2}  \int_{\xi_0-t}^{\xi_0+t}v(y)dy+\frac{1}{2}\int_0^{t}\int_{\xi_0-t+s}^{\xi_0+t-s}[b(\Pi_1Z^0_z(y, s))+\alpha\Pi_2Z^0_z(y,s)]dyds.
\end{aligned}
\end{equation*}

Moreover, 
\begin{equation*}\label{eq:zetabbnd}
    \begin{aligned}
        \bigg|\int_0^{t}&\int_{\xi_0-t+s}^{\xi_0+t-s}[b(\Pi_1Z^0_z(y, s))+\alpha\Pi_2Z^0_z(y,s)]dyds\bigg|\\&\leq \bigg(\sup_{|x|\leq \zeta}|b(x)|+\alpha\sup_{ s>0}\sup_{\xi\in(0,\ell)}\big|\Pi_2Z^0_z(y,s)dy\big|_{\infty}   \bigg)\int_{0}^{t} 2(t-s)ds\\&\leq \zeta_b(2t^2-t^2)
        =\zeta_bt^2.
    \end{aligned}
\end{equation*}

\noindent\underline{\textit{Case 1}}:  $x(\xi_0)-x^*(\xi_0)>0.$ From the last two displays along with \eqref{eq:vcontinuity} we obtain the lower bound
\begin{equation*}
\begin{aligned}
     |\Pi_1Z^0_z(t)-x^*|_{\infty}&\geq  R+\frac{t}{2}\bigg(x'(\xi_0+\delta) -x'(\xi_0-\delta) \bigg)\\&+\frac{1}{4}\int_{\xi_0-t}^{\xi_0+t}\bigg[v(\xi_0)\textnormal{sgn}\big(x(\xi_0)-x^*(\xi_0)\big)\bigg]d\xi-\frac{\zeta_b t^2}{2}\\&\geq 
     R+\frac{t}{2}\bigg(x'(\xi_0+\delta) -x'(\xi_0-\delta) \bigg)+\frac{v_0 t}{2}-\frac{\zeta_b t^2}{2}.
\end{aligned}
\end{equation*}
Since $x'$ is continuous we can (by possibly reducing $t_0$ and remembering that $\delta<t_0$) take 
$$\frac{t}{2}\bigg(x'(\xi_0+\delta)-x'(\xi_0-\delta)\bigg)\geq -v_0t/4    $$
so that
\begin{equation*}
\begin{aligned}
     |\Pi_1Z^0_z(t)&-x^*|_{\infty}\geq  R+tv_0/4- \zeta_bt^2/2. 
\end{aligned}
\end{equation*}
From the choice of $T_0$ we have $t<v_0/2\zeta_b$ and
$tv_0/4-\zeta_bt^2/2>0.$ In turn, the latter imply 
$$|\Pi_1Z^0_z(t)-x^*|_{\infty}>R.$$

\noindent\underline{\textit{Case 2}}:  $x(\xi_0)-x^*(\xi_0)<0.$ From the symmetric lower bound 
\begin{equation*}
     \Pi_1Z^0_z(\xi_0,t)-x^*(\xi_0)\leq -R-\frac{v_0t}{2}+\frac{t}{2}\bigg(x'(\xi_0-\delta)-x'(\xi_0-\delta)\bigg)+ \frac{\zeta_bt^2}{2}
\end{equation*}
and the continuity of $x'$ we obtain 
\begin{equation*}
     \Pi_1Z^0_z(\xi_0,t)-x^*(\xi_0)\leq -R-\frac{v_0t}{8}- \frac{\zeta_bt^2}{2}.
\end{equation*}
As in the previous case, the choice of $T_0$ implies
$$\Pi_1Z^0_z(t, \xi_0)-x^*(\xi_0)<-R\implies |\Pi_1Z^0_z(t)-x^*|_{\infty}>R. $$
Hence, in  both cases and for all $t\in (0,T_0),$ we have
$$\phi(t):=Z^0_z(t)\notin D.$$ Therefore, according to Definition \ref{def:regularpoints}, $z\in \mathcal{B}_{out}$ is indeed regular. Since no control is required to exit $D$ instantaneously, we have $I_{z,T_0}(\phi)=0.$
\end{example}

In contrast to the previous example, boundary points for which the velocity and position have opposite signs at all maximizers of the position are irregular.\\

\begin{example}[Irregular boundary points of cylindrical sets: inward pointing velocity]\label{ex:irregularPointscylinders}
     Let $D$ as in \eqref{eq:cylindricalset} and consider the set

$$\mathcal{B}_{in}:=\bigg\{(x,v)\in C^1_0(0,\ell)\times C_0(0,\ell)\bigg| \;|x-x^*|_{\infty}=R, \forall\mu\in\partial|x-x^*|_{\infty}:\langle v, \mu\rangle< 0\bigg\}\subset \partial D $$
of boundary points with an "inward-pointing" velocity. We claim that $\mathcal{B}_{in}$ consists of irregular boundary points. To this end,  let $z=(x,v)\in\mathcal{B}_{in}$ and fix $T>0$ and a control $u\in L^2([0, T];H).$ Since $\partial_t( \Pi_1Z^u_z)\in H,$  \eqref{eq:rightDinibound} cannot be used directly to estimate the supremum norm of $\Pi_1Z^u_z(t)-x^*.$ At $\xi=\xi_0$ the latter satisfies

\begin{align*}
    \Pi_1 Z^u_z(\xi_0,t) - x^*(\xi_0) = &[\Pi_1 S_0(t)z](\xi_0) -x^*(\xi_0) \nonumber\\
    &+\int_0^t \Pi_1 S_\alpha(t-s) \begin{pmatrix} 0 \\  b(\Pi_1 Z^u_z(s)) \end{pmatrix}ds\nonumber\\
    &+\int_0^t \Pi_1 S_\alpha(t-s) \begin{pmatrix} 0 \\ \sigma(\Pi_1 Z^u_z(y,s))u(s)\end{pmatrix}ds.
\end{align*}

The controlled convolution term is not as regular as the other two terms, so we subtract it off by letting
\begin{equation*}
    q(t) = \Pi_1 Z^u_z(t) - \int_0^t \Pi_1 S_\alpha(t-s)\begin{pmatrix} 0 \\\sigma(\Pi_1 Z^u_z(s)) u(s) \end{pmatrix}ds.
\end{equation*}
$q(\xi,t)$ is differentiable in $t$ and satisfies
$\frac{\partial q}{\partial t} \Big|_{t=0} = v = \Pi_2 z$. 

We estimate the derivative of the supremum norm of $q$ with \eqref{eq:rightDinibound}.
\begin{equation*}
    \frac{d^+}{dt}|q(t) - x^*|_\infty \leq \max_{\mu_0 \in \partial |x-x^*|_\infty}\left< \frac{\partial q}{\partial t}(t,\cdot),\mu\right>=:-\gamma<0. 
\end{equation*}
This limit is negative because the initial data $z$ was assumed to be in $\mathcal{B}_{in}$. Therefore, we can find $t_0$ such that for all $t \in [0,t_0]$,
$$|q(t) - x^*|_{L\infty} \leq R-\frac{\gamma}{2} t.$$

The control term can be estimated using arguments similar to Lemma \ref{lem:integral-bound}. We do the same transformation from Section \ref{Sec:Linfty-decay}. Let $\Psi(t) = e^{-\frac{\alpha t}{2}}\int_0^t S_\alpha(t-s) \begin{pmatrix} 0\\ \sigma(\Pi_1 Z^u_z(y,s))u(s)\end{pmatrix} ds$. Then $\Psi$ satisfies
\begin{align*}
    \Psi(\xi,t) = \frac{\alpha^2}{4}\int_0^t \int_{\xi-(t-s)}^{\xi + (t-s)} \Psi(y,s)dyds + \int_0^t \int_{\xi-(t-s)}^{\xi + (t-s)}\sigma(\Pi_1 Z^u_z(y,s))u(y,s)dyds.
\end{align*}
By H\"older's inequality, uniformly for $\xi \in [0,\ell]$,
\begin{align*}
    &\left|\int_0^t \int_{\xi-(t-s)}^{\xi + (t-s)}\sigma(\Pi_1 Z^u_z(y,s))u(y,s)dyds \right|
    \leq \left(\int_0^t \int_{\xi-(t-s)}^{\xi+(t-s)}dyds \right)^{\frac{1}{2}}|\sigma|_{\infty}|u|_{L^2([0,t]\times[0,\ell])}\nonumber\\
    &\leq t|u|_{L^2([0,t]\times[0,\ell])}.
\end{align*}

The linear term is bounded by
\begin{equation*}
    \sup_{\xi ]in [0,\ell]} \left|\frac{\alpha^2}{8}\int_0^t \int_{\xi-(t-s)}^{\xi + (t-s)} \Psi(y,s)dyds \right| \leq \frac{\alpha^2 t}{8}\int_0^t \sup_{\xi \in [0,\ell]} |\Psi(\xi,s)ds.
\end{equation*}
Gr\"onwall's inequality implies that for small $t>0$,
\begin{equation*}
    |\Psi(t)|_\infty \leq t |u|_{L^2([0,t]\times[0,\ell])}e^{\frac{\alpha^2 t^2}{8}}.
\end{equation*}
Therefore, for small $t>0$,
\begin{equation*}
    \left|\int_0^t S_\alpha(t-s) \begin{pmatrix} 0\\\sigma(\Pi_1 Z^u_z(s))u(s)\end{pmatrix} ds\right|_\infty \leq Ct |u|_{L^2([0,t]\times[0,\ell])}.
\end{equation*}

Choose $t_0$ small enough so that $|u|_{L^2([0,t_0]\times[0,\ell])}\leq \frac{\gamma}{4}$.

These estimates prove that for all $t\in[0,t_0]$,
\begin{equation*}
    |\Pi_1 Z^u_z(t) - x^*|_\infty \leq  R - \frac{\gamma}{4} t <R.
\end{equation*}

We showed that all controlled trajectories issued from $z\in\mathcal{B}_{in}$ will not exit $D$ instantaneously. Per Definition \ref{def:regularpoints}, all points in $\mathcal{B}_{in}$ are not regular.
\end{example}

We conclude the discussion on boundary points of cylindrical sets by considering the case in which the initial velocity is zero on all the maximizing points of the initial position vector. As we shall show below, such points are regular per Definition \ref{def:regularpoints}.

\begin{example}[Regular boundary points of cylindrical sets: "orthogonal" initial velocity]

 Let $D$ as in \eqref{eq:cylindricalset} and consider the set
 $$\mathcal{B}_{\perp}:=\bigg\{(x,v)\in C^2_0(0,\ell)\times C^1_0(0,\ell)\bigg| \;|x-x^*|_{\infty}=R, \exists\mu\in\partial|x-x^*|_{\infty}:\langle v, \mu\rangle=0\bigg\}\subset \partial D. $$
 
 \noindent Let $T>0, \xi_0\in\arg\max|x-x^*|_{\infty},$  $u: C^1([0,T];C^2_0(0,\ell))\rightarrow C((0,\ell)\times [0,T])$ a control function
 that will be specified below. Any controlled trajectory issued from $z=(x,v)\in \mathcal{B}_{\perp}$ satisfies $\phi=Z^u_z\in C^2_0(0,\ell)\times C^1_0(0,\ell).$ A temporal second-order Taylor expansion around $t=0$ thus furnishes
 \begin{equation*}
 \begin{aligned}
     \Pi_1\phi(\xi_0, t)&=\Pi_1 \phi(\xi_0, 0)+\Pi_2 \phi(\xi_0, 0)t+\frac{t^2}{2}\partial_t^2 \Pi_1 \phi(\xi_0, t')\\&
     = x(\xi_0)+v(\xi_0)t
    \\&+\frac{t^2}{2} \bigg(-\alpha\partial_t\Pi_1\phi(\xi_0,t')+\partial_\xi^2\Pi_1\phi(\xi_0,t')+b\big(\Pi_1\phi(\xi_0, t')\big)+\sigma\big(\Pi_1\phi(\xi_0,t')\big)u\big(\Pi_1\phi\big)(\xi_0,t')\bigg)
 \end{aligned}
 \end{equation*}
 for any $t\in[0,T]$ and some $t'\in (0,t).$ For some $\rho>0$ we choose
 $$u(\psi)(\xi,s):=\frac{1}{\sigma(\psi(\xi,s))}\bigg(\rho+\alpha\partial_t\psi(\xi,s)-\partial_\xi^2\psi(\xi,s)-b\big(\psi(\xi,s)\big)      \bigg), (\xi, s)\in (0,\ell)\times(0, t) $$    
and note  that $v(\xi_0)=0.$ Thus,

 \begin{equation*}
 \begin{aligned}
     \Pi_1\phi(\xi_0, t)-x^*(\xi_0)&=x(\xi_0)-x^*(\xi_0)+\rho t^2/2
 \end{aligned}
 \end{equation*}
 and assuming, without loss of generality, that $x(\xi_0)-x^*(\xi_0)>0$ we deduce that 
 $$ | \Pi_1\phi(t)-x^*|_{\infty}\geq \Pi_1\phi(\xi_0, t)-x^*(\xi_0)=|x-x^*|_{\infty}+\rho t^2/2> R.$$
Clearly, the same conclusion follows from a symmetric bound when $x(\xi_0)-x^*(\xi_0)<0.$ Finally, the energy of this path is given by
\begin{equation*}
    \begin{aligned}
        I_{z, t}(\phi)&=\frac{1}{2}|u(\Pi_1\phi)|^2_{L^2[0,t]\times[0,\ell]}\\&
              =\frac{1}{2}\int_0^t\int_0^\ell\bigg|\frac{1}{\sigma(\Pi_1\phi(\xi,s))}\bigg(\rho+\alpha\partial_t\Pi_1\phi(\xi,s)-\partial_\xi^2\Pi_1\phi(\xi,s)-b\big(\Pi_1\phi(\xi,s)\big)      \bigg)  \bigg|^2d\xi ds\\&
        \\&\leq C_\sigma^{-1}\bigg(\rho^2\ell+\sup_{(\xi,s)\in[0,\ell]\times[0,T]}\bigg| \alpha\partial_t\Pi_1\phi(\xi,s)-\partial_\xi^2\Pi_1\phi(\xi,s)-b\big(\Pi_1\phi(\xi,s)\big)\bigg|^2\ell  \bigg)t\\&
        =: C' t,
    \end{aligned}
\end{equation*}
where $C_\sigma$ is the constant from Assumption \ref{Assumption:sigmaNondeg}.

Letting $\delta>0$ and $t<\delta/C'$ we conclude from the last two displays that the path $\phi$ has energy smaller than $\delta, \phi(0)\in\mathcal{B}_{\perp}$ and $\phi(s)\notin \bar{D}^c$ for all $s\in (0, t].$ Hence any point in $\mathcal{B}_\perp$ is regular per Definition \ref{def:regularpoints}.

\end{example}

For the last example, we turn our attention to spherical subsets of $\e.$ In this case, we show that points with zero velocity are regular, provided that the position component is "locally flat" in a neighbourhood of a maximizer.\\

\begin{example}[Regular boundary points of spherical sets]\label{ex:regularspheres}
     Let $R>0$ and 
$$D:=\bigg\{ (x, v)\in\mathcal{E}:\sqrt{ |x-x^*|^2_{\infty}+ |v|^2_{C^{-1}}}<R\bigg\}. $$
 We shall now show that, under Assumption \ref{Assumption:sigmaNondeg}, the set  
 \begin{equation*}
     \begin{aligned}
         \mathcal{B}_{flat}:=\bigg\{(x, 0)\in C_0^2(0,\ell)\bigg|\; |x-x^*|_{\infty}=R,\;\exists \;t_0>0, \xi_0\in\arg\max|x-x^*|_{\infty}: x-x^*=R\;\textnormal{on}\;(\xi_0-t_0, \xi_0+t_0)  \bigg\}
     \end{aligned}
 \end{equation*}
 is regular. Indeed, for $z=(x,0)\in \mathcal{B}_{flat}$ and from a 
second-order Taylor approximation, there exist $\delta_1, \delta_2<t_0$ such that
 \begin{equation*}
\begin{aligned}
      &x(\xi_0+t)=x(\xi_0)+ x'(\xi_0)t+x''(\xi_0+\delta_1)\frac{t^2}{2}=x(\xi_0)+x''(\xi_0+\delta_1)\frac{t^2}{2}, \\&
     x(\xi_0-t)=x(\xi_0)-x'(\xi_0)t+x''(\xi_0-\delta_2)\frac{t^2}{2}=x(\xi_0)+x''(\xi_0-\delta_2)\frac{t^2}{2}.
\end{aligned}
\end{equation*}
 Thus, for $t<t_0,$ 
 \begin{equation*}
     \begin{aligned}
     \Pi_1Z^u_z(\xi_0,t)-x^*(\xi_0)&=\frac{1}{2}\bigg(x(\xi_0+t)+  x(\xi_0-t)    \bigg)-x^*(\xi_0)+\frac{1}{2}\int_0^{t}\int_{\xi_0-t+s}^{\xi_0+t-s}b(\Pi_1Z^u_z(y, s))dyds\\&
     +\frac{1}{2}\int_0^{t}\int_{\xi_0-t+s}^{\xi_0+t-s}\sigma(\Pi_1Z^u_z(y, s))u(y,s)dyds\\&
      =   R+\frac{t^2}{4}\bigg(x''(\xi_0+\delta_1)+x''(\xi_0-\delta_2) \bigg)+\frac{1}{2}\int_0^{t}\int_{\xi_0-t+s}^{\xi_0+t-s}b(\Pi_1Z^u_z(y, s))dyds\\&
      +\frac{1}{2}\int_0^{t}\int_{\xi_0-t+s}^{\xi_0+t-s}\sigma(\Pi_1Z^u_z(y, s))u(y,s)dyds.
\end{aligned}
\end{equation*}
For a constant $V\in\R$ to be specified later, we consider the feedback control  
$$u(x):=\frac{V-b(x)}{\sigma(x)}.     $$
The corresponding controlled dynamics then satisfy
\begin{equation*}
     \begin{aligned}
     \Pi_1Z^u_z(\xi_0,t)-x^*(\xi_0)&
      =  
   R+\bigg(V+x''(\xi_0+\delta_1)+x''(\xi_0-\delta_2)  \bigg)\frac{t^2}{4}.
\end{aligned}
\end{equation*}
For some $\eta>0$ we choose
$$V:=\eta-\bigg(x''(\xi_0+\delta_1)+x''(\xi_0-\delta_2)\bigg) $$ to obtain the lower bound
\begin{equation*}
     \begin{aligned}
     |\Pi_1Z^u_z(t)-x^*|_\infty&
      \geq   R+\eta t^2>R 
\end{aligned}
\end{equation*}
which holds for all $t\in(0,t_0).$ Hence the path $\phi=\Pi_1Z^u_z$ starting from the boundary point $z$ exits $\bar{D}$ instantaneously and stays in $\bar{D}^c$ for all $t<t_0.$
Moreover, for any $\delta>0,$ $C_\sigma$ as in Assumption \ref{Assumption:sigmaNondeg}
$$\kappa:= \frac{1}{2}\bigg(\frac{V+|b|_\infty}{C_\sigma}\bigg)^2\ell$$
and
$$t<\bigg(t_0\wedge\frac{\delta}{\kappa}   \bigg)$$
the energy of this path satisfies
\begin{equation*}
    I_{z,t}(\phi)\leq \frac{1}{2}\int_{0}^{t}\int_{0}^{\ell}|u(y, s)|^2dyds\leq\kappa t<\delta.
\end{equation*}

Therefore, any boundary point in $\mathcal{B}_{flat}$ with the properties described above is regular per Definition \ref{def:regularpoints}.
\end{example}

We conclude with a remark on the non-invariance of cylindrical sets.

\begin{rem}[Non-invariance of cylinder sets]\label{rem:cylinderNonInvariance} From Example 1 it follows that cylindrical sets $D$ of the form \eqref{eq:cylindricalset} cannot be invariant under the deterministic flow of the damped wave equation. Indeed, we showed that any (un-controlled) trajectory issued from $\mathcal{B}_{out}$ exits $D$ in finite time. Nevertheless, the exit time upper bound (Theorem \ref{thm:exitshapeldp}(1)) holds with growth rate $V(z^*, \bar{D}^c).$ According to Assumption \eqref{Assumption:MinimizingRegularPoints}, if the infimum of $\partial D\ni z\mapsto V(z^*,z)\in[0,\infty]$ is attained at some $z\in\mathcal{B}_{out},$ then $V(z^*, \bar{D}^c)=V(z^*, \partial D).$ Of course, as mentioned in Proposition \ref{corr:DExistence}, invariance can be regained by considering domains consisting of all the orbits issued from $D.$ 
\end{rem}

\section{Well-posedness and exit problems for global solutions in $\e$}\label{sec:GlobalSolutions}

The results of the previous sections are formulated for local $\e-$valued mild solutions of \eqref{eq:model}. In particular, we proved local well-posedness for (controlled) wave equations with polynomial nonlinearities $b$ of arbitrary degree and locally Lipschitz noise coefficients $\sigma$ (Theorem \ref{prop:WellPosedness}), identified bounded domains $D\subset\e$ of uniform attraction for the noiseless dynamics to an asymptotically stable equilibrium $z^*$ (Section \ref{sec:NonlinearStability}), formulated and proved a suitable local ULDP for local solutions (Section \ref{sec:LULDP}). Finally, we studied Freidlin-Wentzell asymptotic exit problems from $D$ (Section \ref{Sec:Metastability}) and obtained asymptotic lower bounds for explosion times of such local solutions (Remark \ref{rem:explosiontimes}).

As we commented at the introduction of the paper, what allows us to be able to focus on local solutions, rather than the more restrictive case of global solutions, is the fact that we study the exit problem from a bounded domain, $D\subset\e$ of attraction of the noiseless dynamics. For completeness, in this section, we discuss the case of \eqref{eq:model} for which global solutions do exist. Clearly for this to be the case, stronger assumptions on $b$ and $\sigma$ have to be imposed, see Assumptions \ref{Assumption:bGlobal}-\ref{Assumption:sigmaGlobal} below.

The main novelty of this section is that under Assumptions \ref{Assumption:bGlobal}-\ref{Assumption:sigmaGlobal}, we provide a global well-posedness result on the space $\e$ for the controlled equation \eqref{eq:controlledequation}, see Proposition \ref{prop:WellPosednessGlobal}. Similar results do exist in the literature, but with global well-posedness proven in different spaces. In particular, similar global well-posedness results can be found e.g. in \cite[Proof of Theorem 5.3]{cerrai2003stochastic} and \cite[Theorem 4.2]{cerrai2006smoluchowski}. However, these works are only concerned with $\h=H\times H^{-1}-$valued solutions. Since here we are interested in solutions on the smaller space $\e$ we present below a slightly modified proof and only highlight the differences from the aforementioned results. The assumption that $b$ (from Assumption \ref{Assumption:bGlobal} below) has at most cubic growth is crucial for achieving the necessary bounds in the case of global solutions. We believe that this result is of independent interest.

We mention here that the controllability and quasipotential-regularity results of Section \ref{sec:Controllability} hold true independently of the notion of solutions. Indeed, these results rely on the dynamics starting at $z=z^*$ and the existence of bounded domains of attraction. Moreover, the asymptotics for exit times and places from Section \ref{Sec:Metastability} remain unchanged, provided that the local ULDP (and the corresponding rate function) is replaced by a classical (global) ULDP for global solutions. Thus, in the rest of this section we shall exclusively focus on adapting the results of Section \ref{sec:wellposedness} to global solutions of \eqref{eq:model}.

Throughout this section we shall replace Assumptions \ref{Assumption:b}, \ref{Assumption:sigma} respectively by 

 \begin{customthm}{1'}\label{Assumption:bGlobal}     
      \begin{enumerate}
          \item[a)] The function $b\in C^1(\R)$ and we can write $$b=b_1+b_2,$$ where $b_1$ is globally Lipschitz continuous with Lipschitz constant $L_1>0$ and for some  $\lambda>1$ and a constant $C_1>0$ we have, for all $x\in\R,$
	   	\begin{equation*}\label{bgrowth}
	   	|b_2(x)|\leq C_1 \big(1+|x|^\lambda\big)\;,\;\;|b_2'(x)|\leq C_1 \big(1+|x|^{\lambda-1}\big)\;,\;\; b_2'(x)\leq 0.
	   	\end{equation*}
	   	Moreover, the antiderivative $$\R\ni x\longmapsto \beta(x):=\int_{0}^{x}b(y)dy\in\R$$ satisfies
	   		\begin{equation}\label{antidercond}
	   \beta(x)\leq C_2\big(1-|x|^{\lambda+1}\big)
	   	\end{equation}
	   	for some constant $C_2>0.$
     \item[b)] The growth exponent of $b_2$ satisfies $\lambda\in(1,3]$
      \end{enumerate}
       \end{customthm}

 \begin{customthm}{2'}\label{Assumption:sigmaGlobal} The function $\sigma:\R\rightarrow\R$ is bounded and globally Lipschitz continuous with Lipschitz constant $L_{\sigma}>0.$
   \end{customthm}
 
 The rest of the assumptions that we imposed in the previous section remain unchanged. Proposition \ref{prop:WellPosednessGlobal} guarantees global well-posedness in $\e$ for the controlled equation \eqref{eq:controlledequation}. 

   \begin{prop}[Global well-posedness of controlled equation] \label{prop:WellPosednessGlobal} Under Assumptions \ref{Assumption:bGlobal}, \ref{Assumption:sigmaGlobal},  \eqref{eq:controlledequation} admits a unique global mild solution.    
 \end{prop}

 \begin{proof}  It suffices to prove well-posedness for $\epsilon=1.$ To this end, we recall the localization $b_n$ \eqref{eq:bnlocalization}. As in Proposition \ref{prop:WellPosedness} we consider the localized problem
 \begin{equation}\label{eq:controlledquationLipGlobal}
 	\begin{aligned}
 		dZ^{h}_n(t)=A_\alpha Z^{h}_n(t)dt+B_n(Z_n^{ h}(t))dt+\Sigma\big(Z^{h}_n(t)\big)h(t)dt+\Sigma\big(Z^{h}_n(t)\big)dW(t)\;, Z^{h}_n(0)=z\in\e
 	\end{aligned}
 \end{equation}
which, from the same proposition, has a unique solution $Z^{h}_n\in L^p(\Omega;     C([0,T_0]; \e).$ Next we consider the family of stopping times 
\begin{equation*}\label{eq:localizingtimesGlobal}
	\tau_{n}:=\inf\{t>0 : |\Pi_1 Z^{h}_{n}(t)|_{\infty}\geq n\}
\end{equation*}
and aim to show that
$$ \pr\bigg( \tau:=	\sup_{n\in\N}\tau_n<\infty\bigg)=\lim_{T\to\infty}\lim_{n\to\infty}\pr(\tau_n\leq T)=0.$$
This in turn implies that \eqref{eq:controlledequation} admits a unique, global mild solution per Definition \ref{dfn:mild solutions}. 
 To this end, we pass to the mild formulation

    	   \begin{equation*}
    	   \begin{aligned}
    	   \Pi_1Z_{n}^{h}(t)&=\Pi_1S_\alpha(t)z+\int_{0}^{t}\Pi_1S_\alpha(t-s)B_n(Z_{n}^{h}(s))ds+\int_{0}^{t}\Pi_1S_\alpha(t-s)\Sigma(Z_{n}^{h}(s))h(s)ds\\&+\Pi_1\int_{0}^{t}S_\alpha(t-s)\Sigma(Z_{n}^{h}(s))dW(z)\\&
    	   =:\Pi_1S_\alpha(t)z+\Pi_1\rho_{n}(t)+\Pi_1H_{n}(t)+\Pi_1\Gamma_{n}(t).
    	   \end{aligned}
    	   \end{equation*}
    	  In view of Lemma \ref{lem:StochasticConvolutionApriori} and using the boundedness of $\sigma$ (Assumption \ref{Assumption:sigmaGlobal}) we have the estimate
    	   \begin{equation}\label{convobnd}
    	   \sup_{n\in\N}\ex|\Pi_1\Gamma_{n}|^p_{\infty}\leq C.
    	   \end{equation}
    	   Moreover, from the continuity of $S$ in $\h_{1}$ and Theorem \ref{thm:Linftydecay} we have 
    	   \begin{equation}\label{Ubnd}
    	   \begin{aligned}
    	   |\Pi_1H_{n}(t)|_{H^1}&\leq \int_{0}^{t}\big| S_\alpha(t-s)\Sigma(Z_{n}^{h}(s))h(s)   \big|_{\h_1}ds
    	   \\&
    	   \leq C\int_{0}^{t}e^{-\theta(t-s)}\big|\Sigma(Z_{n}^{h}(s))h(s)   \big|_{\h_1}ds\\&=C\int_{0}^{t}e^{-\theta(t-s)}\big|\sigma(\Pi_1Z_{n}^{h}(s))h(s)   \big|_{L^2}ds\leq C_T|\sigma|_{\infty}|h|_{L^2([0,T];H)}
    	    \end{aligned}
    	   \end{equation}
    	 and 
    	  \begin{equation}\label{zbnd}
    	 \begin{aligned}
    	    |\Pi_1S_\alpha(t)z|_{\infty}\leq Ce^{-\theta t}|z|_{\e}.
    	 \end{aligned}
    	 \end{equation}
    	   The term $\Pi_1\rho_{n}$ solves the nonlinear wave equation
    	    \begin{equation*}\label{rhoeq}
    	   \left\{\begin{aligned}
    	   &\partial_t^2\Pi_1 \rho_{n}(t,\xi)=
    	   \partial_\xi^2 \Pi_1\rho_{n}(t,\xi)-\alpha\partial_t\Pi_1 \rho_{n}(t,\xi) +b_n\big(\Pi_1 Z_{n}^{h}(t,\xi)\big)\;,\;\;(t,\xi)\in [0,\infty)\times(0,\ell)
    	   \\&\Pi_1\rho_{n}(0,\xi)=0, \partial_t \Pi_1\rho_{n}(0,\xi)=0,\;  \xi\in(0,\ell),\; \Pi_1\rho_{n}(t,\xi)=0,\; (t,\xi)\in[0,\infty)\times\{0,\ell\}.
    	   \end{aligned}\right.
    	   \end{equation*}
    	   Multiplying throughout by $\partial_t\Pi_1 \rho_{n}$ and integrating over $[0,\ell]$ we obtain the estimate
    	   
    	   \begin{equation*}
    	  \begin{aligned}
    	   \frac{d}{dt}\big|\partial^2_t\Pi_1 \rho_{n}(t)\big|^2_{H}&+	\frac{d}{dt}\big| \Pi_1\rho_{n}(t)\big|
    	   ^2_{H^1}+2\alpha \big|\partial_t\Pi_1 \rho_{n}(t)\big|^2_{H}\\&\leq 2C_2\frac{d}{dt}\int_{0}^{\ell}\hat{\beta}_n\big( \Pi_1\rho_n(\xi, t) \big)d\xi\\&+\big|b_n\big(\Pi_1 \rho_n(t)+\Pi_1S_\alpha(t)z+\Pi_1H_{n}(t)+\Pi_1\Gamma_{n}(t)\big)-b_n\big(\Pi_1 \rho_n(t)\big)
    	   \big|^2_{H}+\big|  \partial_t\Pi_1\rho_{n}(t) \big|^2_{H}
    	   \end{aligned}
    	   \end{equation*}
    	   which holds for any $n\in\N$ and $\hat{\beta}_n$ is a Lipschitz approximation of the function
    	   \begin{equation}
    	        -\hat\beta(x):=1-\beta(x)/C_2\geq |x|^{\lambda+1},\label{Eq:betahat}
    	        \end{equation}
    	        $\lambda \in(1,3]$
    	   and $\beta, C_2$ are given in \eqref{antidercond}. Now, for $n\geq n_0$ sufficiently large we can proceed as in \cite[pp. 678-679 ]{cerrai2006smoluchowski} to obtain 
    	   \begin{equation*}
    	   \begin{aligned}
    	  \frac{d}{dt}\big|\partial^2_t\Pi_1 \rho_{n}(t)\big|^2_{H}&+	\frac{d}{dt}\big| \Pi_1\rho_{n}(t)\big|^2_{H^1}+2\alpha \big|\partial_t\Pi_1 \rho_{n}(t)\big|^2_{H}\\&-2C_2\frac{d}{dt}\int_{0}^{\ell}\hat{\beta}_n\big(\Pi_1 \rho_n(\xi, t) \big)d\xi\leq C\bigg[1+(\Lambda_n)^{2\lambda}   \bigg]+C(\Lambda_n)^2\bigg(-\int_{0}^{\ell}\hat{\beta}_n\big( \Pi_1\rho_n(\xi, t) \big)d\xi\ \bigg),
    	   \end{aligned}
    	   \end{equation*}
    	   
    	   \noindent where $$ \Lambda_n=\sup_{t\in[0,T]}|\Pi_1S_\alpha(t)z|_{\infty}+\sup_{t\in[0,T]}|\Pi_1H_{n}(t)|_{\infty}+\sup_{t\in[0,T]}|\Pi_1\Gamma_{n}(t)|_{\infty}.$$
    	   In view of Assumption \ref{Assumption:bGlobal} and in particular the fact that 
    	   $$\sup_{n\in\N}|b_n(x)|\leq c(1+|x|^{\lambda}),$$
    	   we can integrate the previous inequality over $t\in[0,T]$ and apply Gr\"onwall's lemma to  obtain
    	     \begin{equation*}
    	   \begin{aligned}
    	  \sup_{t\in[0,T]}\bigg(\big|\partial^2_t\Pi_1 \rho_{n}(t)\big|^2_{H}&+\big| \Pi_1\rho_{n}(t)\big|^2_{H^1}-2C_2\frac{d}{dt}\int_{0}^{\ell}\hat{\beta}_n\big( \Pi_1\rho_n(\xi, t) \big)d\xi\bigg)
    	  +\int_{0}^{T}\big|\partial_t\Pi_1 \rho_{n}(t)\big|^2_{H}dt\\&\leq C\bigg[1+(\Lambda_n)^{2\lambda}   \bigg]\exp\bigg( c(\Lambda_n)^{2}T   \bigg),
    	   \end{aligned}
    	   \end{equation*}
    	   where the constants are independent of the initial conditions and $n$. In particular,
    	   \begin{equation*}\label{rhobnd}
    	   \big| \Pi_1\rho_{n}(t)\big|^2_{\infty}\leq c\big|  \Pi_1\rho_{n}(t)\big|^2_{H^1}\leq     	   C\bigg[1+(\Lambda_n)^{2\lambda}   \bigg]\exp\bigg( c(\Lambda_n)^{2}T   \bigg).
    	   \end{equation*}
    	   Thus, for any $n\geq n_0$ we have 
    	   \begin{equation*}
    	   \begin{aligned} 
    	   \pr\bigg[  \sup_{t\in[0,T]}\big|   \Pi_1Z_{n}^{h}(t)\big|_{\infty}\geq n   \bigg]&\leq \pr\bigg[ \sup_{t\in[0,T]}|\Pi_1\Gamma_{n}(t)|_{\infty}\geq n/4 \bigg]+\pr\bigg[  \sup_{t\in[0,T]}|\Pi_1H_{n}(t)|_{\infty}\geq n/4 \bigg]\\&
    	   +\pr\bigg[\sup_{t\in[0,T]}|\Pi_1S_\alpha(t)z|_{\infty}\geq n/4\bigg]+\pr\bigg[C\bigg[1+(\Lambda_n)^{2\lambda}   \bigg]\exp\bigg( c(\Lambda_n)^{2}T\bigg)\geq n/4\bigg]\\&
    	   \leq \frac{4}{n}\ex[\Lambda_n ]+\pr[  \Lambda_n\geq f_T(n)   ]\leq \bigg(\frac{4}{n}+\frac{1}{f_T(n)}\bigg)\ex[\Lambda_n]
    	   \end{aligned}
    	   \end{equation*}
    	   where $f_T$ is the inverse of the function $x\mapsto C[1+x^{2\lambda } ]\exp( cx^{2}T)$
    	   which diverges to $\infty$ as $x\to\infty.$ In view of \eqref{Ubnd}, \eqref{zbnd}, \eqref{convobnd} we deduce that 
    	   \begin{equation}\label{Znbnd}
    	   \begin{aligned} 
    	  \pr\bigg[  \sup_{t\in[0,T]}|\Pi_1Z_{n}^{h}(t)|_{\infty}\geq n   \bigg]&\leq C \bigg(\frac{4}{n}+\frac{1}{f_T(n)}\bigg)\big(  1+|h|_{H}+|z|_{\e}\big).
    	   \end{aligned}
    	   \end{equation}
    	   Hence, 
    	   $$\lim_{T\to\infty}\lim_{n\to\infty}\pr[ \tau_n\leq T   ]=\lim_{T\to\infty}\lim_{n\to\infty}\pr\bigg[  \sup_{t\in[0,T]}\big|   \Pi_1Z_{n}^{h}(t)\big|_{\infty}\geq n   \bigg]=0$$
    	    and the proof is complete.    
    	   \end{proof}

           Next, we prove estimates and continuity properties for the skeleton equation analogous to those in Lemmas \ref{lem:Zaprioricompactnesslem}, \ref{lem:controliccontinuity}.

\begin{lem}\label{lem:ZaprioricompactnesslemGlobal}
	Let $T,\alpha>0, z\in\mathcal{E},$ $u\in L^2([0,T]; H, Z\in C([0,T];\e).$ With $Z_z^u:=Z_z^{0,u}$ as in \eqref{eq:controlledequation} and under Assumptions \ref{Assumption:b}, \ref{Assumption:sigma} the following hold:
 \begin{enumerate}
     \item Under Assumptions \ref{Assumption:bGlobal}, \ref{Assumption:sigmaGlobal}, there is a constant $0<C=C(\sigma, |u|_{L^2([0,T];H)}, T, z)$ such that 
	\begin{equation}\label{ZaprioriGlobal}
	\sup_{t\leq T}\big|Z^{u}_{z}(t)\big|_{\mathcal{E}}\leq C.
 \end{equation}
 \item  Let $\mathcal{U}_N\subset L^2([0,T];H)$ be the centered ball of radius $N>0,$ endowed with the weak $L^2$ topology. The map 
	$$\e\times \mathcal{U}_N \ni (z, u)\longmapsto Z^u_z\in C([0,T];\e)$$
is continuous.
\item Let $D\subset\e$ bounded. There exists a smooth, non-decreasing function $\Lambda:[0,\infty)\rightarrow [0,\infty)$ with $\Lambda(0)=0, \lim_{x\to\infty}\Lambda(x)=\infty$ such that $$| Z_z^{u}-Z_z^0|_{C([0,T];\e)}\leq \Lambda\big(|u|_{L^2([0,T];H)}\big). $$
 \end{enumerate}
\end{lem}
\begin{proof}
\begin{enumerate}
    \item 
Let $G(t):= \Pi_1S_\alpha(t)z$ and
$$\rho=\Pi_1Z^u_{z}-G$$
solve the skeleton equation starting from $0$. From an energy estimate similar to the one used in Proposition \ref{prop:WellPosednessGlobal} (notice that here we require that the growth exponent of $b$ satisfies $\lambda\in(1,3]$) we have

\begin{equation*}
\begin{aligned}
\frac{d}{dt}\big|\partial_t\rho(t)\big|^2_{H}&+\frac{d}{dt}\big|\rho(t)\big|^2_{H^1}+2\big|\partial_t\rho(t)\big|^2_{H}=2\langle b(\Pi_1Z^u_{z}(t)),\; \partial_t\rho(t)\rangle_{H}+2\langle \sigma(\Pi_1Z^u_{z}(t))u(t),\; \partial_t\rho(t)\rangle_{H}\\&
\leq 2c_2\frac{d}{dt}\int_{0}^{\ell}\hat{\beta}(\rho(\xi, t))d\xi+\big| b(\rho(t)+G(t))-b(G(t)) \big|^2_{H}+|\sigma|^2_\infty|u(t)|^2_{H}+2\big|\partial_t\rho(t)\big|^2_{H},
\end{aligned}
\end{equation*}
where we used the chain rule, the inequality $2\langle x, y\rangle\leq |x|^2+|y|^2$ and Assumption \ref{Assumption:sigmaGlobal}. In view of Assumption \ref{Assumption:bGlobal}, we can then follow the same line of reasoning to obtain 
\begin{equation*}
\begin{aligned}
\sup_{t\leq T}\big|\rho(t)\big|^2_{H^1}&+\sup_{t\leq T}\big|\partial_t\rho(t)\big|^2_{H}\\&\leq \bigg[C\bigg(T+T\sup_{t\in[0,T]}|G(t)|^{2\lambda}_{\infty}+|\sigma|^2_\infty\int_{0}^{T}|u(t)|^2_{H}dt\bigg)\bigg]\exp\big(cT\sup_{t\in[0,T]}|G(t)|^{2}_{\infty}\big)
\end{aligned}
\end{equation*}
and from Theorem \ref{thm:Linftydecay} we obtain
\begin{equation*}
\begin{aligned}
\sup_{t\in[0,T]}|G(t)|^{2}_{\infty}\leq C\sup_{t\in[0,T]}e^{-\theta t}|z|_{\mathcal{E}}\leq C|z|_{\mathcal{E}}.
\end{aligned}
\end{equation*}

Thus, by Sobolev embedding we have 
\begin{equation}\label{eq:skeletonBound}
\begin{aligned}
&\sup_{t\leq T}\big|Z^{u}_{z}(t)\big|^2_{\mathcal{E}}\leq 2\sup_{t\leq T}\big|\Pi_1Z^{u}_{z}(t)\big|^2_{\infty}+2\sup_{t\leq T}\big|\partial_t\Pi_1Z^{u}_{z}(t)\big|^2_{C^{-1}}\\&\leq C\bigg(\sup_{t\leq T}\big|\rho(t)\big|^2_{H^1}+\sup_{t\leq T}\big|\partial_t\rho(t)\big|^2_{H}\bigg)+
\sup_{t\leq T}\big|\Pi_1S_\alpha(t)z\big|^2_{\infty}+\sup_{t\leq T}\big|\Pi_2S_\alpha(t)z\big|^2_{C^{-1}}\\&
\leq C\bigg[T(1+|z|^{2\lambda}_{\e})+|\sigma|^2_\infty N\bigg]\exp\big(cT|z|^2_\mathcal{E}\big)+C |z|^2_{\e}.
\end{aligned}
\end{equation}
The proof is complete.
    \item 
	Since the weak topology on bounded sets is metrizable, it suffices to consider a sequence $\{(z_n, u_n)\}_{n\in\N}\subset\mathcal{U}_N\times\e$ that converges weakly, as $n\to\infty,$ to a pair $(z, u)\in\mathcal{U}_N\times\e.$ For $t\in[0,T]$ we have 
\begin{equation*}
\begin{aligned}
|Z^{u_n}_{z_n}(t)-Z^{u}_{z}(t)\big|_{\e}&\leq \big| S_\alpha(t)(z_n-z)\big|_{\e}+\int_{0}^{t}\bigg| S_\alpha(t-s)\big[B(Z^{u_n}_{z_n}(s))- B(Z^{u}_{z}(s))\big]\bigg|_{\e}ds\\&
+\bigg|\int_{0}^{t} S_\alpha(t-s)\big[\Sigma(Z^{u_n}_{z_n}(s))u_n(s)- \Sigma(Z^{u}_{z}(s))u(s)\big]ds\bigg|_{\e}
\end{aligned}
\end{equation*}
From Assumption \ref{Assumption:bGlobal}, Assumption \ref{Assumption:sigmaGlobal} and Theorem \ref{thm:Linftydecay} we obtain 
\begin{equation}\label{Zcontbnd1Global}
\begin{aligned}
|Z^{u_n}_{z_n}(t)-Z^{u}_{z}(t)\big|_{\e}&
\leq C| z_n-z|_{\e}+C|b_1|_{Lip}\int_{0}^{t}\big|\Pi_1Z^{u_n}_{z_n}(s)-\Pi_1Z^{u}_{z}(s)|_{H}ds\\&+C\int_{0}^{t} \big|b_2(\Pi_1Z^{u_n}_{z_n}(s))- b_2(\Pi_1Z^{u}_{z}(s))\big|_{H}ds\\&
+|\sigma|_{\infty}\int_{0}^{t}\big|\Pi_1Z^{u_n}_{z_n}(s)-\Pi_1Z^{u}_{z}(s)|_{C_0(0,\ell)}|u_n(s)|_{H}ds\\&
+C\sup_{t\in[0,T]}\bigg|\int_{0}^{t}S_\alpha(t-s) \Sigma(Z^{u}_{z}(s))\big[u_n(s)-u(s)\big]ds\bigg|_{\e}.
\end{aligned}
\end{equation}

\noindent Now from \cite{cerrai2006smoluchowski}, Lemma 2.4 (notice that the latter requires Assumption \ref{Assumption:bGlobal};see also pp.678 of the same reference), along with the a-priori bound \eqref{ZaprioriGlobal} and the uniform bound on $u_n$ we have 
\begin{equation}\label{eq:bdifferencebound}
\begin{aligned}
\big|b_2(\Pi_1Z^{u_n}_{z_n}(s))&- b_2(\Pi_1Z^{u}_{z}(s))\big|^2_{H}\\&\leq c\int_{0}^{\ell}\bigg(1-\hat{\beta}\big(\Pi_1Z^{u_n}_{z_n}(
\xi,s)\big)+|\Pi_1Z^{u_n}_{z_n}(\xi,s)-\Pi_1Z^{u}_{z}(\xi,s) |^{2(\lambda-1)}\bigg)|\Pi_1Z^{u_n}_{z_n}(\xi,s)-\Pi_1Z^{u}_{z}(\xi,s) |^2d\xi\\&
\leq C\bigg(1+\sup_{t\in[0,T]}|\Pi_1Z^{u_n}_{z_n}(t)|^{\lambda+1}_{\infty}+\sup_{t\in[0,T]}|\Pi_1Z^{u_n}_{z_n}-\Pi_1Z^{u}_{z}(t)|^{2(\lambda-1)}_{\infty}\bigg)|\Pi_1Z^{u_n}_{z_n}(s)-\Pi_1Z^{u}_{z}(s)\big|^2_{H}\\&
\leq C_{\lambda}\bigg(1+\sup_{t\in[0,T]}|Z^{u_n}_{z_n}(t)|^{\lambda+1}_{\e}+\sup_{t\in[0,T]}|Z^{u}_{z}(t)|^{2(\lambda-1)}_{\e}       \bigg)|\Pi_1Z^{u_n}_{z_n}(s)-\Pi_1Z^{u}_{z}(s)\big|^2_{H}\\&
\leq \tilde{C}|\Pi_1Z^{u_n}_{z_n}(s)-\Pi_1Z^{u}_{z}(s)\big|^2_{H},
\end{aligned}
\end{equation}
\noindent with $\lambda$ as in Assumption \ref{Assumption:b},  $\hat{\beta}$ as in (\ref{Eq:betahat}) and  the constant $\tilde{C}$ in the last line takes the form 
$$C_{\sigma, N, \lambda, T }\bigg[ 1+(1+\sup_n|z_n|^{\lambda(\lambda+1)}+|z|^{2\lambda(\lambda-1)})\bigg(e^{cT\sup_n|z_n|^2_{\e}}+ e^{cT|z|^2_{\e}}     \bigg)+C\big(\sup_n|z_n|^{\lambda+1}_\e+|z|^{2\lambda(\lambda-1)}_\e\big)     \bigg].$$ 
In view of this estimate, \eqref{Zcontbnd1Global} yields
\begin{equation*}
\begin{aligned}
|Z^{u_n}_{z_n}(t)-Z^{u}_{z}(t)\big|^2_{\e}&
\leq C| z_n-z|^2_{\e}+CT^{1/2}\int_{0}^{T}\big|Z^{u_n}_{z_n}(s)-Z^{u}_{z}(s)|^2_{\e}ds\\&+\tilde{C}T^{1/2}\int_{0}^{T} |Z^{u_n}_{z_n}(s)-Z^{u}_{z}(s)\big|^2_{\e}ds\\&
+|\sigma|_{\infty}\sup_{n\in\N}|u_n|^2_{L^2([0,T];H)}\int_{0}^{T}\big|Z^{u_n}_{z_n}(s)-Z^{u}_{z}(s)|^2_{\e}ds\\&
+C\sup_{t\in[0,T]}\bigg|\int_{0}^{t}S_\alpha(t-s) \Sigma(Z^{u}_{z}(s))\big[u_n(s)-u(s)\big]ds\bigg|^2_{\e}.
\end{aligned}
\end{equation*}
Thus, Gr\"onwall's inequality furnishes
\begin{equation*}
\begin{aligned}
\sup_{t\in[0,T]}|Z^{u_n}_{z_n}(t)-Z^{u}_{z}(t)\big|^2_{\e}&
\leq Ce^{C_{\sigma, b_1, \lambda, N, T}}\bigg(| z_n-z|^2_{\e}+\sup_{t\in[0,T]}\bigg|\int_{0}^{t}S_\alpha(t-s) \Sigma(Z^{u}_{z}(s))\big[u_n(s)-u(s)\big]ds\bigg|^2_{\e}\bigg).
\end{aligned}
\end{equation*}
From Lemma \ref{lem:Zaprioricompactnesslem} with $Z=Z^{u}_{z},$ the latter vanishes as $n\to\infty$ and the conclusion follows.
\item 
For $t\leq T$ we have 
\begin{equation*}
   Z_z^u(t)-Z_z^0(t)=\int_{0}^tS_\alpha(t-s)\big[b(  Z_z^u(s))- b(  Z_z^0(s))\big]ds+\int_{0}^tS_\alpha(t-s)\sigma(  Z_z^u(s))u(s)ds.
\end{equation*}
Since $\sigma$ is bounded, the second term on the right hand side satisfies 

\begin{equation}\label{eq:sigmacontrol}
    \begin{aligned}
        \bigg|\int_{0}^tS_\alpha(t-s)\sigma(  Z_z^u(t))u(t)dt\bigg|_{\e}&\leq C \bigg|\int_{0}^tS_\alpha(t-s)\sigma( Z_z^u(t))u(t)dt\bigg|_{\h_1}\\&\leq C|\sigma|_{\infty}\int_{0}^{t}e^{-\theta(t-s)} |u(s)|^2_{\h}ds\leq C|\sigma|_{\infty} |u|^2_{L^2([0,T];H)}.
    \end{aligned}
\end{equation}
As for the first term, a bound similar to \eqref{eq:bdifferencebound} yields 
\begin{equation*}
\begin{aligned}
    \bigg|\int_{0}^{t}S_\alpha(t-s)\big[b(  Z_z^u(s))&- b(  Z_z^0(s))\big]ds\bigg|^2_{\e}\\&\leq \bigg(1+\sup_{t\in[0,T]}|Z^{u}_{z}(t)|^{\lambda+1}_{\e}+\sup_{t\in[0,T]}|Z^{0}_{z}(s)|^{2(\lambda-1)}_{\e}       \bigg)\int_{0}^{T}\sup_{r\leq s}|Z^{u}_{z}(r)-Z^{0}_{z}(r)|^2_{\e} ds,
\end{aligned}  
\end{equation*}
with $\lambda$ as in Assumption \ref{Assumption:bGlobal}. 
From the boundedness of $D$ and the estimates in Lemma \ref{lem:Zaprioricompactnesslem} (see in particular \eqref{eq:skeletonBound}) we obtain the bound
\begin{equation*}
\begin{aligned}
&\sup_{z\in D}\bigg(  \sup_{t\in[0,T] }\big|Z^{u}_{z}(t)\big|^2_{\mathcal{E}}+         \sup_{t\in[0,T] }\big|Z^{0}_{z}(t)\big|^2_{\mathcal{E}}\bigg)
\leq C_{D, T}\bigg[1+|u|^2_{L^2([0,T];H)} \bigg].
\end{aligned}
\end{equation*}
Combining the last two displays it follows that
\begin{equation*}
\begin{aligned}
    \bigg|\int_{0}^{t}S_\alpha(t-s)\big[b(  Z_z^u(s))&- b(  Z_z^0(s))\big]ds\bigg|^2_{\e}\\&\leq C_{T, D, \lambda }\bigg(1+|u|_{L^2([0,T];H)}^{2\lambda+2}+|u|_{L^2([0,T];H)}^{4(\lambda-1)}     \bigg)\int_{0}^{T}\sup_{r\leq s}|Z^{u}_{z}(r)-Z^{0}_{z}(r)|^2_{\e} ds.
\end{aligned}  
\end{equation*}
The latter along with \eqref{eq:sigmacontrol} then furnish 
$$| Z_z^u-Z_z^0|^2_{C([0,T];\e)}\leq C_1\bigg(1+|u|_{L^2([0,T];H)}^{2\lambda+2}+|u|_{L^2([0,T];H)}^{4(\lambda-1)}     \bigg)\int_{0}^{T}\sup_{r\leq s}|Z^{u}_{z}(r)-Z^{0}_{z}(r)|^2_{\e}ds+C_2 |u|^4_{L^2([0,T];H)}.   $$
From Gr\"onwall's inequality we may conclude
\begin{equation*}
 | Z_z^u-Z_z^0|^2_{C([0,T];\e)}\leq C|u|^4_{L^2([0,T];H)}\exp\bigg\{C'\bigg(|u|_{L^2([0,T];H)}^{2\lambda+2}+|u|_{L^2([0,T];H)}^{4(\lambda-1)}\bigg)     \bigg\}
\end{equation*}
for some constants $C, C'>0.$  The proof is complete upon observing that the function $$\Lambda(x)=Cx^4\exp(C'(x^{2\lambda+2}+x^{4(\lambda-1)})/2), x\geq 0$$ satisfies the desired properties.
\end{enumerate}
\end{proof}

We conclude this section by showing that global solutions of \eqref{eq:model} satisfy a (global) LDP, see Definition \ref{Def:GlobalULDP}, that is uniform over bounded sets of initial data. This is the last ingredient needed to transfer the exit problem results from the previous sections to global solutions.

             \begin{prop}\label{prop:ULDPGlobal} Let $D \subset \e$ be a bounded subset of initial data. Under Assumptions \ref{Assumption:bGlobal}, \ref{Assumption:sigmaGlobal}, the family $\{Z^\epsilon_z\}_{z \in D, \epsilon>0}  $ satisfies a (global) ULDP (per Definition \ref{Def:GlobalULDP}) with respect to $D$ and with rate function given by
             \begin{equation*}\label{eq:RateFunctionWaveGlobal}
             I_{z,T}(\phi) := \inf_{u\in L^2([0,T];H)} \left\{ \frac{1}{2}\int_0^T \int_0^\ell |u(s,y)|^2 dyds: \phi=Z^{u}_z \right\}, \phi\in C([0,T];\e),
         \end{equation*}
            where the convention $\inf\varnothing=\infty$ is being assumed.
             \end{prop}

              \begin{proof}    
     The (global) ULDP is equivalent to an Equicontinuous Uniform Laplace Principle (EULP) \cite[Theorem 2.9]{salins2019equivalences}. In view of  \cite[Theorem 2.12]{salins2019equivalences}, the latter holds over bounded subsets of initial data, provided that for any $\delta>0,$ $D\subset\e$ bounded and $N>0,$
    	   \begin{equation}
    	       \label{eq:EULPsufficientGlobal}
            \lim_{\epsilon\to 0}\sup_{z\in D}\sup_{u\in\mathcal{P}_2^N}\pr\bigg[\sup_{t\in[0,T]}\big|Z_z^{\epsilon,u}(t)-Z_z^{u}(t)\big|_{\e}\geq \delta \bigg]=0.
    	   \end{equation}
        Here, $\mathcal{P}_2^N$ is the collection of adapted, $H-$valued stochastic controls $u(t)$ such that $\pr( |u(t)|^2_H\leq N)=1.  $
       
    	   In order to prove the latter, we use the localizing sequences of stopping times $$\tau^{\epsilon,u}_{z,n}:=\inf\{t>0 : |\Pi_1 Z^{\epsilon,u}_{z,n}(t)|_{\infty}\geq n\},\;\;\tau^\epsilon_n:=\tau^{\epsilon,u}_{z,n}\wedge \tau^{0,u}_{z,n},\;\;n\in\N,$$
        where, for each $\epsilon\geq 0$ $Z^{\epsilon,u}_{z,n}$ is the corresponding solution of the localized problem \eqref{eq:controlledquationLipGlobal} (with $W$ replaced by $\epsilon W$). Next, we write
    	   \begin{equation*}\label{eq:p1p2DecompULDP}
    	   \begin{aligned}
    	      \pr\bigg[  \sup_{t\in[0,T]}
           \big|Z_z^{\epsilon,u}(t)-Z_z^{u}(t)\big|_{\e}\geq \delta   \bigg]&=\pr\bigg[  \sup_{t\in[0,T]}\big|Z_z^{\epsilon,u}(t)-Z_z^{u}(t)\big|_{\e}\geq \delta\;,\;\; \tau^{\epsilon}_{n}>T\bigg]\\&+\pr\bigg[  \sup_{t\in[0,T]}\big|Z_{z}^{\epsilon,u}(t)-Z_z^{u}(t)\big|_{\e}\geq \delta\;,\;\; \tau^{\epsilon}_{n}\leq T\bigg]\\&=:p^{\epsilon,u,z}_1+p^{\epsilon,u,z}_2.
    	   \end{aligned}
    	   \end{equation*}
    	  The arguments leading to \eqref{Znbnd} can be used (mutatis mutandis) to obtain
    	    \begin{equation}\label{p2bnd}
    	   \begin{aligned}
    	   \sup_{\epsilon\in(0,1)}\sup_{z\in D}\sup_{u\in\mathcal{P}_2^N}p^{\epsilon,u,z}_2\leq \sup_{\epsilon\in(0,1)}\sup_{z\in D}\sup_{u\in\mathcal{P}_2^N}\bigg(\pr[\tau^{\epsilon,u}_{z,n}\leq T]+\pr[\tau^{0,u}_{z,n}\leq T]\bigg)\longrightarrow 0\;,\text{as}\;\;n\to\infty.
    	   \end{aligned}
    	   \end{equation}
        
        Turning to $p_1$, Chebyshev's inequality  and the fact that on $\{ \tau^{\epsilon}_{n}> T \},$ $Z^{\epsilon,u}_{z,n}=Z^{\epsilon,u}_z, Z^{u}_{z,n}=Z^{u}_z$ furnish
    	   \begin{equation}\label{p1bnd}
    	   \begin{aligned}
    	   	\pr\bigg[  \sup_{t\in[0,T]}\big|Z_z^{\epsilon,u}(t)-Z_z^{u}(t)\big|_{\e}\geq \delta\;,\;\; \tau^{\epsilon}_{n}> T\bigg]&\leq  \pr\bigg[  \sup_{t\in[0,T]}\big|Z_{z,n}^{\epsilon,u}(t)-Z_{z,n}^{0,u}(t)\big|_{\e}\geq \delta\bigg]\\&\leq \frac{1}{\delta}\ex\sup_{t\in[0,T]}\big|Z_{z,n}^{\epsilon,u}(t)-Z_{z,n}^{0,u}(t)\big|_{\e}.
    	   \end{aligned}
       \end{equation}
       \noindent We can now estimate the latter using the local Lipschitz continuity and Gr\"onwall's inequality. Indeed,  
    	   
    	   \begin{equation*}
    	   \begin{aligned}
\big|Z_{z,n}^{\epsilon,u}(t)-Z_{z,n}^{0,u}(t)\big|_{\e}&\leq \bigg|\int_{0}^{t}S(t-s)\big[B_n(Z_{z,n}^{\epsilon,u}(s))-B_n(Z_{z,n}^{0,u}(s))\big]ds\bigg|_{\e}\\&+ \bigg|\int_{0}^{t}S(t-s)\big[\Sigma(Z_{z,n}^{\epsilon,u}(s))-\Sigma(Z_{z,n}^{0,u}(s))\big]u(s)ds\bigg|_{\e}\\&
+\epsilon\bigg|\int_{0}^{t}S(t-s)\Sigma(Z_{z,n}^{\epsilon,u}(s))dW(s)\bigg|_{\e}\\&\leq C\int_{0}^{t}\big|b_n(\Pi_1Z_{z,n}^{\epsilon,u}(s))-b_n(\Pi_1Z_{z,n}^{0,u}(s))\big|_{H}ds\\&
+C\int_{0}^{t}\big|\big[\sigma(\Pi_1Z_{z,n}^{\epsilon,u}(s))-\sigma(\Pi_1Z_{z,n}^{0,u}(s))\big]u(s)\big|_{H}ds+2C\epsilon\sup_{t\in[0,T]}\big|\Gamma_n^\epsilon(t)\big|_\e
    	   \end{aligned}
    	   \end{equation*}
where $\Gamma^\epsilon_n(t)$ is the stochastic convolution.
 Thus, from the Lipschitz continuity of $b_n$ and the (global) Lipschitz continuity of $\sigma$ (Assumption \ref{Assumption:sigmaGlobal}) there exists for each $n$ a constant $L_n>0$ such that 
\begin{equation*}
\begin{aligned}
\big|Z_{z,n}^{\epsilon,u}(t)-Z_{z,n}^{0,u}(t)\big|_{\e}&\leq L_n\int_{0}^{t}\big|\Pi_1Z_{z,n}^{\epsilon,u}(s))-\Pi_1Z_{z,n}^{0,u}(s))\big|_{\infty}ds\\&
+C_{\sigma}\int_{0}^{t}\big|\Pi_1Z_{z,n}^{\epsilon,u}(s)-\Pi_1Z_{z,n}^{0,u}(s)\big|_{\infty}|u(s)|_{H}ds+C\epsilon\sup_{t\in[0,T]}|\Gamma_n^\epsilon(t)|_\e.
\end{aligned}
\end{equation*}

\noindent From the $L^2-$bound on $u$ and Gr\"onwall's inequality we get

\begin{equation*}
\begin{aligned}
\ex\sup_{t\in[0,T]}\big|Z_{z,n}^{\epsilon,u}(t)-Z_{z,n}^{0,u}(t)\big|^2_{\e}\leq C \epsilon^2e^{C_{\sigma, n}T}\sup_{\epsilon\in(0,1), n\in\N}\ex\sup_{t\in[0,T]}\big|\Gamma_n^\epsilon(t)\big|^2_\e\leq C\epsilon^2e^{C_{\sigma, n}T},
\end{aligned}
\end{equation*}
 \noindent where the last estimate follows by applying Lemma \ref{lem:StochasticConvolutionApriori} with $\Psi_2=0, p=2$ and using the boundedness of $\sigma.$  Combining the latter with \eqref{p1bnd}, \eqref{p2bnd} and taking limits, first as $\epsilon\to0 $ and then as $n\to\infty,$ we see that \eqref{eq:EULPsufficientGlobal} holds true for any $\delta, T>0$ and any bounded set $D\subset\e.$ 
\end{proof}

 \appendix
 \section{}\label{section:Appendix}

 \subsection{Proof of Lemma \ref{lem:StochasticConvolutionApriori}}
 \label{App:StochConv}
 
 \begin{proof}
 Recall that $\hat{\cdot}$ denotes anti-differentiation (see Section \ref{Sec:Notation}). Since we do not require uniform estimates, it suffices to prove the case $\epsilon=1.$ Furthermore, from linearity we can take $\Psi_2=0$ (the general case will then follow by substituting $\Psi_1$ by $\Psi_2-\Psi_1).$ To this end, let $\Psi\in L^p(\Omega; C([0,T]; C_0(0,\ell))).$ Then $\Gamma:=\Gamma_{\Psi}$ is a mild solution of
 $$ \partial_t^2\Gamma+\alpha\partial_t\Gamma=\partial^2_{x}\Gamma+\Psi\dot{W}\;, \Gamma(0)=\partial_t\Gamma(0)=0.     $$
 Hence, letting $G=e^{\alpha t/2}\Gamma$ and applying the chain rule it holds that 
 $$  \partial_t^2G=\partial^2_{x}G+\frac{\alpha^2}{4}G+e^{\alpha t/2}\Psi\dot{W}, G(0)=\partial_tG(0)=0,$$
in the sense of distributions, or in mild form
$$G(t)=\frac{\alpha^2}{4}\int_0^ t S_0(t-s)G(s)ds+\int_0^t S_0(t-s)(0, e^{\alpha s/2}\Psi(s))dW(s).$$
\noindent The above, along with Gr\"onwall's inequality, furnishes
\begin{equation*}
\begin{aligned}
    \sup_{t\in[0,T]}|G(t)|_{\e}\leq \exp(C T)\sup_{t\in[0,T]}\bigg|\int_0^t S_0(t-s)(0, e^{\alpha s/2}\Psi(s))dW(s)   \bigg|_{\e},
\end{aligned} 
\end{equation*}
for a constant $C>0$ that depends on the wave semigroup, $T$ and $\alpha.$
Moreover, in view of \eqref{equation:D'Alembert1}, \eqref{equation:D'Alembert2}
\begin{equation*}
\begin{aligned}
    \Pi_1\int_0^t S_0(t-s)(0, e^{\alpha s/2}\Psi(s))dW(s)
  =\frac{1}{2}\int_0^t\int_{x-(t-s)}^{x+(t-s)}e^{\alpha s/2}\Psi(y,s)dW(y,s)
\end{aligned} 
\end{equation*}
and 
\begin{equation*}
\begin{aligned}
     -\widehat{\Pi}_2\int_0^t S_0(t-s)(0, e^{\alpha s/2}\Psi(s))dW(s)= \frac{1}{2}\int_0^t\bigg(\int_{0}^{x+(t-s)}+\int_0^{x-(t-s)}\bigg)e^{\alpha s/2}\Psi(y,s)dW(y,s),
\end{aligned} 
\end{equation*}
where $W$ is interpreted as a Brownian sheet in $[0,
\infty)\times\R.$
Thus, from the BDG inequality,
\begin{equation*}
\begin{aligned}
    \ex \bigg|\Pi_1\int_0^t S_0(t-s)(0, e^{\alpha s/2}\Psi(s))dW(s)   \bigg|_{\infty}^p&\leq C_p\bigg(\int_0^t\int_{x-(t-s)}^{x+(t-s)}e^{\alpha s}\ex\Psi^2(y,s)dyds\bigg)^{p/2}\\&
    \leq C(t-s)e^{p\alpha T/2}\ex\sup_{t\in[0,T]}|\Psi(s)|^p_\infty
\end{aligned}
\end{equation*}
and similarly 

\begin{equation*}
\begin{aligned}
   \ex \bigg|\Pi_2 \int_0^t S_0(t-s)(0, e^{\alpha s/2}\Psi(s))dW(s)   \bigg|_{C^{-1}}^p&= \ex \bigg|\widehat{\Pi}_2\int_0^t S_0(t-s)(0, e^{\alpha s/2}\Psi(s))dW(s)   \bigg|_{\infty}^p\\&\leq C\ex\sup_{t\in[0,T]}|\Psi(s)|^p_\infty,
\end{aligned}
\end{equation*}
where $C$ depends on $T, \ell, \alpha.$ The last two estimates yield

\begin{equation*}
\begin{aligned}
    \ex\sup_{t\in[0,T]}|\Gamma_{\Psi}(t)|^p_{\e}= e^{-\frac{p\alpha t}{2}}\ex\sup_{t\in[0,T]}|G(t)|^p_{\e}\leq C \ex\sup_{t\in[0,T]}|\Psi(s)|^p_\infty.
\end{aligned} 
\end{equation*}
The proof is complete.
 \end{proof}

\subsection{On the subdifferential of the norm on a Banach space}\label{App:Subdifferential} The following fact is stated in \cite[Equation (D.2)]{da2014stochastic} without proof. For the sake of completeness and readers' convenience we present a proof here.

\begin{lem}{\cite[Equation (D.2)]{da2014stochastic}} Let $E$ be a Banach space. For any $x,y\in E$ we have
$$D^+|x|_E(y):=\limsup_{h\to 0}\frac{|x+hy|_E-|x|_E}{h}=\max\big\{\langle y, x^\star\rangle; x^\star\in\partial|x|_{E}  \big\}.          $$   
\end{lem}	
	\begin{proof}
		Let $h \in \mathbb{R}$. For any $x^\star \in \partial|x|_E$,
		\begin{equation*}
			|x+hy|_E - |x|_E \geq \left<x+hy,x^\star\right> - \left<x,x^\star \right>\geq h \left<y,x^\star\right>.
		\end{equation*}
		For $h>0$,
		\begin{equation*}
			\frac{|x + hy|_E - |x|_E}{h} \geq \left<y,x^\star\right>.
		\end{equation*}
		This is true for all $x^\star \in \partial |x|_E$, so 
		\begin{equation} \label{eq:geq-max}
			\liminf_{h \downarrow 0} \frac{|x + hy|_E - |x|_E}{h}\geq \max\{\left<y,x^\star\right>: x^\star \in \partial|x|_E\}.
		\end{equation}
		To prove the other direction, chose a sequence $h_n \downarrow 0$ such that
		\begin{equation*}
			\lim_{n \to \infty} \frac{|x + h_n y|_E - |x|_E}{h_n}  = \limsup_{h \downarrow 0} \frac{|x + hy|_E - |x|_E}{h}. 
		\end{equation*}
		Let $x_n^\star \in \partial|f + h_ng|_E$. Note that $\{x_n^\star; n\in\N\}\subset B_{E^\star}(0,1).$ By virtue of Alaoglu's theorem $B_{E^\star}(0,1)$ is  weak-$\star$ compact. Hence there exists a subsequence (relabelled $h_n$) such that
		$x_{n}^\star$ converges in weak-$\star$ topology to some $x^\star\in E^{\star}$ and, as a consequence, 
		\begin{align*}
			&\left<x,x^\star_{n} \right> \to \left<x,x^\star\right>\\
			&\left<y, x^\star_{n} \right> \to \left<y, x^\star \right>.
		\end{align*}
		Therefore, 
		\begin{equation*}
			\lim_{n \to \infty} \frac{|x+h_n y|_E - |x|_E}{h_n} \leq \lim_{n \to \infty}\frac{\left<x, x_{n}^\star\right> + h_n \left<x, x_{n}^\star\right> - \left<x, x_{n}^\star\right>}{h_n} = \left<y,x^\star\right>.
		\end{equation*}
		Now we claim that $x^\star$ in $\partial |x|_E.$ Indeed, we know that 
		\begin{equation*}
			\lim_{n \to \infty} |x + h_n g|_E = |x|_E
		\end{equation*}
		and that
		\begin{equation*}
			\lim_{n \to \infty} |x + h_n y|_E =\lim_{n \to \infty}\left( \left< x,x_{n}^\star \right> + h_n \left< y, x_{n}^\star\right>\right)  = \left<x,x^\star\right>.
		\end{equation*}
		Hence $x^\star \in \partial |x|_E$.
		We have proved that 
		\begin{equation} \label{eq:leq-max}
			\limsup_{h \downarrow 0} \frac{|x + hy|_E - |x|_E}{h} \leq \left<y,x^\star \right>_E.
		\end{equation}
		From \eqref{eq:geq-max} and \eqref{eq:leq-max} the proof is complete.
	\end{proof}

  \bibliographystyle{plain}
	   \nocite{}
	   \bibliography{References}

 \end{document}